\documentclass{amsart}

\usepackage{amsthm} 
\usepackage{graphicx}  
\usepackage[all]{xy}
\usepackage{color}

\usepackage{tikz}
\usetikzlibrary{matrix}

\usepackage{epsfig}
\usepackage{latexsym}
\usepackage{url}
\usepackage{psfrag}
\usepackage{mathrsfs}
\usepackage{tikz-cd}

\usepackage{exscale}

\title{Deforming convex projective manifolds}

\author{Daryl Cooper}
\address{Department of Mathematics, University of California, Santa Barbara, CA 93106, USA}
\email{cooper@math.ucsb.edu}
\urladdr{http://web.math.ucsb.edu/~cooper/}

\author{Darren Long}
\address{Department of Mathematics, University of California, Santa Barbara, CA 93106, USA}
\email{long@math.ucsb.edu}
\urladdr{http://web.math.ucsb.edu/~long/}
\author{Stephan Tillmann}
\address{School of Mathematics and Statistics, The University of Sydney, NSW 2006, Australia}
\email{tillmann@maths.usyd.edu.au}
\urladdr{http://www.maths.usyd.edu.au/u/tillmann/}

\newtheorem{theorem}{Theorem}[section]
\newtheorem{lemma}[theorem]{Lemma}
\newtheorem{corollary}[theorem]{Corollary} 
\newtheorem{definition}[theorem]{Definition} 
\newtheorem{proposition}[theorem]{Proposition}

\def\text{\mbox}
\def\RR{{\mathbb R}}
\def\PP{{\mathbb P}}
\def\SS{{\mathbb S}}
\def\CC{{\mathbb C}}

\def\FF{{\mathbb F}}
\def\ZZ{{\mathbb Z}}

\def\RA{{\mathbb L}}
\def\PPp{{\mathbb P}_+}
\def\PGLp{\operatorname{P_+GL}}
\def\RPn{{\mathbb{RP}}^n}
\def\RP{{\mathbb{RP}}}

\def\Fr{{\operatorname{Fr}}}
\def\Diff{{\operatorname{Diff}}}
\def\Im{\operatorname{Im}}
\def\interior{\operatorname{int}}
\def\SL{\operatorname{SL}}
\def\CH{\operatorname{CH}}

\def\PGL{\operatorname{PGL}}
\def\GL{\operatorname{GL}}
\def\UT{\operatorname{UT}}

\def\VFG{\operatorname{VFG}}

\def\Hom{\operatorname{Hom}}
\def\Rep{\operatorname{Rep}}

\def\PC{\operatorname{Rep_{ce}}}

\def\cl{{\rm cl}}
\def\dev{\operatorname{dev}}

\def\Hol{\mathcal Hol}
\def\Sym{\operatorname{Sym}}

\def\hol{\operatorname{hol}}

\def\flow{\Phi}

\def\charfn{\rchi}
\def\diag{\operatorname{diag}}
\def\flowfn{\operatorname{c}}
\def\Vcal{\mathcal V}
\def\Ucal{\mathcal U}

\def\Bcal{\mathcal B}

\def\R{\mathbb R}

\def\Kleinian{\mathcal K}

\def\devGX{\mathcal Dev}
\def\GC{\mathcal{GC}}
\def\devc{{\mathcal Dev}_c}
\def\vinbergsurface{S}
\def\sphere{\mathbb S}
\def\enddata{\mathcal Rel\mathcal Hol}

\def\enddatamap{\operatorname{\mathcal E}}
\def\domains{{\green\mathcal P}}
\def\alldomains{\domains'}
\def\image{\operatorname{Im}}
\def\proj{\mathbb P}

\def\core{\operatorname{core}}

\def\Ztwo{\mathbb Z_2}

\def\clsdsub{{\green\mathfrak{C}}}
\def\Cone{{\green\mathcal{C}}}

\DeclareRobustCommand{\rchi}{{\mathpalette\irchi\relax}}
\newcommand{\irchi}[2]{\raisebox{\depth}{$#1\chi$}} 

\def\edit{}

\def\green{}

\newcommand{\bv}{\left[\begin{array}{c}}
\newcommand{\ev}{\end{array}\right]}
\newcommand{\bbmat}{\begin{bmatrix}} 
\newcommand{\ebmat}{\end{bmatrix}}
\newcommand{\bmat}{\begin{matrix}} 
\newcommand{\emat}{\end{matrix}}
\newcommand{\bpmat}{\begin{pmatrix}} 
\newcommand{\epmat}{\end{pmatrix}}

\begin{document}

\begin{abstract}    
We study a properly convex real projective manifold with
(possibly empty) compact, strictly{\edit-}convex boundary,
and which consists of a compact part plus finitely many convex ends.
We extend a theorem of Koszul which asserts that for a compact manifold
without boundary the holonomies
of properly convex structures form an open subset of the representation variety.
We also give a relative version for non-compact $(G,X)$-manifolds of the openness of
their holonomies.
\end{abstract}

\maketitle


Given a subset $\Omega\subset{\mathbb R}P^n$ 
the {\em frontier} is $\Fr(\Omega)=\cl({\Omega})\setminus \interior(\Omega)$ and the {\em boundary} is $\partial\Omega=\Omega\cap\Fr(\Omega)$.
 A  {\em properly convex projective manifold} is $M=\Omega/\Gamma,$ where
  $\Omega\subset{\mathbb R}P^n$
is a convex set with non-empty interior, and $\cl({\Omega})$ {\edit does not contain any $\RP^1$,} 
and $\Gamma\subset\PGL(n+1,{\mathbb R})$ acts freely and properly discontinuously on $\Omega$.
If, in addition,  $\Fr(\Omega)$ contains no line segment then $M$ and $\Omega$ are {\em strictly{\edit-}convex}.
The {\em boundary of $M$ is  strictly-convex} if $\partial\Omega$ contains no line segment. 

If $M$ is a {\em compact} $(G,X)$-manifold then a sufficiently small deformation of the holonomy
gives another $(G,X)$-structure on $M$.
In  \cite{Kos1, Kos2} Koszul proved  a similar result holds for closed, 
properly convex, projective manifolds.
 In particular,  nearby holonomies continue to be discrete and faithful representations
of the fundamental group.  

Koszul's theorem cannot be 
generalized to the case of non-compact manifolds without some 
qualification---for example,  a sequence of hyperbolic surfaces whose completions have cone singularities
can converge to a hyperbolic surface with a cusp. The holonomy of a cone surface 
 in general is neither discrete nor faithful.
Therefore we must impose conditions on the holonomy of each end.

{\edit If $M$ is a geometrically finite hyperbolic manifold $M$ with a convex core that has compact boundary, then
every end of $M$
 is topologically a product, and is foliated by strictly{\edit-}convex hypersurfaces.
These hypersurfaces are either convex {\edit towards} $M$ so that cutting along one gives a submanifold
of $M$ with convex boundary, and the holonomy of the end
contains only hyperbolics; or else
convex {\em away} from $M$, in which case the end is a cusp and the holonomy of the end contains only parabolics. }

 {\edit This paper  studies properly convex} manifolds whose ends are either convex towards or away from $M$. {\edit An end that
 is convex towards $M$ may be compactified by adding a convex boundary.}
 {\em Generalized cusps}
are those that are convex away from $M$ with
virtually nilpotent fundamental group. {\edit The holonomy of a generalized cusp
may contain both hyperbolic and parabolic elements. }

   \begin{definition} A {\em generalized cusp}  is a 
properly convex manifold $C$ homeomorphic to $\partial C\times[0,\infty)$ 
with compact, strictly-convex boundary and with $\pi_1C$ virtually
nilpotent. \end{definition}

For instance, all ends of 
 a finite volume hyperbolic manifold  are generalized cusps.
For an $n$--manifold $M$, possibly with boundary, 
 define $\Rep(\pi_1M)=\Hom(\pi_1 M,\GL(n+1,{\mathbb R}))$ and $\PC(M)$ to be the 
 subset of $\Rep(\pi_1M)$ consisting of holonomies of properly convex structures on $M$ 
 with $\partial M$ strictly{\edit-}convex{\edit,} and such that each end is a generalized cusp. 
{\edit A group $\Gamma\subset \GL(n+1,{\mathbb R})$  is a 
 {\em  virtual flag group}  if it contains
a subgroup  of finite index that  is conjugate
into the upper-triangular group. The set of virtual flag groups is written $\VFG$. }

\begin{theorem}\label{deformmfd} Suppose $N$ is a compact connected $n$--manifold 
and $\Vcal = \cup_i V_i \subseteq\partial N$ is the union of some of the boundary components of $N.$
Let $M=N\setminus\Vcal$. Assume $\pi_1V_i$
is virtually nilpotent for each $i$. {\edit Let $B_i\cong V_i\times[0,1)$ be the
 end of $M$ corresponding to $V_i$.
Then $\PC(M)$ is an open subset of \[\VFG(M):=\{\rho\in \Rep(\pi_1M) : \forall i\  \rho(\pi_1B_i)\in\VFG\}.\]}
\end{theorem}

A similar statement holds for orbifolds since a properly convex orbifold has a finite cover which is
a manifold, and the property of being properly convex is unchanged by coverings. This theorem
is a consequence of our main theorem (\ref{mainthm}) that a certain map is open. By
 (\ref{VnilVFG})  $\rho(\pi_1B_i)\in\VFG$ iff there is a finite index subgroup $\Gamma<\rho(\pi_1B_i)$
such that every eigenvalue of  every element of $\Gamma$ is real.

There is a Margulis
lemma for {\em properly} convex manifolds that says the local fundamental group is virtually nilpotent
(0.1) in Cooper, Long and Tillmann \cite{CLT1},  also Crampon and Marquis \cite{CM1}. 
There is a {\em thick-thin} decomposition for {\em strictly} convex manifolds
 (0.2) in \cite{CLT1},  also Crampon and Marquis \cite{CM2}; but
 not for {\em properly} convex manifolds. Each component
of the thin part of a {\em strictly} convex manifold 
is a Margulis tube or a {\em cusp} and has virtually nilpotent fundamental group consisting of parabolics.
This motivates the
definition of {\em generalized cusp} above.
There is a discussion of cusps in properly convex manifolds in \S 5 of  \cite{CLT1}. 

 Here is some intuition for the proof of (\ref{deformmfd}). Many of the ideas are already present for surfaces. 
Suppose $M$ is a projective compact surface without boundary. 
There is a developing map $\dev:\widetilde{M}\rightarrow \RP^2$. If $M$ is properly convex then this map
is injective and the image is a domain $\Omega\subset\RR^2$ bounded by a convex closed curve $C\subset\RR^2$. There
is a compact polygonal fundamental domain $D\subset \Omega$ and the images of $D$ under the holonomy
$\rho:\pi_1M\rightarrow \Gamma\subset \PGL(3,\RR)$ tesselate $\Omega$. 
Small images of $D$ accumulate on $C$. Suppose $\rho'$ is a nearby homomorphism.  Our aim is
to show there is a convex closed curve $C'$ close to $C$ that bounds a domain $\Omega'$ that is 
preserved by $\Gamma'=\rho'(\pi_1M)$. 

The convexity of $C$
 is a phenomenon that {\em takes place at infinity} with regard to $\Omega$.
A priori, there is no reason to expect this phenomenon to be stable with respect to small deformations.
A finite generating 
set for $\pi_1M$ determines a word metric on $\pi_1M$. If $g\in\pi_1M$ is not
too far from the identity in the word metric then $\rho'(g)$ and $\rho(g)$ are close and send $D$ to almost the same
set. However, for large elements $g$, one {\em loses control,} and there is no obvious reason why 
 images of $D$ should accumulate on some convex curve $C'$. Thus {\em convexity} of $C=\partial\Omega$ is
 a {\em limiting feature} of large group elements in $\pi_1M$, and this might be destroyed by arbitrarily
 small deformations $\rho'$. In fact this is what happens with the example of surfaces with cone
 singularities discussed above. Convexity is {\edit stable} for {\em closed} surfaces for reasons that we now outline.
 
 {\edit Let $U\subset\RR^3$ be the half space $x_3\ge 0$ and $P=\partial U$ the plane $x_3=0.$
  Suppose $D\subset U$ is a disc, bounded by the simple closed curve
   $C=D\cap P$, and $S=\interior(D)$.  If $S$ is convex down then {\edit
 $C$ is convex, and} bounds a convex domain $\Omega\subset P$,
 and  $S$ is a graph over $\Omega$.
 The condition that $S$ is convex down is a condition over (finite) points in $\Omega$ rather than a condition at infinity.
 One might imagine that $S$ is the surface of a mountain.  If this surface
  is convex then the boundary of the base of the mountain is also convex.
 
 The {\em tautological line bundle} $\xi M$ over $M$ is an affine manifold. 
 For a properly convex projective
 structure on $M$,  there is a section $\sigma(M)$ of this bundle that
 is a surface in $\xi M$ which is {\em strictly} convex in the sense
 that the Hessian is strictly positive. The image of {\edit the universal cover of} 
 $\sigma(M)$ under the developing map is a surface  $S\subset\RR^3$ that limits
 on the sphere at infinity. By 
 viewing $\RR^3\subset\RP^3$, and choosing
 a suitable new affine patch where the sphere at infinity becomes $P$, 
 we can instead view  $S$ as a surface in $U$
  as above, and
 then $M=\Omega/\rho(\pi_1M)$.
 If $M$ is compact  then $\sigma(M)$ is
 compact. A small deformation, $M'$ of $M$, gives a nearby
 strictly{\edit-}convex surface $\sigma(M')$ in $\xi M'$, that develops to another
  convex surface $S'$ in $\RR^3$ and gives a convex domain $\Omega'\subset P$. }
  This is the intuition for Koszul's theorem when $M$ is closed. 
  
 Now suppose that $M$ is a compact manifold, $N$, union a  (generalized) cusp $C$. One first shows
 that for the deformed cusp $C'$ that $\xi C'$ contains a nearby 
  strictly
 convex surface $\edit{\sigma(C '})$. Since $N$ is compact $\xi N'$ contains a nearby convex surface {\edit $\sigma(N')$}.
 One now deforms and joins $\sigma(N')$ and $\sigma(C')$, maintaining convexity,
  to obtain a strictly{\edit-}convex hypersurface 
 in $\xi M'$. This  implies $M'$ is properly convex.

Section \ref{sec:deforming (G,X)} describes the {\em $(G,X)$-Extension Theorem} (\ref{relopen}). This generalizes a
 well-known result for compact manifolds (the holonomies of $(G,X)$-structures
form an open subset of the representation variety) by providing  a {\em relative} version.
Section \ref{tautbundlesec} recalls the definition and properties of the tautological bundle. 
Section \ref{sec:hessian} reviews Hessian
metrics and  gives a characterization of properly convex manifolds in terms
of the existence of a certain kind of Hessian metric on the tautological line bundle. 
Section \ref{charfnssection} shows that various functions on properly
convex projective manifolds are uniformly bounded, including a proof of the folklore result
that they admit Riemannian metrics with all sectional curvatures  bounded in terms of  dimension.

The {\em Convex Extension Theorem}  (\ref{kleinianedndatamapopen}) is a version of
(\ref{relopen}) for properly convex manifolds with strictly{\edit-}convex boundary.
A consequence is  (\ref{extendpc}) below. Roughly this says that if you can convexly deform
the ends of a properly convex manifold then you can convexly deform the manifold.
 There is no assumption that the fundamental group of an end is virtually nilpotent for this result.

\begin{theorem}\label{extendpc}
Suppose $M=A\cup\mathcal B$ is a properly convex manifold with (possibly empty) 
compact strictly{\edit-}convex boundary, and
$A$ is a compact  connected submanifold of $M$ with $\partial A=\partial M\sqcup\partial\mathcal B$, and {\edit
$\mathcal B\cong\partial\mathcal B\times[0,\infty)= B_1\sqcup\cdots \sqcup B_k$,  
 and each $B_i$ is connected and $\pi_1$-injective in $M$.}

 Suppose $\rho:(-1,1)\to\Rep(\pi_1M)$ is continuous and  $\rho_t:=\rho(t)$ 
and $\rho_0$ is the holonomy of $M$.
Let $\clsdsub$ denote the space of closed subsets of $\RPn$ with the Hausdorff topology. Suppose 
 for all $1\le i\le k$ and all $t\in(-1,1)$ that 
\begin{itemize}
\item[(1)]   there is a properly convex set $\Omega_i(t)\subset \RPn,$
that is preserved by $\rho_t(\pi_1B_i),$
\item[(2)]  $P_i(t)=\Omega_i(t)/\rho_t(\pi_1B_i)$ is a properly convex manifold and $\partial P_i(t)$ is strictly{\edit-}convex,
\item[(3)] there is a projective diffeomorphism from $P_i(0)$ to $B_i,$
\item[(4)] $P_i(t)$ is diffeomorphic to $B_i,$
\item[(5)] the two maps $t\mapsto \cl(\Omega_i(t))$ and $t\mapsto\cl(\partial\Omega_i(t))$ into $\clsdsub$ are  continuous.
\end{itemize}
 Then there is $\epsilon>0$
such that for all $t\in(-\epsilon,\epsilon)$ there is a properly convex projective structure on $M$ with holonomy
$\rho(t)$ such that $\partial M$ is strictly{\edit-}convex and $B_i$ is projectively diffeomorphic to $P_i(t)$.
\end{theorem}

{\edit 
 If $\Omega\subset\RP^n$ we write $\PGL(\Omega)$ for the subgroup of $\PGL(n+1,\RR)$  that preserves $\Omega$}.
Section~\ref{gencusps} proves that generalized cusps contain {\em homogeneous} cusps (\ref{sttdcusp}):

\begin{theorem}\label{homgcuspsthm} Suppose $C=\Omega/\Gamma$ is a generalized cusp. Then $C$ contains a generalized cusp  $C'=\Omega'/\Gamma$ such that $\PGL(\Omega')$ acts transitively
on $\partial\Omega'$.
\end{theorem}

An algebraic argument  {\edit (\ref{transgrp}) uses $\Gamma$ is virtually nilpotent to show} that if $C=\Omega/\Gamma$ is 
a generalized cusp then $\Gamma$ has a finite index subgroup  that is a lattice
in a connected Lie group $T=T(\Gamma)$ that  is conjugate into the upper-triangular group.

Next (\ref{convexTorbit}) shows that the $T$-orbit of some point $p\in\Omega$  
is a strictly{\edit-}convex hypersurface
$S=T\cdot p$. 
The convex hull of $S$ is a domain $\Omega_T$  that is preserved by
 all of $\Gamma$, and we may shrink
$C$ to be $\Omega_T/\Gamma$ giving (\ref{homgcuspsthm}). 

From (\ref{homgcuspsthm}) it follows that generalized cusps are
{\em stable} (\ref{cuspstab}):
 if $\Gamma$ is deformed to a nearby virtual flag
group $\Gamma'$, then $T'=T(\Gamma')$ is a nearby Lie group, so $S'=T(\Gamma')\cdot p$  
is a nearby strictly{\edit-}convex hypersurface
which gives a nearby domain $\Omega_{T'}$ and a nearby 
generalized cusp $C'=\Omega_{T'}/\Gamma'$.  

The convex extension theorem and  the 
stability of  generalized 
cusps
  imply the main theorem (\ref{deformmfd}). Ballas, Cooper and Leitner \cite{BCL} 
 classify generalized cusps, and their properties are studied. 
This classification for 3-manifolds is given without proof in section \ref{3mfd}.

A function is {\em Hessian-convex} if it is smooth and has positive definite Hessian. This property is preserved
by composition with diffeomorphisms that are close to affine. 
Section \ref{smoothing}  {\edit reviews} various types of convexity and contains a theorem about approximating strictly-convex functions 
on affine manifolds by Hessian-convex ones.
 As a result, for our purposes {\em strictly} and {\em Hessian} convex are more or less equivalent.
Section~\ref{sec:Benzecri} is a short proof of Benz{\'e}cri's Theorem. We have put these results at the end of the paper with the intention of not breaking the narrative.

There is an entirely PL approach to (\ref{deformmfd})  which, however, we do not develop in this paper.
It is based on using the convex hull of the orbit of one point instead of a characteristic surface.

Theorem (\ref{deformmfd}) does not always remain true if $\partial M$
is convex but not strictly{\edit-}convex. However, in some cases, the theorem can still be applied. 
For instance, a hyperbolic manifold $M$ with compact, totally geodesic boundary is a
submanifold of a finite volume hyperbolic manifold with strictly{\edit-}convex
smooth boundary obtained by fattening.
 In particular, any small deformation in $\PGL(4,\R)$ of the holonomy in $PO(3,1)$ of a
compact Fuchsian manifold is the holonomy of a strictly{\edit-}convex  projective structure
on (surface)$\times[0,1].$

The reader only interested in the proof of (\ref{deformmfd}) when $M$ is compact need only read
section \ref{sec:deforming (G,X)} up to (\ref{deformGX}), and
then sections \ref{tautbundlesec}
to \ref{charfnssection}
stopping before (\ref{strongtopdef}). Those interested only in the proof of (\ref{extendpc}) can omit
section \ref{gencusps}.

Most of sections 1-4 is not new{\edit,} and there is considerable overlap
 in the first 5 sections with the results and methods  of Choi in \cite{Choi1}.  
Marquis determined the holonomies of properly
convex surfaces with cusps  \cite{MXY} and \cite{Marq1}.
In \cite{CL} Cooper and Long give 
a method of constructing fundamental domains for some strictly{\edit-}convex manifolds with cusps.
Using the main result of this paper, Ballas found
new properly convex structures on the figure
eight knot obtained by deforming the complete hyperbolic structure \cite{SB}. The type of geometry
in a generalized cusp can change during a deformation. For example a generalized
cusp with diagonal holonomy can {\em transition} to one with parabolic holonomy. 
 This is related to the study of geometric transition, see Cooper, Danciger and Wienhard
 \cite{CDW}.

 In section 5 of \cite{CLT1} there is a discussion of properly convex $n$-manifolds with parabolic holonomy.
Such a manifold is diffeomorphic to a product $C\times\RR$ and $\Gamma=\pi_1C$ is virtually nilpotent. The
manifold $C$ is compact iff the Hirsch rank of $\Gamma$ is maximal: namely $(n-1)$. In \cite{CLT1}
these cusps are called {\em maximal}, and it is shown in this case that $\Gamma$ is conjugate into $PO(n,1)$.
In general{\edit,} parabolics are not conjugate into $PO(n,1)$. In this paper 
frequent use is made of the fact that $\partial C$ is compact {\edit is equivalent to}
$H_{n-1}(C;\Ztwo)\cong\Ztwo$. For consistency with  \cite{CLT1}  one might call such generalized cusps {\em maximal}.
However, to keep this paper from becoming even longer,  
these are the only type of generalized cusp we consider, therefore we do not use the term {\em maximal}.

\subsection{Acknowledgments}
{\edit Work partially supported by U.S. National Science Foundation grants DMS 1107452, 1107263, 1107367 {\em RNMS: GEometric structures And Representation varieties} (the GEAR Network). 
Cooper was partially supported by NSF grants DMS 1065939, 1207068 and  1045292
and thanks IAS and the Ellentuck Fund for partial support.
Long was partially supported by grants from the NSF.
Tillmann was partially supported by Australian Research Council grant DP140100158.
We thank the referee for an excellent job.}

 \section{(G,X) structures and Extending Deformations}
 \label{sec:deforming (G,X)}\
 The goal of this section is a relative version 
 of the well-known fact (\ref{deformGX}) that for {\em compact} manifolds the set of holonomies
of $(G,X)$-structures is an open subset of the representation variety. The {\em Extension Theorem}
 (\ref{relopen})  implies that if $\mathcal B$ is a codimension-0 submanifold
of  a connected manifold $M$ with $M\setminus\mathcal B$ compact and  connected, then given a $(G,X)$-structure on  $M$ with holonomy $\rho$,
together with a nearby representation $\sigma$, and given a nearby $(G,X)$-structure
on  each component of $\mathcal B$ with holonomy
the restriction of $\sigma$,  there
is a nearby $(G,X)$-structure on $M$ with holonomy $\sigma$ that extends the structure
on $\mathcal B$. 
 
A {\em geometry} is a pair $(G,X)$ where $G$ is a Lie group which acts transitively and real-analytically on a manifold $X$.
A {\em $(G,X)$-structure}  on a manifold $M$ (possibly with boundary)
 is a maximal atlas of charts which takes values in $X${\edit,} so that transitions maps are locally
 the restriction of elements of $G$. A map between $(G,X)$ manifolds is a {\em $(G,X)$ map} if locally
 it is conjugate via $(G,X)$-charts to an element of $G$.  
 
 Let $\pi:\widetilde{M}\to M$ be (a fixed choice for) the universal cover of $M$.
 We regard $\pi_1M$ to be {\em \edit defined} as the group of covering transformations of this covering.
A local diffeomorphism $f:\widetilde{M}\to X$ determines a $(G,X)$-structure on $\widetilde{M}$.
If  all the covering transformations are $(G,X)$-maps then there is a unique $(G,X)$-structure on $M$
such that the covering space projection is a $(G,X)$-map. In this case $f$ is called a {\em developing map}
for this structure and determines  a homomorphism $\hol=\Hol(f):\pi_1M\to G$ called {\em holonomy}.

 For smooth manifolds $M^m$ and $N^n$ the set of smooth maps $C^{\infty}_w(M,N)$  has the {\em weak topology}
\cite[Hirsch page 35]{Hirsch}. 
The space of diffeomorphisms $\Diff(M)$ is a  subspace of $C^{\infty}_w(M,M)$. If $N=\R$
then $C^{\infty}_w(M):=C^{\infty}_w(M,\R)$.

The  set of all developing maps
is denoted $\devGX(M,(G,X))$ or just $\devGX(M)$. The $(G,X)$-structure
on $M$ given by $\dev\in\devGX(M)$ is written $(M,\dev)$.
There is a   natural map of $\devGX(M)$ into
$C_w^{\infty}(\interior\widetilde M,X)$ given by
restricting the developing map to  $\interior\widetilde M$.  This map is injective because $\interior\widetilde M$
is dense in $\widetilde M$.

\begin{definition}\label{geomtopdef}
The {\em geometric topology}  on $\devGX(M)$ 
 is the subspace topology from $C_w^{\infty}(\interior\widetilde M,X)$. \end{definition}

 Thus two developing maps  for $M$ are close
if they are close on a large compact set in   $\widetilde M$ that is disjoint from $\partial \widetilde M$.
The following is due to Thurston  \cite{WPT}, see also Goldman
 \cite{GOLD} and Choi \cite{CHOIGS}. The topology on $\Hom(\pi_1M,G)$ is the compact-open topology. 
  
\begin{proposition}[holonomy is open]
\label{deformGX} 
Suppose $M$ is a compact connected smooth manifold possibly with boundary.
Then $\Hol:\devGX(M,(G,X))\to \Hom(\pi_1M,G)$ is continuous and open.
\end{proposition}

Given $\dev_M\in\devGX(M)$ and $\dev_N\in\devGX(N)$ 
 a smooth map $f:M\to N$
 is {\em close} to a $(G,X)$ map if it is
covered by $F:\widetilde M\to\widetilde N$ and there is $g\in G$ such that $g\circ\dev_N\circ F$
is close to $\dev_M$ in  $C^{\infty}_w(\widetilde{M},X)$. This means
there is a large compact set $K\subset \interior \widetilde M$  and  some $g\in G$ such that 
 for each
$x\in K$ there is an open neighborhood $U\subset\widetilde{M}$ with $V=\dev_M(U\cap K)$ and the map
$g\circ\dev_N\circ F\circ(\dev_M|_{U\cap K})^{-1}$ is  close to the inclusion map $\edit V\hookrightarrow X$
 in $C^{\infty}_ w(V,X)$.
This notion of {\em close} depends on $\dev_M$ but not on the choice of developing map $\dev_N$
for a given $(G,X)$-structure on $N$.

 There is a nice description of what it means for  two developing maps in $\devGX(M)$ to be close when one of 
them is {\em injective}.
Suppose  $\dev\in\devGX(M)$ is injective and $\Omega=\dev(\widetilde M)$ and
$\rho=\Hol(\dev)$ and $\Gamma=\rho(\pi_1M)$. Then $N=\Omega/\Gamma$
is a $(G,X)$ manifold that is $(G,X)$-diffeomorphic to $M$. We choose
the universal cover $\widetilde{N}=\Omega$
then $\pi_1N=\Gamma$ by our definition as the group of covering transformations.
There is a homeomorphism  {\edit between spaces of developing maps $\devGX(M)\to \devGX(N)$}. 
\begin{definition}\label{GXbasepoint} Replacing $\devGX(M)$ by $\devGX(N)$ 
is called
{\em choosing $\dev$ as the  basepoint for the space of developing maps}.
\end{definition}

The developing map $\dev_*\in\devGX(N)$ for $N$ is
the inclusion map $i:\widetilde N\hookrightarrow X$
and $\Hol(\dev_*):\Gamma\hookrightarrow G$ is also the inclusion map.
If $N$ has no boundary then
$\devGX(N)$ is a subspace of $C^{\infty}_w(\widetilde N,X)$ so
$\dev'\in \devGX(N)$
 is close to $\dev_*$ if $\dev'$ is close to $i$ in $C^{\infty}_w(\widetilde N,X)$. 

 The idea for the extension theorem is the following. 
Suppose $(M=A_1\cup A_2,\dev)$ is a $(G,X)$-manifold with holonomy $\rho$, and 
$C=A_1\cap A_2\cong\partial A_i\times[0,1]$ is a connected collar neighborhood of $\partial A_i$ and the
inclusion map $C\hookrightarrow M$
is $\pi_1$-injective.
Suppose $\rho':\pi_1M\rightarrow G$ is a nearby homomorphism and $(A_i,\dev_i')$ is a nearby
$(G,X)$-structure to $(A_i,\dev|A_i)$ with holonomy $\rho'|\pi_1A_i$.  
If $C$ is compact we show there is a {\em nearby} $(G,X)$-structure on $M$
obtained by gluing $(A_1,\dev_1')$ to $(A_2,\dev_2')$ by a $(G,X)$-diffeomorphism. 
This is done in (\ref{developtrick}) which uses analytic continuation
to map  the submanifold $C\subset (A_1,\dev_1')$ into $(A_2,\dev_2')$ by a $(G,X)$-diffeomorphism.

 \begin{lemma}[lifting developing maps]\label{developtrick} In this statement all manifolds and maps are $(G,X)$.
  Suppose $N$ and $P$ 
 are connected manifolds and $\theta:\pi_1N\to\pi_1P$ is a homomorphism such that $\hol_N=\hol_P\circ\theta$. Suppose 
 $\pi_P:\widetilde{P}\to P$ and $\pi_N:\widetilde{N}\to N$ are universal covers and $i:Q\hookrightarrow \widetilde{N}$
 is the inclusion map of a {\edit path-}connected set $Q$ with $\pi_N(Q)=N$. Suppose $\dev_N\circ i:Q\to X$ lifts to a map $j:Q\to\widetilde{P}$ such that $\dev_P\circ j=\dev_N\circ i$.
 Then there is $k:N\to P$  covered by $\widetilde k:\widetilde{N}\to\widetilde{P}$ that extends $j$.  \end{lemma}
  \begin{figure}[ht]	 
\begin{center}
		 \includegraphics[scale=1]{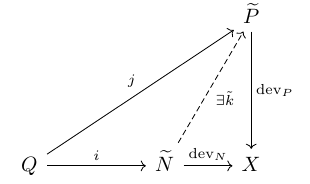}
\end{center}
\end{figure}
\begin{proof} Because the covering translates of $Q$ cover $\widetilde{N}$ and the
  manifolds $N$ and $P$ have (via $\theta$) the same holonomy, $j$
can be extended by analytic continuation to an equivariant $(G,X)$-map $\widetilde k:\widetilde{N}\to \widetilde{P}$. Equivariance implies
 $\widetilde k$ covers  a $(G,X)$-map
$k:N\to P$. \end{proof}

 For a {\em closed} manifold $M$ two nearby developing maps with the same holonomy differ
by composition with a diffeomorphism of $M$ close to the identity. If $M$ has boundary then the boundary can move,
so this result only holds outside a small neighborhood of the boundary, cf. the discussion  in I.1.5 and I.1.6 leading to the proof  of I.1.7.1 in Canary, Epstein and Green \cite{CEG}.

 If $P$ is a smooth manifold then $\Diff_ 0(\widetilde{P},{\edit P})\subset\Diff(\widetilde P)$
 is {\edit defined to be} the  identity component of the subgroup of diffeomorphisms that cover an element of $\Diff(P)$. 
   The next result says that if two developing maps are close, and have the same holonomy then, after changing
   one by a small isotopy, the developing maps are equal on a compact submanifold in the interior.
    
\begin{corollary}\label{isotopsubstructure}  Suppose $P$ is a smooth manifold.
Let $\rho\in\Hom(\pi_1 P,G)$
be the holonomy of $\dev\in\devGX(P,(G,X))$ and $\devGX_{\rho}(P)\subset\devGX(P)$  be the subspace
of developing maps with holonomy $\rho$. Then the map $\Diff_ 0(\widetilde{P},{\edit P})\to\devGX_{\rho}(P)$
given by $f\mapsto\dev\circ f$ is an open map.

It follows that, if $N$ is a compact codimension-0 manifold  in the interior of $P$, and $\dev'\in\devGX_{\rho}(P)$
is close enough to $\dev$, then there is $k\in\Diff(P)$ covered by $\widetilde{k}\in\Diff_ 0(\widetilde{P},{\edit P})$ such that
$\dev=\dev'\circ \widetilde k$ on $N${\edit,} and $k$ is isotopic to the identity
by a small isotopy supported in a small neighborhood of $N$.
 \end{corollary}
 
 \begin{proof} Let $\pi_P:\widetilde{P}\to P$ and $\pi_N:\widetilde{N}\to N$ be  universal covers.
 Let $Q\subset\widetilde{N}$ be a compact connected manifold such that $\pi_N(Q)=N$. 
 Since $\dev|_Q:Q\to X$ factors through the inclusion $j:Q\hookrightarrow\widetilde{P}$ and 
 $ (\pi_p\circ j)(Q)\subset\interior(P)$ it follows that if
 $\dev'$ is close enough to $\dev$ then $\dev'|Q:Q\to X$ has a nearby lift $j':Q\to \widetilde{P}$.
 By (\ref{developtrick}) 
there is a $(G,X)$-map
$k:(N,\dev'|N)\to (P,\dev|N)$ that lifts to a map that extends $j'$. If $\dev'$ is sufficiently close to $\dev$,
 then $k$ is close to the {\edit inclusion}, and the result now follows from the  fact that a diffeomorphism 
 (namely $k$) close to an inclusion
is ambient isotopic to the inclusion by a small ambient isotopy, Lima \cite{LIMA}.
\end{proof}

 Given a connected $\pi_1$-injective submanifold $B\subset M$ we fix a choice of some component $\widetilde{B}\subset\widetilde{M}$
of the preimage $B$ in 
the universal cover of $M$ and {\edit identify} $\pi_1B$ {\edit with} those covering transformations 
{\edit of $\widetilde M$} that preserve $\widetilde{B}$.
If $\dev_M$ is a developing map for a $(G,X)$-structure on $M$  
the {\em restriction to $B$ is $\dev_{M|B}:=\dev_M|\widetilde{B}$.}

Suppose $M$ is a smooth manifold with (possibly empty) boundary and
$\mathcal B\subset M$ is a codimension-0 submanifold that is a closed subset such
 that $A=\cl(M\setminus \mathcal B)$ is compact
manifold.
Suppose
$\mathcal B=B_1\sqcup\cdots B_k$ has $k<\infty$ connected components  and each component is $\pi_1$-injective.
Define the {\em relative-holonomy space}
$$\enddata(M,\mathcal B,(G,X))\subset  \Hom(\pi_1M,G)\times\prod_{i=1}^k \devGX(B_i,(G,X))$$ to be the subset of all
 $(\rho,\dev_1,\cdots,\dev_k)$ such that $\Hol(\dev_i)=\rho|\pi_1B_i$. This space
 has the subspace topology of the product topology.

\begin{definition}\label{defrelhol}
A developing map for $M$ restricts to give developing maps on each component
of $\mathcal B$  and this defines the {\em relative holonomy map} $\enddatamap:\devGX(M,(G,X))\longrightarrow \enddata(M,\mathcal B,(G,X))$
$$ \enddatamap(\dev_M)=(\Hol(\dev_M),\dev_{M|B_1},\cdots,\dev_{M|B_k})$$
\end{definition}
This map depends on a fixed choice of one component $\widetilde{B}_i\subset\widetilde{M}$ for each $i$.
In the special case that  $\mathcal B$ is empty then $\enddatamap=\Hol$.
We will apply this when $\mathcal B$ consists of the ends of $M$ which is why the symbol $\enddatamap$ is
used. However the result is of interest even when everything is compact.

\begin{theorem}[Extension theorem]
\label{relopen} Suppose $M$ is a smooth manifold with (possibly empty) boundary and
$\mathcal B\subset M$ is a  $\pi_1$-injective codimension-0 submanifold that is a closed subset such that 
$A=\cl(M\setminus \mathcal B)$ is a compact  connected manifold.
 Then
 $\enddatamap:\devGX(M,(G,X))\longrightarrow \enddata(M,\mathcal B,(G,X))$ is continuous and open.
\end{theorem}
\begin{proof} Continuity is easy. We prove openness. 
For simplicity we will assume that ${\mathcal B}=B$
 is connected; the multi-end case merely requires more notation. 
 Suppose 
 $\enddatamap(\dev_{\rho,M})=(\rho,\dev_{\rho,M|B})$
 and $(\sigma,\dev_{\sigma,B})$ is nearby in $\enddata(M,\mathcal B,(G,X))$. 
 
  Let $E\subset B$ be a compact collar of $\partial B$ and
$C=A\cup E$.
 By (\ref{deformGX}) there is $\dev_{\sigma,C}:\widetilde C\to X$ close to $\dev_{\rho,M|C}$ 
 with holonomy (the restriction of) $\sigma$.
   Using (\ref{isotopsubstructure})  to change $\dev_{\sigma,E}$ by a small isotopy, we may assume
   $\dev_{\sigma,C}$ and $\dev_{\sigma,B}$ are  equal on a smaller collar $E^-\subset E$. This gives a developing
   map $\dev_{\sigma,M}:\widetilde{M}\to X$ close to $\dev_{\rho,M}$ that is
    given by $\dev_{\sigma,C}$ on $\widetilde A$ and $\dev_{\sigma,B}$ on 
   $\widetilde B$.
 \end{proof}

\section{Tautological  Bundles}\label{tautbundlesec}
There is a bundle $\xi M\to M$ over a real projective manifold $M$ called the tautological line bundle.
In the next section we show that $M$ is properly convex iff $\xi M$ admits a certain kind of metric.

{\em Radiant affine geometry}  is $\RA=(\GL(n+1,\RR),\RR^{n+1}\setminus 0)$.
 A manifold with this
structure is called a {\em radiant affine manifold}.
 It {\em ought} to be called a {\em linear manifold} since
transition functions are linear maps. 

{\em Projective geometry} over a real vector space $V$ 
 is $\PP=(\PGL(V),\PP(V))$ where $\PP(V)=(V-0)/\RR^*$. 
{\em Positive projective space}  is $\PPp(V)=(V-0)/\RR_{\edit+}$ and
 the action of $\GL(V)$ on $V$ induces an effective action of
$\PGLp(V)=\GL(V)/\RR_{\edit+}$ on 
$\PP_+(V)$ which gives
{\em positive projective geometry}  $\PPp=(\PGLp(V),\PP_+(V))$.
If $X\subset V$ we write
$\PP(X)$ for its image in $\PP(V)$ and similarly $\PPp(X)\subset\PPp(V)$.

We identify $\PP_+(\RR^{n+1})$ with the unit sphere $\SS^n\subset\RR^{n+1}$
and {\em radial projection} $\pi_{\xi}:\RR^{n+1}\setminus 0\to \SS^n$
is $\pi_{\xi}(x)=x/\|x\|$. An action of $A\in\SL_{\pm}(n+1,\RR)$  on $\SS^n$ is given by $A(\pi_{\edit \xi} x)=\pi_{\edit\xi}(Ax)$.
Clearly $\PPp\cong \SS:=(\SL_{\pm}(n+1,\RR),\SS^n)$.

For each of the geometries $\mathbb G$ above there is a space of
developing maps $\devGX(M,\mathbb G)$ with the geometric topology.
By lifting developing maps one obtains:

\begin{proposition}\label{holonomylifts} The natural map $\devGX(M,\SS)\to\devGX(M,\proj)$ is $2:1$. \end{proposition}
{\edit Thus
every projective structure on $M$ lifts to a
positive projective structure. If $M$ is a real projective $n$--manifold, 
then the holonomy $\rho:\pi_1M\longrightarrow \PGL(n+1,{\mathbb R})$
lifts to $\widetilde{\rho}:\pi_1M\longrightarrow \SL_{\pm}(n+1,{\mathbb R})$ and $\dev:\widetilde  M\to \RPn$ lifts to $\widetilde{\dev}:\widetilde M\to\SS^n$.
}
We will pass back and forth between projective geometry and positive projective geometry without mention.
 The {\em tautological bundle} over $\SS^n$ is 
$\pi_{\xi}: \R^{n+1}\setminus 0\longrightarrow \SS^n$.
The total space is a radiant affine manifold.
There is an action of $(\RR,+)$ on the total space
called
the {\em radial flow} given by
  $\Phi_t(x)=\exp(-t)x$.  This bundle is a principal $(\RR,+)$-bundle.
  All this structure is preserved by the action
  of $\GL(n+1,\RR)$ on $\RR^{n+1}\edit -0$ covering the action of $\SL_{\pm}(n+1,\RR)$ on $\SS^n$

Suppose $M$ is a projective $n$--manifold defined by a developing map
$\dev_{_M}:\widetilde M\to\SS^n$  with holonomy $\rho:\pi_1M\to\SL_{\pm}(n+1,\RR)$ and
 with universal cover $\pi_{_M}:\widetilde M \to M$. Then
 pullback gives
 a  line-bundle $\pi_{\xi}:\xi\widetilde{M}\to \widetilde{M}$ where
$$\xi\widetilde{M}=\{(\widetilde m,x)\in\widetilde{M}\times(\RR^{n+1}\setminus 0) : \dev(\widetilde m)=\pi_{\xi}(x)\} \qquad
 \pi_{\xi}(\widetilde m,x)=\widetilde m$$
Recall that we {\em defined} $\pi_1M$ as the group
of covering transformations of $\widetilde{M}$.
There is an action of ${\edit\tau\in} \pi_1M$ on $\xi \widetilde M$ given by $\tau\cdot(\widetilde m,x)=(\tau(\widetilde m),(\rho(\tau))(x))$.
The quotient  of $\xi\widetilde M$ by $\pi_1M$ is called the {\em tautological bundle} $\xi M$. There is a natural bundle map $\xi_{_M}:\xi M\to M$
given by $\xi_{_M}[\widetilde m,x]=\pi_{_M}(\widetilde m)$.
There is also a natural
 radiant affine manifold structure on $\xi M$ with developing map $\dev_{\xi}:\xi\widetilde{M}\to\RR^{n+1}\setminus 0$
 given by $\dev_{\xi}(\widetilde m,x)=x$, {\edit and with holonomy $\rho\circ(\xi_M)_*$}

There is a radial flow on $\xi M$ given by $\Phi_t[m,x]= [m,\exp(-t)\cdot x]$ so $\xi M$ is a principal $(\RR,+)$ bundle {\edit over $M$}.
   Orbits are called
  {\em flow-lines}.
    The {\em tautological circle bundle} is $\xi_1M=\xi M/\Phi_1$. It is sometimes called an {\em affine suspension}.
    Observe that the developing maps of $\xi M$ and $\xi_1M$ are the same.
    
 {\edit \begin{definition}\label{trick} We make  use of the following {\em covering space trick}. 
 If $M$ is a compact projective manifold (possibly with boundary) then $\xi_1 M$ is a compact affine
  manifold. Since $\pi_1(\xi_1M)=\pi_1M\oplus\ZZ$, 
  small deformations of the holonomy of $M$ give small deformations
  of the holonomy of $\xi_1M$. The latter is compact, so (\ref{deformGX}) implies there is a nearby affine structure on $\xi_1M$ with
  the deformed holonomy,
  and hence also a nearby 
  affine structure on the non-compact manifold $\xi M$.\end{definition}}
 
 \begin{definition}
 A {\em flow function} is a function $c:\xi M\to \RR$ that
  is {\em flow equivariant}, which means  that
$c(\Phi_t(p))=t+c(p)$  for all $p,t$.
\end{definition}
A flow function determines  a section $\sigma:M\to\xi M$ of the bundle 
$\xi_{_M}:\xi M\to M$ defined by $c(\sigma(x))=0$.
Conversely a section $\sigma$ determines a flow function $c$ via $c(x)= -t$ if $\Phi_t(x)=\sigma(\pi x)$.
So  the negative of the flow function is the amount of time  it takes a point to flow onto
{\edit this }section.

We will mostly be concerned with the situation where  $\Omega\subset \RPn$ is properly convex and $\dev_{_M}:\widetilde M\to\Omega$ is injective. In this case $\dev_{\xi}$ is a diffeomorphism
onto the cone $\Cone\Omega\subset\RR^{n+1}\setminus 0$ {\edit defined in section \ref{charfnssection}}. This identifies $\xi M$ with $\Cone\Omega/\Gamma,$
where $\Gamma=\hol(\pi_1M)$. Moreover $\dev_{_M}$ identifies $\widetilde M$ with a subset of
$\SS^n$. Using these identifications $\xi_{_M}:\xi M\to M$ is covered by $\pi_{\xi}$.

\section{Hessian Metrics and Convexity}\label{sec:hessian}

The ideas in this section go back to Koszul~\cite{Kos1, Kos2}, and we have followed the 
exposition in Shima and Yagi \cite{SHIMAYAQI}.
However our notation and terminology are somewhat different.

Suppose $M$ is a simply connected affine manifold 
and $\dev: M\to\RR^n$ is some developing map. Given $a,b\in M$ 
a {\em segment in $M$ from $a$ to $b$} is a map $\gamma:[u,v]\to M$
such that $\gamma(u)=a$ and $\gamma(v)=b$ and $\dev\circ\gamma$ is affine. We often denote such a  segment by $[a,b]$.
It is a {\em unit segment} if $[u,v]$ is the unit interval $I:=[0,1]$.
A {\em ray} in $M$ is a non-constant affine map $\gamma:[0,s)\to M$ with $s\in(0,\infty]$ which 
does not  extend to a segment. 
A {\em unit triangle} in $M$ is a map $\tau:\Delta\to M$ such that $\dev\circ\tau$ is affine where $\Delta\subset\RR^2$
is the triangle with vertices $0,e_1,e_2$. The sides of a triangle are segments.

A $C^2$ function $c:M\to\RR$ is {\em  Hessian} convex if for every
(non-degenerate) segment $\gamma:[-1,1]\to M$ the function $F=c\circ\gamma$
satisfies $F''>0$. Then $c$ defines a Riemannian metric on $M$ via
$\|\gamma'(0)\|^2=F''(0)$ called a {\em Hessian metric}. See Shima \cite{Shima} for a discussion.

An affine manifold $M$  has {\em convex boundary} if for each $p\in \partial M$
 there is an affine coordinate chart
$(U, \phi)$ with $p\in U$
 and a closed half-space $H\subset\RR^n$
such that $\phi(U)\subset H$ and $\phi(p)\in \partial H$.  

\begin{theorem}\label{hessianconvex} Suppose $M$ is a simply-connected
 affine $n$--manifold with convex boundary and $M$ has a Hessian metric 
that makes $M$ into a complete metric space.  Then $\dev:M\rightarrow\RR^n$ is an affine isomorphism  onto a convex subset of $\RR^n$.
\end{theorem}
\begin{proof} It suffices to show that for every pair of segments $[p,a]$ and $[p,b]$ in $M$  there is a segment
$[a,b]$ in $M$. This is because every pair of points in $M$ can be connected by a  polygonal path composed
of finitely many segments. One may replace a pair of adjacent segments in this path by one segment. It 
 follows that $a$ and $b$ are contained in a single segment. 
Since $\dev:M\to\RR^n$ sends segments to segments,
  if $\dev(a)=\dev(b)$ then the segment in $M$ from $a$
 to $b$ maps to a segment in $\RR^n$ with both endpoints the same. Hence  this segment is a single point,
 so $a=b$, and $\dev$ is injective. Since every pair of points in $M$ are contained in a segment, the same is true
 of $\dev(M)$, therefore $\dev(M)$ is convex. Thus $\dev$ is an affine isomorphism onto a convex set.

Given unit segments $\alpha:I\to [p,a]$ and $\beta:I\to [p,b]$ let $\mathcal I\subset I$ be the set
of $t \in I$ such there is a unit triangle  $\tau:\Delta\rightarrow M$ with vertices $p=\tau(0)$ and $\alpha(t)=\tau(e_1)$
 and $\beta(t)=\tau(e_2)$. Then $\mathcal I$ is connected
and contains $0$. It suffices to show $\mathcal I=I$ since then $\gamma(t)=\tau(te_1+(1-t)e_2)$ is a segment
containing $a$ and $b$.

Since $\partial M$ is convex it easily follows from the standard argument about sets
with  convex boundary that $\mathcal I$ is open. To show $\mathcal I$ is closed we may assume $\mathcal I=[0,1)$
by reparameterizing.  After this reparameterization, $\tau$ is defined on the interior of $\Delta$ and
also on  the two sides given by $\alpha$ and $\beta$ since $a,b\in M$. 
However $\tau$ might not be defined on part of
the {\edit side connecting $e_1$ to $e_2$.}

The Hessian metric is given by some function $c:M\to\RR$.
Given any segment $\gamma$ {\edit in $M$} define $\ell(\gamma)$ to be its length.
If $\gamma$ is a unit segment and $F=c\circ\gamma$ then
 $$\ell(\gamma)=\int_0^{1}\sqrt{F''(t)}dt$$
By the Cauchy-Schwartz inequality {\edit for $L^2$}
$$\ell(\gamma)\le\left(\int_0^1 F''(t)dt\right)^{1/2}\left(\int_0^1 dt\right)^{1/2}\le\sqrt{|F'(1)|+|F'(0)|}$$
For $s\in[0,1)$ there is a unit segment $\gamma_s$ given by $\gamma_s(t)=\tau(s(t  e_1+(1-t)e_2))$
with endpoints $\alpha(s)$ and $\beta(s)$. 
By the triangle inequality
$$d(p,\gamma_s(t))\le d(p,\gamma_s(0))+d(\gamma_s(0),\gamma_s(t))\le \ell(\alpha)+\ell(\gamma_s)$$  
The function $G(s,t)=\tau(s(t  e_1+(1-t)e_2))$ 
is defined and smooth for all $(s,t)$ in the domain {\edit $[0,1)\times[0,1]\cup \{1\}\times\{0,1\}$}.
 By compactness there is $K>0$ such
that  $|\partial  G/\partial t| \le  K$  for all $s\in[0,1]$ and $t\in\{0,1\}$.
It follows that for all $s\in[0,1)$ and $t\in[0,1]$ we have
$$d(p,\gamma_s(t))\le \ell(\alpha)+\sqrt{2K}=:R.$$
Since the metric on $M$ is complete{\edit,} the ball $P\subset M$ with center $p$ and radius $R$ 
 is compact and contains all the segments $\gamma_s$ {\edit with $s\in[0,1)$}.
It follows that $\gamma_s$ converges to a segment $\gamma_1\subset P$ as $s\to 1$ so $1\in\mathcal I$.
\end{proof}

\begin{definition}\label{defflowfn} If $M$ is a projective $n$--manifold a {\em convexity function for $M$}
is a Hessian-convex flow function $c:\xi M\to \RR$.   
It is {\em complete} if the Hessian metric given by $c$ is complete.
\end{definition}

The flow-equivariance of $c$ implies the radial flow acts by isometries of the Hessian
metric on $\xi M$ given by $c$. The 1-form $d c$ is preserved by the flow and therefore
is the pullback of a 1-form $\alpha$ on $\xi_1M$. Koszul works with $\alpha$ but we work with $c$.

 The {\em Hilbert metric} on a properly convex subset $\Omega\subset\RP^n$  is a {\edit Finsler} metric given
by the {\em Hilbert-Finsler norm}
 on $T_x\Omega$, see 
Papadopoulos and Troyanov \cite{Pap} and Marquis \cite{Marquishandbook}. For the definition of {\em Hessian-convex hypersurface} see the start of section \ref{smoothing}. 
 The next result is that the flow function $c$ is Hessian-convex iff the level set $c^{-1}(0)$ is a Hessian-convex hypersurface that
 is convex in the {\em backward} direction of the flow, {\edit i.e. $c^{-1}(-\infty,0]$ is convex}.

\begin{lemma}\label{backwardconvex} Suppose $M$ is properly
convex and $N=\xi M$  and $\|\cdot\|$ is the Hilbert-Finsler norm
 on $T_xN$.  Suppose 
 $c:N\to \RR$ is a flow function and $S=c^{-1}(0)$.
Then at $x\in S$  there is a splitting $T_xN=V\oplus E$ which is orthogonal
with respect to $Q:=D^2_xc$ where 
$V=\ker d_xc\subset T_x N$ is 
 the tangent hyperplane to the hypersurface $S$ and $E=\langle e\rangle$ where $e=\Phi'_0(x)$ is a tangent vector
  to the flow.
  
  Moreover $Q(e,e)=\|e\|^2=1$. Thus if $\kappa\in[0,1]$
 then $Q\ge\kappa\|\cdot\|^2$ iff $Q|V\ge\kappa(\|\cdot\|\ |V)^2$.
 
 In particular $c$ is Hessian-convex iff $S$ is a Hessian-convex hypersurface
 that is convex in the {\em backward} direction of the radial flow
\end{lemma}
\begin{proof} This is a local question so
it suffices to assume $\xi M$ is a properly convex cone in $\RR^{n+1}\setminus 0${\edit,}
and $S$ is a hypersurface{\edit,} and the radial flow is $\Phi_t(x)=\exp(-t)\cdot x$.
Since $c$ is a flow function $c(\Phi_t(x))=c(x)+t$.
This implies $c(s\cdot x)=c(x) -\log s$.
From this it follows that $D^2_xc(e,v)=0$ for all $v\in V$ which
proves the $Q$-orthogonality of the direct sum.

The Hilbert-Finsler norm on $(0,\infty)$ is $ds/s$. The radial flow on $(0,\infty)$
is $\Phi_t(s)=\exp(-t)s$ so $e=\Phi'_0(s)=s\cdot \partial/\partial s$ and $\|e\|=1$.
Moreover 
\[
Q(e,e)=s^2Q(\partial/\partial s,\partial/\partial s)=s^2d^2(-\log s)/d s^2=1.
\]
 Observe that $Q|V$ is positive definite iff $S$ Hessian-convex in the backward direction of the flow.
\end{proof}

\begin{theorem}\label{RFhessianisconvex} Suppose $M$ is a projective manifold with (possibly empty) convex boundary
 and $c:\xi M\to\RR$  is a complete convexity function.
 Then $M$ is properly convex.
\end{theorem}
\begin{proof} By (\ref{hessianconvex}) $\dev:\xi\widetilde M\to\RR^{n+1}\setminus 0$ is injective and the
image is a convex cone
$\Cone\subset\RR^{n+1}$. It suffices to show {\edit that $\Omega:=\pi_{\xi}(\Cone)$} is properly convex.
{\edit Let $\pi:\xi\widetilde M \to\xi M$ be the universal cover and 
$\widetilde c=c\circ\pi$.}  The function 
$f={\widetilde c}\circ\dev^{-1}:\Cone\to \RR$ is {\edit a complete convexity function: it is} strictly{\edit-}convex{\edit,}
and the hypersurfaces $S_t=f^{-1}(t)$ are connected{\edit,} and strictly{\edit-}convex{\edit,} and foliate $\Cone$. 
The radial flow on $\xi {\edit\widetilde M}$ is conjugate to the radial flow $\Phi_{t}(x)=\exp(-t)\cdot x$ on $\RR^{n+1}$ so
$\Phi_s(S_t)=S_{t+s}$. Define $S:=S_0$.

  \begin{figure}[ht]	 
\begin{center}
		 \includegraphics[scale=1]{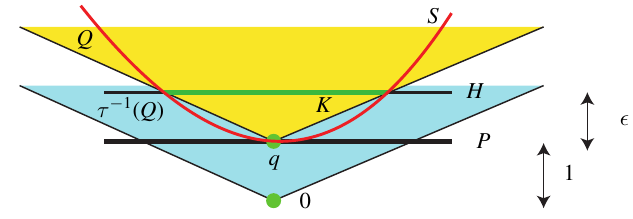}
\end{center}
\caption{Flowing S backward}\label{fig1}
\end{figure}

 Let $q$ be a point in the interior of $S$. We can choose coordinates in $\RR^{n+1}$
so that $S$ is tangent at $q=(1,0,\cdots,0)=e_{1}$ to the hyperplane $P$ given by $x_{1}=1$
and $S$ lies on the opposite side of $P$ to $0$. 

The sublevel set $W=f^{-1}(-\infty,0]=\cup_{t\le0}\Phi_t(S)\subset\Cone$ is obtained by flowing $S$ backward. 
Let $H$ be the hyperplane $x_{1}=1+\epsilon$. Refer to Figure \ref{fig1}.
We do not know that $S$ is {\em properly} embedded in $\RR^{n+1}$. However
if $\epsilon>0$ is small enough we can work in a chart
for a small neighborhood of $\dev^{-1}(q)$ in $\xi\widetilde M$ 
and see that $K=H\cap W$ is a compact convex set 
and $\partial K= H\cap S$.
 
  Let $Q$ be the convex cone consisting of the set of rays starting at $q$ 
and intersecting $K$. 
Since $q\in\partial W=S$ and $W$ is convex
it follows that $Q$ contains the subset of $W$ above $H$.
Unit vertical translation upwards $\tau:\RR^{n+1}\to\RR^{n+1}$ is given by $\tau(x)=x+e_1$. Note
that $\tau(Q)\subset Q$.
Since  we can assume $\epsilon<1$  it follows that $\tau(S)$ is above $H$, therefore $Q$ contains $\tau(S)$. Hence
 $\tau^{-1}(Q)$  contains $S$.
 Since $\tau^{-1}(Q)$  is the cone from $0$ of $\tau^{-1}(K)$, it is preserved by $\Phi$ so it contains
the entire orbit $\flow\cdot S=\Cone$. It follows that  $\Omega=\pi_{\xi}(\Cone)\subset\RPn$ 
is contained in  $\pi_{\xi}(\tau^{-1}(K))$. Since  $\tau^{-1}(K)$ is a compact convex set
in $x_n=\epsilon$ it follows that $\Omega$ is properly convex.  
\end{proof}

\section{The Characteristic Convexity Function}\label{charfnssection}

In this section $V={\mathbb R}^{n+1}$ and $\Omega\subset\sphere(V)=\SS^n$ is an open 
properly convex set.  The open convex cone ${\Cone}\Omega\subset V$ consists of all $t\cdot v$ with $v\in\Omega$ and $t>0$.
The {\em dual cone} ${\Cone}\Omega^*\subset V^*$ is the set of all $\phi\in V^*$
 with $\phi(x)>0$ for all $x\in{\Cone}\overline\Omega$. The {\em dual domain} is 
 $\Omega^*=\PP(\Cone\Omega^*)\subset\PP(V^*)$.
The
{\em characteristic function} 
 $\charfn=\charfn_{_\Omega}:{\Cone}\Omega \longrightarrow {\mathbb R}^+$
of Koecher~\cite{Koecher1957} and Vinberg~\cite{Vinberg1963homog}
is defined by
$$\charfn(x)=\int_{{\Cone}\Omega^*}e^{-\psi(x)}d\psi$$ 
where $d\psi$ is a {\edit fixed choice of} Euclidean volume form on $V^*$. 
This function is real analytic, non-negative, 
and ${\charfn}(tx)=t^{-(n+1)}\charfn(x)$ for $ t > 0$. More generally, 
if $A$ is in the subgroup $\GL({\Cone}\Omega)\subset \GL(V)$ that preserves ${\Cone}\Omega,$  then 
$\charfn(Ax)=(\det A)^{-1} \charfn(x)$.
The level sets of $\charfn,$ called {\em characteristic hypersurfaces}, are smooth, convex, and meet each ray in
${\Cone}\Omega$ once transversely. The {\em characteristic section} is the map
$\sigma_{_\Omega}:\Omega\longrightarrow {\Cone}\Omega$ 
 given by  
$$\sigma_{_\Omega}(x)=x\cdot (\rchi(x))^{1/(n+1)}$$
It has image  {\em the characteristic hypersurface} 
 $S_{\Omega}=\rchi^{-1}(1)$.
 The radial flow $\flow_t(x)=e^{-t}\cdot x$ on $V$ preserves ${\Cone}\Omega$ 
  {\edit and  there is a flow function on ${\Cone}\Omega$ given by}
$$\flowfn=\flowfn_{_\Omega}=(n+1)^{-1}\log \charfn$$ 
 The Hessian $D^2 c$ is a positive definite quadratic 
form  at each point of ${\Cone}\Omega$  and gives a complete metric on ${\Cone}\Omega$. 
Thus $\flowfn_{_\Omega}:\Cone\Omega\to \RR$ is a complete convexity function  
called the {\em characteristic convexity function}.  {\edit References for the above are
Goldman \cite{Gold1}, 
and \cite{Gold2}.}

If $\Gamma\subset\SL_{\pm}(\Cone\Omega)$ is the holonomy of a properly
convex manifold $M=\Omega/\Gamma$  with developing map $\dev$, then $\xi M$ is identified with ${\Cone}\Omega/\Gamma$.
Since $c_{_\Omega}$ is preserved by $\Gamma$ it covers a map $c_{\dev}=c_{_M}:\xi M\to \RR$.  
 This is a convexity function for $M$ called
the {\em characteristic convexity function} for $M$.

 \begin{definition}\label{devcdef} The subspace
 $\devc(M)\equiv\devGX_c(M,\proj_+)\subset \devGX(M,\proj_+)$ consists of the developing maps of properly 
 convex structures for which $\partial M$ is
 strictly{\edit-}convex. \end{definition}

\begin{proof}[Proof of (\ref{deformmfd}) when $M$ is closed.] {\edit Suppose} $M$ is properly convex 
{\edit with holonomy $\rho$ and $c_M:\xi M\to\RR$
  is a characteristic convexity function.} 
If
 {\edit $\rho'$ is close to $\rho$} then, by
 (\ref{deformGX}), there
is a  radiant affine manifold $N_1$ {\edit with holonomy $\rho'$} and a diffeomorphism $f:\xi_1 M\to N_1$  that is 
 close to an affine map.
Taking infinite cyclic covers gives a map $F:\xi M\to N$ that is   close to affine.
The  hypersurface $S=c^{-1}(0)\subset\xi M$ maps to a hypersurface in $N$. {\edit Since $S$ is compact, Hessian-convex, and 
transverse to the radial flow,
if $F$ is close enough in $C^2(\xi M, N)$ to affine, then $F(S)\subset N$ is Hessian-convex,
and transverse to the radial flow $\Phi_N$ on $N$.} 
This section of the radial flow defines a convexity function on $N$ by (\ref{backwardconvex}). 
This convexity function  is complete
because $N_1$ is compact and
every Riemannian metric on a compact manifold is complete. It follows from (\ref{RFhessianisconvex}) that $N/\Phi_N$ is properly
convex.\end{proof}

 From here until (\ref{strongtopdef}) we allow $\Omega$ to have boundary $\partial\Omega$
 that is an open subset of $\Fr(\Omega)$.
Let $\clsdsub$ be the set of closed subsets of $\SS^n$ equipped with the Hausdorff topology.
Let  $\domains$ be   the set of properly convex  $n$--manifolds  $\edit\Omega\subset\SS^n$ with (possibly empty) 
strictly{\edit-}convex boundary.
  There is an injective map $\iota:\domains\to \clsdsub\times \clsdsub$
  defined by $f(\Omega)=(\edit\cl(\Omega),\cl(\partial\Omega))$.
 The {\em Hausdorff boundary topology} on $\domains$ is the subspace topology given by this embedding. 
 Thus a neighborhood of $\Omega$ consists of {\edit all $\Omega'\in\domains$ {\edit close to} $\Omega$ such
 that $\partial\Omega'$ is also close
 to $\partial\Omega$}. This topology is given by a metric.

\begin{definition}\label{strongtopdef} The {\em strong geometric topology} on $\devc(M)$ is the  smallest refinement of the geometric topology
 such that the map $\devc(M)\to\domains$ given by $\dev\mapsto\image(\dev)$ is continuous.
\end{definition}
If $M$ is closed, the strong geometric topology
equals the geometric topology {\edit because
fixed points of elements of the holonomy are dense in $\partial\image(\dev)$.}
 In general two developing maps are close in this topology if they are close in  the $C^{\infty}$ topology 
on a large compact set
in the universal cover of the interior 
and, in addition, their images are close in the above sense. This can be expressed more simply using basepoints
in the space of developing maps as in  (\ref{GXbasepoint}):

Suppose
$\dev_{\rho}\in\devGX_c(M)$
and $\rho=\Hol(\dev_{\rho})$ and $\Gamma=\rho(\pi_1M)\subset\SL{_\pm}(n+1,\RR)$ and
$\Omega_{\rho}=\image(\dev_{\rho})\subset\SS^n$.
\emph{Choosing $\dev_{\rho}$ as a basepoint} means: replace $M$ by $\Omega_{\rho}/\Gamma$.
Thus $\dev_{\rho}=i:\widetilde{M}\hookrightarrow \SS^n$ is now the inclusion.
 Then \emph{$\dev_{\sigma}\in\devc(M)$ is close to $\dev_{\rho}$ in the strong geometric topology}
 means: {\edit the restrictions of $\dev_{\sigma}$ and $i$} are close in $C^{\infty}_w(\edit \interior(\widetilde M),\SS^n)$
 and $\Omega_{\sigma}=\image(\dev_{\sigma})$ is close to $\Omega_{\rho}$ in $\domains$.

{\edit There is a  {\edit similar} notion  for the radiant affine manifolds.}
The  radiant affine manifold $N=\Cone\Omega_{\rho}/\Gamma$ is  $\RA$-equivalent to $\xi M_{\rho}$.
The developing map  for $N$
is the inclusion $i=\dev_{\rho}^{\xi}:\Cone_{\edit\rho}\Omega\hookrightarrow\RR^{n+1}$  {\edit and $\dev_{\rho}^{\xi}\in \devGX(N,\RA)$}. 
 \emph{A nearby developing map
  $\dev^{\xi}_{\sigma}\in\devGX(N,\RA)$  in the strong geometric topology}
means:  {\edit the restrictions of $\dev^{\xi}_{\sigma}$ and $i$ are close} in $C^{\infty}_w(\edit\interior(\Cone\Omega_{\rho}),\RR^{n+1})$
 and in addition $\Cone\Omega_{\sigma}$ is close to
 $\Cone\Omega_{\rho}$
 in $\domains$.

Let $\domains'\subset\domains$ be the subspace of open properly convex sets.
For $K\subset V$ define  $\edit\alldomains(K)=\{\Omega\in  \domains'\ :\ K\subset\Cone\Omega\ \}$. The map $\domains'\to\domains'$ given by $\Omega\mapsto\Omega^*$
is continuous.

\begin{lemma}\label{vinbergconv} If $K\subset\R^{n+1}{\setminus 0}$ is compact, then
the function $\overline{\rchi}:\alldomains(K)\to C^{\infty}(K)$ defined by $\overline{\rchi}(\Omega)=\rchi_{_\Omega}|K$ is continuous.\end{lemma}
\begin{proof} Since both topologies are metrizable it suffices to show that the image of a convergent sequence converges.
Suppose the sequence $\Omega_k\in\alldomains(K)$ converges to $\Omega_{\infty}\in\alldomains(K)$, and denote the respective characteristic functions by $\rchi_k$ and  $\rchi_\infty$.
Define the smooth function
 $h:V\times V^*\longrightarrow {\mathbb R}$ by $h(x,\phi)=\exp(-\phi x)$.
Then for $x\in K$,
if $\partial^{\alpha}$ is an $n$'th order mixed partial derivative on $V$, then 
 $\partial^{\alpha} h(x,\phi)=p(\phi)h(x,\phi)$ where $p(\phi)$ 
 is a monomial of degree $n$ in the coordinates of $\phi$. 
   Let $U=\Omega_{\infty}^*\ \Delta\ \Omega_k^*$ be the symmetric difference
   then
  $$|\partial^{\alpha} \charfn_{\infty}(x)-\partial^{\alpha} \charfn_k(x)|\le \int_{{\Cone} U} |p(\phi) h(x,\phi)| d\phi.$$
Since $K\subset{\Cone}(\Omega_k\cap\Omega_{\infty})$ it follows that $\phi(x)>0$ for all $x\in  K$ and $\phi\in \Cone U$. 
Now $p(\phi)$ is polynomial in $\phi$, and $h(x,\phi)$ is exponential in $\phi$, so $p(\phi) h(x,\phi)\to0$ exponentially fast as 
$\phi\to\infty$ in ${\Cone} U$.
It follows that if $U$ is small enough, then 
 $|\partial^{\alpha}\charfn_{\infty}-\partial^{\alpha}\charfn_k|<\epsilon$ on $K$. See (I.3.1) of
 Faraut and Kor{\'a}nyi \cite{FK} for more details.  \end{proof}
  
  {\edit
  It  follows that nearby properly convex manifolds  (without boundary) have nearby characteristic convexity functions:

\begin{lemma}\label{nearbycharfn}  Suppose $\partial M=\emptyset$. The map  $\devGX_c(M)\to C^{\infty}_w(\xi M)$ given by
 $\dev\mapsto c_{\dev}$ is continuous.
 Here, the strong geometric topology is used on  $\devGX_c(M).$\end{lemma}
\begin{proof} If $\dev,\dev'\in \devGX_c(M)$ are close in the strong geometric topology
then $\Omega'=\Im(\dev')$ is close to $\Omega=\Im(\dev)$ in $\domains$. By  (\ref{vinbergconv})
the restrictions to $K$ of $\rchi_{\Omega}$ and $\rchi_{\Omega'}$ are close. Composing with $\log$
shows that
$c_{\Omega}$ and $c_{\Omega'}$ are close on $K$. Thus the characteristic convexity functions 
$c_{\dev}$ and $c_{\dev'}$ are close.
\end{proof}
}

  We wish to give universal bounds on the derivatives of certain real-valued functions
 defined on  radiant affine manifolds of the form $N={\Cone}\Omega/\Gamma$.
  If $M$ is a smooth manifold and $f\in C^{\infty}(M)$ is a smooth function, then
  the $k$-th derivative $D^kf_x$ at $x\in M$ is a symmetric $k$-linear map on the vector space $V=T_xM$  
  (an element of $\Hom(\Sym^k(V),\RR)$).
   Given a norm on $V$ we  get an operator norm $\| D^kf_x\|$ defined as the infimum of $K$ for which
  $| D^kf_x(v_1,\cdots,v_k)|\le K\|v_1\|\cdots\|v_k\|$.  In our case $M={\Cone}\Omega$ is properly convex, 
  and hence a 
  Finsler manifold using the Hilbert metric on $\Cone\Omega$. This gives a norm $\|\cdot\|_{\Cone\Omega}$
   called the {\em Hilbert-Finsler norm}  on the 
  tangent space to ${\Cone}\Omega$. {\edit There is a} corresponding operator norm. The group $\GL(\Cone\Omega)$
  acts by isometries of this norm,  {\edit which therefore} pushes down to {\edit a norm on}
   the tangent space  of
    $N=\Cone\Omega/\Gamma$.
  
Given a point $x\in{\Cone}\Omega$ there is a {\em Benz{\'e}cri chart} $\tau$ for ${\Cone}\Omega$ (see \ref{benzthm}) 
centered on $x$. This chart determines a Euclidean metric $d_{E}$ on ${\Cone}\Omega$, and there
is also the Hilbert metric $d_H=d_{{\Cone}\Omega}$.
There is a constant $K >0$ depending only on dimension such that
in the ball of $d_H$-radius $1$ around $x$ we have
$K^{-1}\cdot d_{E}\le d_{H}\le K\cdot d_{E}$.
 
 It follows that universal bounds on operator norms  using the 
 Hilbert metric give bounds in the Euclidean metric for Benz{\'e}cri coordinates, and vice-versa. Thus we
 may  regard these universal bounds as bounds 
 on ordinary partial derivatives of functions
 defined in a small neighborhood of the origin in ${\mathbb R}^n$ by means of  Benz{\'e}cri
 coordinates.
 We now use Benz{\'e}cri's compactness theorem (\ref{benzcpct}) {\edit with (\ref{vinbergconv})}
  to provide uniform bounds on various properties
 of characteristic functions. 
 
 The restriction of the Hessian metric $D^2c$ to {\edit the characteristic hypersurface $\vinbergsurface=\vinbergsurface_{\Omega}$}
is  a Riemannian metric that is preserved by $\SL_{\pm}(\Cone\Omega)$. 
If $M=\Omega/\Gamma$ is a properly convex manifold,  then radial projection gives a natural identification
${\edit \widetilde M}\equiv\vinbergsurface$ and this puts a Riemannian metric on $M$ called the {\em induced metric}.  
The following seems to be folklore:

\begin{corollary}[bounded curvature]\label{boundedcurvature} For each dimension $n>0$ there is $k_n>0$ such that if
$M$ is a properly convex projective manifold of dimension $n$, then all sectional curvatures $\kappa$ of the induced
metric on $M$   satisfy
$|\kappa|< k_n$. Moreover the induced  metric is $k_n$-biLipschitz equivalent to the Hilbert metric, and is therefore complete.
\end{corollary}
\begin{proof} 
If the  first assertion is false there is a sequence $M_k=\Omega_k/\Gamma_k$ and a point $x_k\in M_k$
and a sectional curvature $\kappa>k$ at $x_k$. By Benz{\'e}cri  compactness (\ref{benzcpct}) we may assume these domains are in Benzecri position (\ref{benzthm}) with $x_k=0$ and
$\Omega_k\to\Omega_{\infty}$. The sectional curvature is given by {\edit a function that is a} formula involving 
various partial derivatives of {\edit the flow function} $c$.
By  (\ref{vinbergconv}) these {\edit functions} converge to some (finite) sectional curvature for $M_{\infty}$, a contradiction.
This also proves the biLipschitz result.
\end{proof}
 
  \begin{lemma}[uniform Hessian-convexity]\label{uniformconvexity}
    For each dimension $n$ there is 
   ${\edit \kappa=\kappa(n)>0}$ with the following property.   
   Suppose $\Omega\subset{\mathbb R}P^n$ is open and properly convex   
    and $c:{\Cone}\Omega\longrightarrow{\mathbb R}$ is the characteristic convexity function. Then
   $D^2\flowfn\ge \kappa \|\cdot\|_{\Cone\Omega}^2$ everywhere.
   \end{lemma}
  \begin{proof} Since $\flowfn$ is preserved by each element of $\GL({\Cone}\Omega)$ 
  up to adding a constant, it suffices
  to show there is $\kappa$ such that the result holds at the center
   of every  Benz{\'e}cri domain $\Omega=\SS^n\cap\Cone\Omega$. The set of all such domains is compact (\ref{benzcpct}),
   and by (\ref{vinbergconv}) the characteristic function varies smoothly with the Benz{\'e}cri domain, 
   {\edit so} the result follows.
     \end{proof}

If $f:(-\epsilon,\epsilon)\to \Cone\Omega$ is an arc parameterized by arc length
then $(c\circ f)''$ is a {\em second directional derivative}. The conclusion can be rephrased
as    $(c\circ f)''\ge\kappa$ for every second directional derivative. We will abuse notation
and write this as $c''\ge\kappa$.

 Suppose $B$ is  a properly convex submanifold
 of a properly convex manifold $M$, both without boundary{\edit,}  then $\xi B\subset\xi M$.
  The next result says that far inside $B$ 
  the  characteristic convexity functions for $B$ and $M$
 are almost equal.

 \begin{lemma}[convexity functions on submanifolds]\label{closecharfn} 
  Given $\epsilon>0$  and a dimension $n,$ there is $R=R(\epsilon,n)>0$ with the following property. 
  Suppose $B\subset M$   are properly convex $n$--manifolds with
  characteristic convexity functions $c_B$ and $c_M$.
  Let $U\subset B$ be the subset of all $x$  with $d_{M}(x,M\setminus B)>R$
   and define
   $g=c_M-c_B:\xi U\to\RR$.
Then  $\| D^k g\|<\epsilon$   for $0\le k\le 2$.
  \end{lemma}
  
  \begin{proof}  Let $\Omega_U\subset\Omega_B\subset \Omega_M\subset\SS^n$ be images
  of the developing maps of $U\subset B\subset M$
  respectively. 
  Since $ g$ is constant along rays from the origin in ${\Cone}\Omega_U$ 
  it suffices to show the bounds hold for $x\in\Omega_U:=\SS^n\cap\Cone\Omega_U$.
  Choose a Benz{\'e}cri chart for $\Omega_M$ 
  centered on $x$. In this chart the Euclidean distance between
   $\partial\Omega_M$
    and $\partial\Omega_B$  is bounded above by a function $f(R)$ independent of $\Omega_M$ and $\Omega_B$  
    and $f(R)\to 0$ as $R\to \infty$.
  The result
now  follows {\edit using (\ref{benzcpct}) and (\ref{vinbergconv})}.
    \end{proof}

\section{Deforming Properly Convex Manifolds rel ends}\label{sec:reldeform}

In this section we prove a version of
(\ref{relopen}) 
for convex manifolds. We
show that the only obstruction to deforming a properly convex manifold is whether
 the ends 
have such a deformation.
Suppose $M_{\rho}$ is a properly convex manifold with holonomy $\rho$.  
The main result of this section (\ref{kleinianedndatamapopen}) is that
for representations $\sigma$ sufficiently close to $\rho$, if the ends of $M_{\rho}$ 
can be deformed to properly convex manifolds with holonomy the restriction of $\sigma$,
then these deformations can be extended to all of $M_{\rho}$ to give
 a properly convex structure  $M_{\sigma}$. 

\begin{definition} A Finsler manifold $M=\Omega/\Gamma$ has {\em controlled ends}
if there is a 
smooth proper function, called an {\em exhaustion function},
$f:M\to[0,\infty)$ and $K>0$ such that $\|Df\|, \|D^2f\|<K$ in the Finsler norm.  \end{definition}

For example every finite volume complete hyperbolic manifold has controlled ends.
If $C\cong \partial C\times[0,\infty)$ is a horocusp in a hyperbolic manifold $M$ then the horofunction $f(x)=d_M(x,\partial C)$
is an exhaustion function. A similar construction works on a {\em generalized cusp}
(\ref{controlfn}).
 There are complete Riemannian manifolds with no exhaustion function. However:

\begin{proposition}\label{PCcontrol} Every properly convex manifold has controlled ends.
\end{proposition}
\begin{proof} By (\ref{boundedcurvature}) every properly convex manifold admits a complete Riemannian metric that
is biLipschitz equivalent to the Hilbert metric and which has
 bounded sectional curvature.
 It is a result of Schoen and Yau \cite{SY} (see also Tam \cite{TAM} and Proposition 26.49 in Chow {\em et al.} \cite{CHOW})
 that a complete Riemannian manifold of bounded sectional curvature
has a proper function with bounded gradient and Hessian. \end{proof}

\begin{definition} A {\em localization function} on a Finsler manifold $M$ is a smooth function $\lambda:M\to[0,1]$
with compact support and $\| D\lambda\|,\| D^2\lambda\|\le 1$.
\end{definition}

\begin{corollary}\label{transitionfn} If $M$ is a properly convex manifold and $ X\subset M$
 is compact, then there is a localization
function $\lambda$ on $M$ with $\lambda( X)=1$.
\end{corollary}
 \begin{proof} By (\ref{PCcontrol}) there is an exhaustion function $f:M\rightarrow[0,\infty)$.
By multiplying $f$ by a small positive scalar we may assume $\|Df\|, \|D^2f\|<1$ and that 
 $f(X)\subset[0,1]$. Let $\beta:{\edit [0,\infty)}\rightarrow[0,1]$ be a  smooth  
function with compact support and
$\beta({\edit t})=1$ for all ${\edit t}\le 1${\edit,} and $\|\beta'(t)\|,\|\beta''(t)\|\le 1/10$ for all $t$. Then $\lambda=\beta\circ f$
has compact support and 
$\lambda(X)=1$.
By the chain rule $\| D\lambda\|,\| D^2\lambda\|\le 1$.
\end{proof}

    Suppose $M=A\cup {\mathcal B}$ is a connected  $n$--manifold and  $A$ is a compact
submanifold with $\partial A=\partial M\sqcup \partial {\mathcal B}$ and $\mathcal B$ has $k$ 
components $B_i$ with $1\le i\le k$ such that
 $B_i=\partial{B_i}\times[0,\infty)$.
By (\ref{defrelhol})   there is  a relative holonomy map $$\enddatamap_{\proj}:\devGX(M,\proj)\to\enddata(M,\mathcal B,\proj)$$
The subspace
 $\devGX_c(M,\proj)\subset \devGX(M,\proj)$ consists of the developing maps of properly 
 convex structures for which $\partial M$ is
 strictly{\edit-}convex. 
 The subspace $\enddata_e(M,\mathcal B,\proj)\subset\enddata(M,\mathcal B,\proj)$ consists of the data for
 which each $B_i$ is properly convex{\edit, and $\partial B_i$ is strictly{\edit-}convex}. Then 
 $\devGX_e(M,\proj)=\enddatamap_{\proj}^{-1}\enddata_ e(M,\mathcal B,\proj)$
 consists of developing maps for which these ends are properly convex with strictly{\edit-}convex boundary. Finally
 $\devGX_{ce}(M,\proj)=\devGX_{c}(M,\proj)\cap \devGX_{e}(M,\proj)$
 is the subspace of developing maps for properly convex structures on $M${\edit,} with $\partial M$ strictly{\edit-}convex{\edit,}
 and for which  each $B_i$ is properly convex{\edit, and $\partial B_i$ is strictly{\edit-}convex}.
 {\edit The following is well known for manifolds of negative sectional curvature, cf (2.3) in Baker and Cooper \cite{BC1}.
 \begin{lemma}\label{convexsubmfdinjective} Suppose $M$ is a properly convex real projective manifold, possibly with boundary,
  and $B\subset M$ is a properly convex submanifold.
Then $B$ is $\pi_1$-injective in $M$.
\end{lemma}
\begin{proof} The holonomy for $B$ is injective because $B$ is properly convex. The holonomy
for $B$ factors through the holonomy for $M$, therefore the map induced by inclusion 
$\pi_1B\to\pi_1M$ is injective.
\end{proof}
} 

   \begin{theorem}
    \label{extenddef}
$\enddatamap_{_\proj}:\devGX_{ce}(M,\proj)\to\enddata_e(M,\mathcal B,\proj)$ is   open using
the geometric topology on the domain and the strong geometric topology on the codomain. \end{theorem}

\begin{proof}  Initially assume $M$ has no boundary.  Given a developing map in $\devGX_{ce}(M,\proj)$
the first step is to show a nearby relative holonomy
is given by a nearby (possibly not convex) projective structure that has the given end data. Then this structure is shown to be properly convex.

{\edit By (\ref{convexsubmfdinjective}) each component of $\mathcal B$ is $\pi_1$-injective in $M$. It then follows
from
(\ref{relopen}) that} $\enddatamap_{_\proj}:\devGX(M,\proj)\to\enddata(M,\mathcal B,\proj)$ is open
using
the geometric topologies in domain and codomain.
Hence the restriction $\enddatamap_{_\proj}:\devGX_e(M,\proj)\to\enddata_e(M, \mathcal B,\proj)$ is also open with these topologies. 
Thus it is open using the strong geometric topology (which is finer than the geometric topology) 
on the codomain and the geometric topology on the domain.
The {\em end geometric topology} on $\devGX_e(M,\proj)$ is {\edit defined to be} the smallest refinement of the geometric topology
such that $\enddatamap_{_\proj}$ is continuous. Then $\enddatamap_{_\proj}$ is open and continuous
with the end geometric topology on the domain and the strong geometric topology on the codomain.  This completes
the first step.

As usual we will assume that ${\mathcal B}=B$
 is connected.  It suffices to show that $\devGX_{ce}(M,\proj)$ is open in $\devGX_e(M,\proj)$ with respect to the end geometric
topology. A neighborhood $\mathcal U\subset \devGX_e(M,\proj)$ of $\dev_{\rho}$ in this topology consists of all developing maps $\dev_{\sigma}$ that are nearby in 
$C^{\infty}_w(\widetilde{M},\RPn)$ and in addition have the property that $\dev_{\sigma}(\widetilde B)$ is close  to $\dev_{\rho}(\widetilde B)$ in 
$\domains$.

 Suppose $\dev_{\rho}\in\devGX_{ce}(M,\proj)$ has holonomy $\rho$ and $\dev_{\sigma}\in\mathcal U$
 has holonomy $\sigma$.   The corresponding projective structures (charts)
 on $M$ are  denoted by $M_{\rho}$ and $M_{\sigma}$. We must show $\dev_{\sigma}\in\devGX_{c}(M,\proj)$.
  To do this we construct a complete convexity function on the tautological bundle $\xi M_{\sigma}$.
  It then follows that $M_{\sigma}$ is properly convex by (\ref{RFhessianisconvex}).

   In this sketch various manifolds should be replaced by
   the corresponding tautological line bundles, but for ease of exposition we do not do this. 
    There are convexity functions for $M_{\rho}$ and $B_{\sigma}$.
 If $\rho$ is close to $\sigma$ then there is a diffeomorphism $M_{\rho}\rightarrow M_{\sigma}$ 
  that is close to the identity  over a large compact set $K$ whose complement is far out in the cusp $B$.
  The convexity function on $M_{\sigma}$ is obtained by using the one for $M_{\rho}$ over most of $K$,
  and the one for $B_{\sigma}$ outside $K$. We slowly transition from one function to
  the other over $\partial K\times[0,1]$ using a localization function to give
   a convex combination that changes in the $[0,1]$ direction.
    This ends the sketch.

We will use $M_{\rho}$ as a basepoint for $\devc(M)$ as in (\ref{GXbasepoint}), see also (\ref{strongtopdef}).
Thus we replace $M$ by $M_{\rho}$ and will usually omit the subscript $\rho$. Then 
$\widetilde M=\Omega_{\rho}\subset \SS^n$ and
$\dev_{\edit\rho}:\Omega_{\rho}\hookrightarrow\SS^n$  is the inclusion map. Similarly we use $\xi M:=\Cone\Omega_{\rho}/\Gamma$
as a basepoint for $\devGX(\xi M_{\rho},\RA)$ and write this as
$\devGX(\xi M,\RA)$.  Then $\xi\widetilde M_{\rho}=\Cone\Omega_{\rho}\subset\RR^{n+1}$.

We use the Hilbert-Finsler {\edit norm $\|\cdot\|$} on $\xi M$ to calculate operator norms. Recall
$\xi_1M=\xi M/\Phi_1$ is the {\edit tautological circle bundle} and
has an infinite cyclic cover $\xi M$. 
 Let $\kappa=\kappa(\dim(M))>0$ be the lower bound on the Hessian of characteristic functions
 given by (\ref{uniformconvexity}) and $\epsilon=\kappa/10$.
 Let  $R=R(\epsilon,\dim(\xi M))$ be the constant given by  (\ref{closecharfn})
 and $K\subset M$  a compact connected submanifold such that $\xi_1 K$ contains the 
 $R$-neighborhood
  of $\xi_1 A$ in $\xi_1M$.    {\edit Hence
  the characteristic functions $c_{\rho,B}$ and $c_{\rho,M}$ are
    $\epsilon$--close in $C^2(\xi (M\setminus K))$.}

     By (\ref{transitionfn}) there is a localization function $\lambda:\xi_1M \to[0,1]$ 
 with $\lambda(\xi_1 K)=1$ that has support inside a compact connected 
 submanifold $\xi_1 L$. 
 Define $J=\cl(L\setminus K)$. Then every point in $\xi_1 J$ is distance at least 
 $R$ from $\partial (\xi_1B)$.
 All these submanifolds depend on the choice of $\epsilon$. Let $\widetilde{\lambda}:\xi M\to[0,1]$
 be the function that covers $\lambda$. We  abuse notation by writing $\widetilde\lambda$ as $\lambda$. Observe that $\lambda^{-1}(0,1)\subset \xi J$.

{\bf Claim 1.} \emph{There is a convexity function  $c:\xi M\to \R$ 
which equals $c_{\rho,M}$ on $\xi K$ and equals $c_{\rho,B}$ on $\xi(M\setminus L)$ and
$D^2c\ge(\kappa/2)\|\cdot\|^2$.}

\begin{proof}[Proof of Claim 1]
First blend
$c_{\rho,M}$ and $c_{\rho,B}$  inside $\xi J$ using $\lambda$ to get $f:\xi M\to\RR$ given by 
 $$f = \lambda\cdot c_{\rho,M}+(1-\lambda)\cdot c_{\rho,B}= 
 c_{\rho,M}+(1-\lambda)\cdot g$$
 where $g=c_{\rho,B}-c_{\rho,M}$.
The map $f$ is well defined even though $c_{\rho,B}$ is only defined on $\xi B$ because $(1-\lambda)=0$
outside $\xi B$. 

Subclaim:  $D^2f\ge (\kappa/2)\|\cdot\|^2$.  
Outside $\xi J$ this follows from (\ref{uniformconvexity})   since {\edit$f=c_{\rho,M}$ on $\xi K$,
and $f=c_{\rho,B}$} on $\xi (M\setminus L)$.
On $\xi J$ we  show this using directional derivatives. 
 By the product rule $$f''=c''_{\rho,M} +g''- (\lambda'' g + 2\lambda' g'+\lambda g'').$$
Since $M_{\rho}$ is properly convex $c''_{\rho,M}\ge \kappa$ by (\ref{uniformconvexity}).
Also  $|\lambda|,|\lambda'|,|\lambda''|\le 1$ because  $\lambda$ is a localization function
 and  $|g|, |g'|, |g''|<\epsilon=\kappa/10$ on $\xi J$ by definition of $R$ and $K${\edit,}
 so $$|\ g''- (\lambda'' g + 2\lambda' g'+\lambda g'')\ |\le 5\epsilon=\kappa/2.$$

Thus $f''\ge \kappa/2$ which proves the subclaim. 
The level set $S=f^{-1}(0)$ is  Hessian-convex in the
backward direction of the flow and is the $0$-set of a unique flow function $c$ which coincides
with $c_{\rho,B}$ outside $\xi L$.
It follows from (\ref{backwardconvex})  that $c''\ge \kappa/2$ also. This
proves claim 1.
\end{proof}

To avoid a proliferation of notation, and because what we are about to do is similar to
what we just did, we reuse notation as follows. 
We define the new $K$ to be the old $L${\edit,} and the new $\lambda$ is a localization function on $\xi_1 M$
with $\lambda(\xi_1K)=1${\edit,} and the new $L\subset M$ is a compact connected manifold
so that $\xi_1 L$ contains the support of $\lambda$. Then redefine $J=\cl(L\setminus K)$.
Let $E:=\cl(M\setminus K)\subset B$.
Again we write the lift as $\lambda:\xi M\to\RR$.  There are characteristic convexity functions
$c_{\rho,B}:\xi B_{\rho}\to\RR$ and $c_{\sigma,B}:\xi B_{\sigma}\to\RR$.

Since $\xi_1 L_{\rho}$ is compact, 
if $\mathcal U$ is small there is  a diffeomorphism $H:\xi_1 M_{\rho}\to\xi_1 M_{\sigma}$
such that $H|\xi_1 L_{\rho}$ is  close in $C^{\infty}$ to the identity in the following sense. The map $H$ is covered
by $\widetilde H:\xi \widetilde M_{\rho}\to\xi  \widetilde M_{\sigma}$, and the restriction of $\widetilde H$
 is  close to the
inclusion $\xi \widetilde L_{\rho}\hookrightarrow\RR^{n+1}$
in $C^{\infty}_w(\xi\widetilde L_{\rho},\RR^{n+1})$. The map $\widetilde H$ also covers {\edit a map}
$h:\xi M_{\rho}\to\xi M_{\sigma}$.

Set $g=(c_{\sigma,B})\circ h-c_{\rho,B}:\xi E_{\rho}\to\RR$. {\edit If $\Ucal$ is small enough
then by (\ref{nearbycharfn}) $c_{\sigma,B}$ and $c_{\rho,B}$ are close on $\xi J_{\rho}$. Since $\widetilde H$ is close to the identity map, and covers $h$, if follows that}
$\|D^k g \| <\epsilon$ for $k\in\{0,1,2\}$ everywhere on $\xi_1J_{\rho}$.
Define
 $f:\xi M_{\rho}\to\R$ by 
 $$f = \lambda\cdot c+(1-\lambda)\cdot (c_{\sigma,B})\circ h.$$
As before this is well defined.  {\bf Claim 2.} \emph{$f''\ge \kappa/2$ on $\xi L_{\rho}$.}

\begin{proof}[Proof of Claim 2]
When $\lambda=1$ then $f''=c''\ge\kappa/2$ by claim 1. 
The set where $\lambda < 1$ is contained in $\xi J_{\rho}$.
On $\xi J_{\rho}$ and $c=c_{\rho,B}$  so
$$f=\lambda\cdot c_{\rho,B}+(1-\lambda)\cdot(c_{\sigma,B})\circ h= c_{\rho,B}+(1-\lambda)\cdot g$$
and $$f''=c_{\rho,B}'' +g''- (\lambda''\cdot g + 2\lambda' g'+g'').$$
Then $c_{\rho,B}''\ge \kappa/2$ by (\ref{uniformconvexity}). As before $|\lambda|,|\lambda'|,|\lambda''|\le 1$ 
and by the above $|g|,|g'|,|g''|<\epsilon$. Since $\epsilon<\kappa/10$ this proves claim 2.
\end{proof}

Since $\widetilde H$ is  close to the inclusion in $C^{\infty}_w(\xi \widetilde L_{\rho},\RR^{n+1})$ it follows that
 $f\circ h^{-1}$ is Hessian-convex on  $\xi L_{\sigma}$. 
Outside this set $f\circ h^{-1}=c_{\sigma,B}$ which is Hessian-convex. This proves $f:\xi M_{\sigma}\to\RR$ is
Hessian-convex everywhere. 

Again it follows from  (\ref{backwardconvex}) that there is a Hessian-convex
flow function $c_{\sigma}:\xi M_{\sigma}\to\RR$ defined by $f\circ h^{-1}$.
 The corresponding Hessian metric on $\xi_1M_{\sigma}$ is complete because 
$\xi_1 L_{\sigma}$ is compact so the metric is complete on $\xi_1 L_{\sigma}$, and 
outside $\xi_1 L_{\sigma}$ it is the complete metric given by the properly convex end $\xi_1B_{\sigma}$.
It follows that the Hessian metric on $\xi M_{\sigma}$ is also complete.
 This completes
the proof when $M$ has no boundary.

Now suppose $M$ has (compact) boundary and set $P=\interior(M)$.
Then $P$ is properly convex with  a characteristic convexity function $c:\xi P_{\rho}\to\RR$. 
 The idea is to shrink $M$ a bit to obtain a submanifold $N\subset P$ with Hessian-convex boundary.
The restriction to $N$ of the convexity function for $P$ can be used in the above arguments. It is a complete
metric with $\partial N$ at a finite distance.

By (\ref{smoothbdry}) there is a submanifold $N\subset M$
with Hessian-convex{\edit,} compact boundary such that $\cl(M\setminus N)$ is a collar of $\partial M$.
 The restriction of $c$ to $\xi N$ is a complete convexity function.
There is a diffeomorphism $F:\xi M\to \xi N$ close to the identity in $C^2$
 that is the identity outside a small collar of $\partial({\edit\xi}M)$.
Then $c_{\rho,M}:=(c|_{\xi N})\circ F:\xi M\to \R$ is a complete convexity function. The
 pullback of the restriction to $\xi N$ of the Hilbert metric on $\xi P$ 
is a complete metric on $\xi M$. The proof now proceeds as above to construct
a complete convexity function on $\xi M_{\sigma}$.
\end{proof}

To apply (\ref{extenddef}) involves finding deformations of the cusps that are nearby in the
strong geometric topology. This involves finding a diffeomorphism from the original cusp
to the deformed cusp that is close to projective. To make this task easier we 
show such a map exists for a small deformation of the holonomy
if the deformed domain is close to the original domain. 

The {\em projective Kleinian group space} for a smooth manifold $M$ is
 $$\Kleinian(M)=\{(\Omega,\rho)\ \in \domains\times\Rep(M):\ M\ {\rm diffeomorphic\ to\ } \Omega/\rho(\pi_1M)\}$$
with topology given by the subspace topology
  of the product topology on $\domains\times\Rep(M)$. 
  This topology is given by a metric. {\edit If $(\Omega,\rho)\in\Kleinian(M)$ then 
  $\Omega/\rho(\pi_1M)$ is a properly convex manifold with strictly{\edit-}convex boundary.
   If $M$ is closed then $\rho$ determines $\Omega$.}
   There is a natural map
 $$\Kleinian:\devGX_c(M,\proj)\to\Kleinian(M)$$ 
 given by $\Kleinian(\dev)=(\dev(\widetilde{M}),\hol(\dev))$. 
 
\begin{proposition}\label{kleinianendopen} Suppose $M\cong
\partial M\times[0,\infty)$ is a connected smooth manifold and $\partial M$ is
compact. 
{\edit  Then 
 $\Kleinian$ is a continuous open map for the {\em strong} geometric topology on $\devGX_c(M,\proj)$.}
\end{proposition}
 \begin{proof} {\edit Continuity is obvious.} Suppose $\dev_{\rho}\in\devGX_c(M,\proj)$ and $\Kleinian(\dev_{\rho})=(\Omega_{\rho},\rho)$
 and that $(\Omega_{\sigma},\sigma)\in\Kleinian(M)$ is  close. Then $Q=\Omega_{\sigma}/\sigma(\pi_1M)$
 is a properly convex manifold. We identify $M\equiv\Omega_{\rho}/\rho(\pi_1M)$.
 It suffices to show there is a diffeomorphism $M\to Q$ which is almost a projective map
 between large compact sets in  the interiors.
   
{\edit By (\ref{smoothbdry})}  there is a diffeomorphism $M\cong\partial M\times[0,\infty)$ so
that $\partial M\times t$ is $\dev_{\rho,M}$-Hessian-convex for all $0<t\le 1$.
For $k>1$ define  
$N=\partial M\times [1/k,k]$ and $W=\partial M\times[0,k+1]$ and ${\edit E}=\partial M\times 1/k$.  
{\edit These are all compact}.
  By (\ref{deformGX}) there is a $\dev_{\sigma,W}\in\devGX(W,\proj)$
  with  holonomy $\sigma$  that is  close to $\dev_{\rho,M}{\edit|W}$ over a compact set in $\widetilde W$
  that covers $N$. 
  
  By (\ref{isotopsubstructure}) we may change $\dev_{\sigma,W}$ by a small isotopy so that
there is a projective embedding $f:N\to Q$.
 If $\sigma$ is close enough to $\rho$  then, since  {\edit the hypersurface} ${\edit E}$ 
 is Hessian-convex for $\dev_{\rho,M},$ it follows that
 ${\edit E}$ is also Hessian-convex for $\dev_{\sigma,W}$. Hence $f({\edit E})$ is Hessian-convex in $Q$.

 Let $P$ be the closure of the component of $Q\setminus f(N)$ that contains $\partial Q$. Since
 $\partial M$ is compact, for homology reasons $f({\edit E})$ separates $\partial Q$ from the end
 of $Q$, thus $P$ is compact
 and $\partial P=\partial Q\sqcup f({\edit E})$. 
{\bf Claim:} $P$ is diffeomorphic to ${\edit E}\times I$. 
   
  {\edit Since $E$ is $\dev_{\sigma,N}$-Hessian-convex there is a nearest point retraction (using the Hilbert metric on $Q$) $r:P\to {\edit E}$ with fibers that are lines{\edit,}
   and this
 gives a homeomorphism $P\to {\edit E}\times I$.  
 By \cite{WH} smooth manifolds are PL.
  The {\em $M\times I$ theorem}, Hirsch and Mazur \cite{HM}, says that if $M$ is a PL manifold, then every smoothing of $M\times I$ is diffeomorphic to a product.
Thus  $P$ is diffeomorphic to $E\times I$,
  which proves the claim.}

It follows that $P$ is a collar of $\partial Q$ so $R=P\cup f(N)\cong {\edit E}\times[0,k]$ is also
a collar of $\partial Q$.
Thus $Q'=\cl(Q\setminus R)$ is diffeomorphic to ${\edit E}\times[k,\infty)$. 
 Clearly $P$ lies in
 a small neighborhood of $\partial Q$.
 We can now extend $f$ to a diffeomorphism $f:M\to Q$
 by sending $\partial M\times[0,1/k]$ to $P$ and $\partial M\times[k,\infty)$
 to $Q'$.  This is close to a projective map on $N$. Define $\dev_{\sigma,M}:\widetilde{M}\to\RPn$
 by $\dev_{\sigma,M}=\dev_{\sigma,Q}\circ \widetilde f$. Since $f$ is close to projective over $N$ 
 it follows that  $\dev_{\sigma,M}$ is close to $\dev_{\rho,M}$.
  \end{proof}

Suppose $M=A\cup\mathcal B$ is a smooth manifold with (possibly empty) boundary and
$A$ is a compact submanifold of $M$ with $\partial A=\partial M\sqcup\partial\mathcal B$. {\edit Suppose} 
$\mathcal B=B_1\sqcup\cdots \sqcup B_k$ has $k$ connected components, and $B_i\cong\partial B_i\times[0,\infty)$.
Define the {\em Kleinian relative-holonomy space}
\begin{equation}\label{relkleinholdef}
\enddata(M,\mathcal B,\Kleinian)\subset \Rep(\pi_1M)\times\prod_{i=1}^k \Kleinian(B_i)\end{equation} to be the subset of all
 $(\rho,(\Omega_1,\rho_1),\cdots,(\Omega_k,\rho_k))$ such that $\rho_i=\rho|\pi_1B_i$. This space
 has the subspace topology of the product topology.

For each $B_i\subset M$ we fix a choice of some component $\widetilde{B}_i\subset\widetilde{M}$
of the preimage $B_i$ in 
the universal cover of $M$.
Then $\Omega_i=\dev(\widetilde{B}_i)$ and $\Gamma_i=\hol(\pi_1B_i)$
gives a point in $\Kleinian(B_i)$. This defines the {\em Kleinian relative holonomy map} 
$$\enddatamap_{\Kleinian}:\devGX_{ce}(M,\proj)\longrightarrow \enddata(M,\mathcal B,\Kleinian)$$

\begin{theorem}[Convex Extension Theorem]\label{kleinianedndatamapopen}
$\enddatamap_{\Kleinian}:\devGX_{ce}(M,\proj)\to\enddata(M,\mathcal B,\Kleinian)$ is 
open using
the  geometric topology on the domain.
\end{theorem}
\begin{proof}  Follows immediately from (\ref{extenddef}) and (\ref{kleinianendopen}).\end{proof}

\begin{proof}[Proof of (\ref{extendpc})] The map $\gamma:(-1,1)\to\enddata(M,\mathcal B,\Kleinian)$
defined by
$$\gamma(t)=(\rho_t,(\Omega_1(t),\rho_t|_{\pi_1B_1}),\cdots,(\Omega_k(t),\rho_t|_{\pi_1B_k}){\edit)}$$
{\edit is continuous by hypothesis (5) of (\ref{extendpc}).}
{\edit By (\ref{kleinianedndatamapopen})}  $\enddatamap_{\Kleinian}$ is open{\edit,} and $\gamma(0)\in\image(\enddatamap_{\Kleinian})${\edit, thus}
for some $\epsilon>0$ that ${\edit\gamma}(-\epsilon,\epsilon)\subset\image(\enddatamap_{\Kleinian})$.
So for $|t|<\epsilon$ there is $\dev_t\in\devGX_{ce}(M,\proj)$ with $\enddatamap_{\Kleinian}(\dev_t)=\gamma(t)$.
Define $M_t$ to be the projective structure on $M$ defined by $\dev_t$. Then $M_t$ is properly
convex{\edit, with holonomy $\rho_t$, and $\partial M_t$ is strictly-convex}. {\edit  Moreover the projective structure
on $M_t$ restricted to $B_i$ 
is diffeomorphic to} $P_i(t)$ by definition of $\enddatamap_{\Kleinian}$.
\end{proof}

\section{Generalized Cusps}\label{gencusps}
 A generalized cusp is a certain kind of properly convex projective manifold. 
 The main result of this section is that
holonomies of generalized cusps with fixed topology  form an open
 subset in a certain semi-algebraic set (\ref{cuspstab}). 
 This follows from the fact that {\em a generalized cusp contains
a homogeneous cusp} (\ref{sttdcusp}). We then prove the main theorem (\ref{mainthm}). 

A cusp in a {\em hyperbolic manifold} viewed as a projective manifold is characterized by being projectively
equivalent to an affine manifold that
has a foliation by strictly{\edit-}convex hypersurfaces that are images of horospheres,
together with  a transverse foliation by parallel lines. This characterization does not work in general. 
Consider the affine manifold $M=U/\Gamma\cong T^2\times[0,\infty),$ where 
$$U=\{(x_1,x_2,x_3):  x_3\ge x_1^2+x_2^2>0\ \}$$
and $\Gamma$ is the cyclic group generated by $(x_1,x_2,x_3)\mapsto (2x_1,2x_2,4x_3)$. 
 It has a foliation by  tori that are the images of the strictly{\edit-}convex
  hypersurfaces $z=K(x^2+y^2)$ for $K\ge 1,$ and it has a transverse foliation by
  vertical lines.  However $M$ is {\em not convex}.

\begin{definition}\label{gencusp} A {\em generalized cusp} is a properly convex manifold
 $C=\Omega/\Gamma$ homeomorphic
to $\partial C\times[0,\infty)$ with $\partial C$ a closed manifold and $\pi_1C$ virtually nilpotent
 such that $\partial\Omega$ contains no line segment,  i.e.\thinspace $\partial C$ is strictly{\edit-}convex.
The group $\Gamma$ is called a {\em generalized cusp group}. 

A {\em quasi-cusp}  is a properly convex manifold  with interior
homeomorphic
to $Q\times\RR${\edit,}  and $Q$  is a closed manifold{\edit,} and 
$\pi_1Q$ {\edit is} virtually nilpotent.

\end{definition}

If $\Gamma$ contains no hyperbolics,  then $C$ is called a {\em cusp} and $\Gamma$ is
conjugate to a subgroup of $PO(n,1)$ by Theorem (0.5) in \cite{CLT1}. An example of a quasi-cusp is 
$\Delta/\Gamma$ for any
discrete subgroup $\Gamma\cong\ZZ^{n-1}$ of the diagonal group in $SL(n+1,\RR),$
where $\Delta\subset\RPn$ is the interior of an $n$-simplex that is preserved by $\Gamma$.

\begin{definition} A generalized cusp $\Omega/\Gamma$ is {\em homogeneous} if $\PGL(\Omega)$ acts transitively on $\partial\Omega$. The group $\PGL(\Omega)$ is called
a {\em (generalized) cusp Lie group}.
 \end{definition}
 For example a cusp in a hyperbolic manifold is homogeneous if and only if it is
the quotient of a  horoball $\Omega\subset\mathbb H^n$. 
 In this case  $\PGL(\Omega)$ is conjugate to the 
  subgroup of $PO(n,1)\cong Isom({\mathbb H}^n)$ that
 fixes one point at infinity.  Cusp Lie groups for 3--manifolds are listed in section \ref{3mfd}.

 \begin{theorem}\label{sttdcusp} Every generalized cusp 
contains a homogeneous generalized cusp. 
 \end{theorem}
 
 There is an equivalence relation on generalized cusps  generated by
the property that one cusp can be projectively embedded in another. Equivalent cusps have conjugate holonomy. 
One can shrink a cusp by removing a collar from the boundary. However
sometimes one can remove a submanifold at the other end. For example
there might  be a totally geodesic codimension--1 compact submanifold in the interior
of the cusp, which one could cut along. It simplifies matters to do this ahead of time.

\begin{definition} A generalized 
cusp $C$ is {\em minimal} if, for every cusp $C'\subset C$
 with $\partial C'=\partial C$,  it follows that $C=C'$.\end{definition}

  If $M$ is a convex manifold and $X\subset M$ then the {\em convex hull}, $\CH(X)$, of $X$ is the intersection of all convex
 submanifolds of $M$
that contain $X$.
{\edit Suppose
$f:S^1\times(-1,1]\rightarrow C$ is a diffeomorphism, and  $C$ has a  hyperbolic metric} such that $\gamma=f(S^1\times 0)$ is a geodesic, 
and the distance $d(f(e^{i\theta},t),\gamma)=\edit |t|$. {\edit Thus} $C$ is a hyperbolic annulus with
one convex boundary component, and the other boundary component deleted.
 {\edit Moreover $C$ is a generalized cusp}, and $C'=\CH(\partial C)=f(S^1\times{\edit(}0,1])$ is a minimal cusp.

\begin{lemma}\label{mincusplem} Every generalized cusp contains a unique minimal cusp.
A finite cover of a minimal cusp is minimal.
\end{lemma}
\begin{proof} Suppose $C=\Omega/\Gamma$ is a generalized cusp.
Let $\Omega'$ be the convex hull of $\partial\Omega$. Then $\Omega'\subset\Omega$ is properly
convex and $\Gamma$-invariant and $\partial\Omega'=\partial\Omega$. 
 The cusp $C'=\Omega'/\Gamma$ is the unique minimal cusp contained in $C$.
 If $M$ is a finite cover of $C'$ then $M=\CH(\partial M)$ so $M$ is also minimal.
\end{proof}

 The following will be used frequently
\begin{lemma}\label{Kpione} Suppose $M$ is a quasi-cusp of dimension $n$, 
and $P\subset M$ is a convex submanifold, and $\pi_1P\rightarrow\pi_1M$ is an isomorphism
then
\begin{itemize}
\item[1)] the universal cover $\widetilde M$ is contractible.
\item[2)] $M$ is an Eilenberg-MacLane space: a $K(\pi_1(M),1)$. 
\item[3)]  $H_{n-1}(\pi_1M;\ZZ_2)\cong H_{n-1}(M;\ZZ_2)\cong\ZZ_2$.
\item[4)] If $\dim P<n$  then  $\dim(P)= n-1$ and $P$ is  a closed manifold. 
\item[5)] If $P\subset\interior(M)$ and $\dim(P)<\dim(M)$ then
 $P$ separates.
 \end{itemize}
\end{lemma}
\begin{proof} 1,2 and 3 follow immediately from the definition.
  Since $P$ and $M$ are convex, they are aspherical. Hence  the inclusion $P\hookrightarrow M$
is a homotopy equivalence.  By (3) $H_{n-1}(P;\ZZ_2)\cong\ZZ_2$ from which (4) follows.

 For (5), it follows from (4) that $\dim P=n-1$. 
By definition of quasi-cusp, the interior of $M$ is homeomorphic to $Q\times\RR$ for some closed manifold $Q$.
Since $P$ and $Q$ are both closed, for sufficiently 
large $t\in\RR$ if follows that $P$ separates $Q\times t$ from $Q\times(-t)$.
\end{proof}

{\edit By (\ref{holonomylifts})}  the holonomy of a projective structure lifts to $\GL(n+1,\R)$, and we will 
 use this lift in what follows.
 
 {\edit If $V$ is a finite dimensional vector space of dimension $n$ then a {\em (complete) flag for $V$}
 is a sequence of subspaces $0=V_0<V_1\cdots<V_n=V$ with $\dim(V_i)=i$.}
 The subgroup $\UT(n)<\GL(n,\RR)$ consists of all upper-triangular matrices with
positive diagonal entries. 
A group $\Gamma\subset\GL(n,\RR)$ is conjugate into $\UT(n)$ if and only if $\Gamma$ preserves
a {\edit flag} and every weight of $\Gamma$ is positive.

 A {\em connected} nilpotent subgroup $\Gamma$
 of $\GL(n,{\mathbb C})$ preserves a {\edit flag} {\edit for $\CC^n$}.
  However if $\Gamma$ is not connected this need not be true.
For example the quaternionic group of order $8$ in $\GL(2,{\mathbb C})$ does not preserve
a flag. First we show (\ref{VnilVFG}) that there is
 a finite index subgroup of  $\Gamma$  that preserves a {\edit flag}.
 The index of a subgroup $H<G$ 
is written $|G:H|$.  A subgroup $H\le G$ is {\em characteristic} if every automorphism of $G$
  preserves $H$. {\edit It is routine to show}
 
\begin{lemma}\label{finindex} $\exists\; h(n,k)\edit >0$
such that if the group $G$ is generated by $k$ elements, then there is a characteristic subgroup
$C\le G$ with $|G:C|\le h(n,k)$ such that {\edit if $H\le G$ is any subgroup with index
 $|G:H|\le n$ then $C\le H$.}\end{lemma}

Suppose $V$ is a vector space over  $\CC$. A  {\em weight} of a subgroup $\Gamma\subset \GL(V)$ 
 is a homomorphism ({\em character}) $\lambda:\Gamma\rightarrow \CC^*$ such that the 
 {\em weight space} $E(\lambda)$ and {\em generalized
weight space} $V(\lambda)$ are both non-trivial. Here,
$$\displaystyle E(\lambda) = \bigcap_{\gamma\in\Gamma} (\ker (\gamma-\lambda(\gamma))
\quad\text{and}\quad V(\lambda) = \bigcup_{n>0}\  
{\bigcap}_{\gamma\in\Gamma} (\ker (\gamma-\lambda(\gamma))^n.$$ 
A (generalized) weight space is $\Gamma$
invariant. A one-dimensional weight space is the same thing as a one-dimensional $\Gamma$--invariant subspace.
The vector space $V$ has a  {\em generalized weight decomposition} if
 $V=\oplus V(\lambda),$ where the sum is over all weights.
 
The group $\Gamma$ is {\em polycyclic of (Hirsch) length (at most) $k$} if there is a  subnormal series
$\Gamma=\Gamma_k\vartriangleright \Gamma_{k-1}\cdots\vartriangleright \Gamma_{1}\vartriangleright \Gamma_0=1$
with $\Gamma_{i+1}/\Gamma_{i}$ cyclic  for every $i$. 
A subgroup of a polycyclic
group of length $k$ is polycyclic of length  at most $k$.
Every finitely-generated nilpotent group is polycyclic. 

\begin{lemma}\label{onedim}
$\exists\; c=c(n,k)$ such that if
 $\Gamma< \GL({\mathbb C}^n)$ is polycyclic of length at most $k,$ then 
  there is a characteristic subgroup $C\le\Gamma$ with $|\Gamma:C|\le c$
and
$C$ preserves a one-dimensional subspace of ${\mathbb C}^n$.
\end{lemma}
\begin{proof} We use induction on $n$ and $k$. For $n=1$ the result is obvious.
 For $k=1$ the result follows from Jordan normal form with $c=1$.
Assume the result true for $k$. Suppose $\Gamma$ is polycyclic of length $k+1$.
Then $\Gamma$ contains a normal  polycyclic group $\Gamma_k$ of length $k$
with $\Gamma/\Gamma_k$ cyclic. There is a characteristic 
subgroup $C_k\le \Gamma_k$ of
index at most $c(n,k)$ that preserves a one-dimensional subspace $W$.

There is some weight $\lambda:C_k\longrightarrow\CC^*$ 
with $W$ contained in the weight space
$E=E(\lambda)$.
There are at most $n$   weights for $C_k$. If $\theta$ is an automorphism of $C_k$ then 
$\lambda\circ\theta$
is a weight for $C_k$. Since $C_k$ is a characteristic subgroup of $\Gamma_k$, 
and $\Gamma_k$ is normal in $\Gamma$, it follows that $C_k$ is preserved
 by  all inner automorphisms of $\Gamma$. Thus an inner automorphism of $\Gamma$
  permutes these weights,
 so an element $\gamma\in \Gamma$
induces  a permutation   of the weights with order  $m\le   n!$. Choose $\gamma\in \Gamma$ which
 generates $\Gamma/\Gamma_k$. Then $\gamma^m$ induces the identity permutation.
Hence the subgroup $\Gamma'=\langle C_k,\gamma^m\rangle$ preserves $E$. Applying
Jordan normal form to $\gamma^m|E$ gives a one-dimensional subspace of $E$ that is preserved
by $\gamma^m$. This subspace is also preserved by $C_k$. Then 
$|\Gamma:\Gamma'|\le m|\Gamma_k:C_k|\le n !\cdot c(n,k)$ since $m\le n !$. 
By (\ref{finindex}) there is a characteristic subgroup $C\le \Gamma'\le \Gamma$ with 
$|\Gamma:C|\le c(n,k+1)= h( n!\cdot c(n,k),k+1)$.\end{proof}

\begin{proposition}\label{polyVFG} 
$\exists\; d(n,k)$ such that for all  polycyclic groups $G$ of length at most $k$ 
there is  a characteristic subgroup $C\le G$ with $|G:C|\le d(n,k)$
 such that if
$\rho:G\longrightarrow \GL(n,\CC),$ 
then $\rho(C)$ preserves a {\edit flag} in $\CC^n$.
\end{proposition}
\begin{proof} Below we show by induction on $n$ that, for a fixed $\rho$, there is a subgroup of
index at most  $e(n,k)=\prod_{i=1}^n c(i,k)$ that preserves a {\edit flag}. The result follows 
from (\ref{finindex}) with $d(n,k)=h(e(n,k),k)$.

For $n=1$ the result is clear. By (\ref{onedim}) there is a subgroup
$\Gamma'<\Gamma=\rho(G)$
of index at most $c(n,k)$ that preserves a one-dimensional subspace $E\subset V= \CC^n$.
Then $\Gamma'$ acts on $V/E\cong  \CC^{n-1}$. By induction there is $\Gamma''<\Gamma'$
with $|\Gamma':\Gamma''|\le e(n-1,k)$ that preserves a {\edit flag}  $\mathcal F$ in $V/E$. 
The preimage of $\mathcal F$ in $V$, together with $E$, forms a {\edit flag} for $V$
which is preserved by ${\edit\Gamma}''$. Moreover
$|\Gamma:\Gamma''|=|\Gamma:\Gamma'|\cdot |\Gamma':\Gamma''|\le c(n,k)e(n-1,k)=e(n,k)$.
\end{proof}

  \begin{definition} Suppose $G$ is a  finitely-generated, virtually nilpotent group.
 Let   
$k$ the smallest integer such that $G$ is polycyclic of length $k$. Given ${\edit n}>0$
the {\em ${\edit n}$--core of $G$} is the subgroup, $\core(G,{\edit n})$,  of $G$ that is the intersection of all subgroups of $G$ of index at most $2^{\edit n} \cdot d({\edit n},k)$.  \end{definition}

{\edit Clearly $\core({\edit n},G)$  is a characteristic  subgroup of  finite-index in $G$
 that is contained in every  subgroup of index at most $2^n$ in the subgroup $C\le G$ from
(\ref{polyVFG}).} 

\begin{corollary}\label{VnilVFG} Suppose $G$  is a  finitely-generated, virtually nilpotent group and 
$H=\core(G,{\edit n})$. Then for every homomorphism 
$\rho:G\to \GL({\edit n},\FF)$:
\begin{itemize}
\item[(1)] If $\FF=\CC$, then $\rho(H)$ preserve a {\edit flag} in $\CC^{\edit n}$.
\item[(2)] If $\FF=\RR$ and every weight of $\rho(H)$ is real, then  $\rho(H)$ is 
conjugate into $UT({\edit n})$.
{\edit\item[(3)] If $\FF=\RR$ then $\rho(G)\in \VFG$ if and only if every weight of $\rho(H)$ is real.
\item[(4)] $\VFG(G,n)=\{\rho\in\Hom(G,\GL({\edit n},\RR))\ :\ \rho(G)\in\VFG\ \}$  is  semi-algebraic.}
\end{itemize}
\end{corollary}
\begin{proof} {\edit Let $C$ be the characteristic subgroup of $G$
 given by (\ref{polyVFG}). Then $H\le C$ so
(1) follows.}
(2) follows from (1) as follows. Set $U=\RR^{\edit n}$ and $V=U\otimes \CC$
so $G\subset \GL(U)\subset\GL(V)$.
By (1) $V=\oplus V(\lambda)$ 
where $V(\lambda)=\cap_{{\edit c\in C}}\ker(\rho({\edit c})-\lambda({\edit c}))^{\edit n}$.
Observe that $V(\lambda)\subset\CC^{\edit n}$ is given by linear equations that are defined over $\RR$
 because $\lambda({\edit C})\subset\RR$
and $\rho({\edit C})\subset\GL({\edit n},\RR)$.
Thus $V(\lambda)$ is the complexification of  
$U(\lambda)=\cap_{{\edit c\in C}}\ker(\rho({\edit c})-\lambda({\edit c}))^{\edit n}\subset\RR^{\edit n}$
so $U=\oplus U(\lambda)$. Hence $\rho({\edit C})$ preserves a {\edit flag} in $\RR^{\edit n}$.
By replacing $\edit C$ by a {\edit certain} subgroup, {\edit $C_0$, of index at most $2^n$ 
we may ensure that all real weights
are positive. Since $H\le C_0$ it follows that  $\rho(H)$ is conjugate into $\UT(n)$. Clearly (2) implies (3)}. 
{\edit (4)} follows from (2) and the observation that
 {\edit the condition that every} weight {\edit is} real is defined by the semi-algebraic equations that say every eigenvalue of every element of $\rho(H)$ is real.
\end{proof}

Suppose $U$ is a real vector space and $\Gamma< \GL(U)$ preserves a {\edit flag} in 
$V=U\otimes {\mathbb C}$. Then
combining each weight $\lambda$ for $V$
with the complex-conjugate weight $\overline\lambda$  gives a real invariant subspace $U(\lambda,\overline{\lambda})=(V(\lambda)+V(\overline\lambda))\cap U\subset U$ and
 $U=\bigoplus U(\lambda,\overline{\lambda})$. We call $U( \lambda,\overline{\lambda})$ a {\em conjugate generalized weights space.
For each $\gamma\in\Gamma$ the
 eigenvalues of 
 $\gamma|_{U(\lambda,\overline{\lambda})}$ are $\lambda(\gamma)$ and $\overline\lambda(\gamma)$.

\begin{proposition}\label{gencuspisVFG} 
Suppose  $P=  \Omega/\Gamma$  is a quasi-cusp of dimension $n$. 
Then $\core(\Gamma,n+1)$ is conjugate into $\UT(n+1)$.
In particular $\Gamma\in\VFG$.\end{proposition}
\begin{proof} Write $V={\mathbb R}^{n+1}$ so $\Gamma\subset\PGL(V)$. 
By (\ref{holonomylifts}) we may lift to get 
$\Gamma\subset\GL(V)$. By  (\ref{VnilVFG})(1) we can conjugate so that  $H=\core(\Gamma,n+1)$
is contained in the upper-triangular subgroup in $\GL(n+1,{\mathbb C})$. We replace $\Gamma$  by $H$.
  Then $V=A\oplus B$ where $A$ is the sum of the generalized weight spaces
for real weights and $B=\oplus B_i$ is the sum of the remaining conjugate generalized weights spaces. 
It suffices to show $B=0$, since then by (\ref{VnilVFG})(2) $\Gamma$ is conjugate into $\UT(n+1)$.

 Each vector $x\in V$  is uniquely expressed as a
 linear combination $a+b_1+\cdots + b_k$ with $a\in A$
and $b_i\in B_i$. Define $n(x)$ to be the  number of distinct $i$ with $b_i\ne 0$.
Choose $x\ne 0$ with $[x]\in\Omega$ so that $n(x)$ is minimal. 
{\bf Claim}  $n(x)=0$. 

\begin{proof}[Proof of the claim] 
If $n(x)\neq 0,$ then some $b_j\ne 0$. There is $\gamma\in\Gamma$ which has  
eigenvalues $\lambda_j(\gamma),\overline\lambda_j(\gamma)$ that are not real. 
Let $\langle\gamma\rangle$ be the cyclic group generated by $\gamma$.
Let $C\subset B_j$ be the  
convex hull of the orbit $\langle\gamma\rangle\cdot b_j$.

Suppose $0\notin C$. 
Then $K=\cl({\mathbb P}_+(C))$ is a closed
convex cell in ${\mathbb P}_+(B_j)$
that is preserved by $\gamma$. By the Brouwer fixed point 
theorem, $\gamma$ fixes a point  $[v]\in K$, so $v\in B_j$ 
is an eigenvector of $\gamma|B_j$ with a positive eigenvalue. 
However every eigenvector for $\gamma$ in $B_j$ has eigenvalue $\lambda_j(\gamma)$ or 
$\overline\lambda_j(\gamma)$ which are both not real. This contradiction shows that $0\in C.$

The convex cone ${\Cone}\Omega\subset V$  is preserved by $\Gamma$.
Since $0\in C$ there is a finite 
convex combination $\sum t_i\gamma^i b_j=0$ with  $t_i\ge 0$ and $\sum t_i=1$.  
Since $x\in{\Cone}\Omega$ and this cone is 
 $\Gamma$-invariant
 it follows that  $\gamma^i x\in{\Cone}\Omega$. Since $\Cone\Omega$ is convex{\edit,} the convex combination
 $x'=\sum t_i\gamma^i x\in{\Cone}\Omega$. In particular $x'\ne 0$ and  $[x']\in\Omega$. The component of $x'$ in $B_j$ is $\sum t_i\gamma^i b_j=0$.
Since the conjugate weights spaces
are $\Gamma$ invariant, the property that a point has a  zero component in some $B_i$ is preserved by $\Gamma$, so $n(x')<n(x)$ contradicting minimality.
Hence no such $b_j$ exists{\edit, and} this proves the claim.
\end{proof}
 
 Since $x\ne0$ it follows that $A\ne 0$ and $[x]\in W:=\Omega\cap{\mathbb P}(A)$  is a nonempty properly convex set that is
preserved by $\Gamma$.
The submanifold $M=W/\Gamma$ of $P$ is  convex and $\pi_1M\rightarrow\pi_1P$ is an isomorphism
so $\dim(M)\ge n-1$ by (\ref{Kpione}).
Now $B_i$ has real dimension at least $2$, so
$\dim A\le\dim V-\dim B_i\le n-1$. {\edit But $\dim M=\dim\PP(A)=\dim A-1\le n-2$,} which  is a contradiction.
\end{proof}

{\edit Suppose $H$ is a Lie group.}
A {\em virtual syndetic hull} of a discrete subgroup $\Gamma< H$
 is a connected Lie subgroup $G< H$ such that
$|\Gamma:G\cap\Gamma|<\infty$ and $(G\cap\Gamma)\backslash G$ is compact.  In other
words $\Gamma$ is virtually a (cocompact) lattice in $H$.}
When  syndetic hulls exist they are not always unique because the exponential map on $\mathfrak{gl}(n)$ is not injective for $n\ge 2$. 
It is useful to have a unique version of a syndetic hull.  For more about syndetic hulls see 
Witte \cite{Witte}.  Some of the arguments that follow are inspired by 
section 9 of \cite{CLT1} which derives
the classification of cusps in {\em strictly} convex projective manifolds. In particular  this applies
to the role of the syndetic hull.

Let $\mathfrak r\subset\mathfrak{gl}$ be the subset of all matrices $M$ such that all the
eigenvalues of $M$ are real. The set  $R=\exp(\mathfrak r)$
 consists of all matrices $A$ such that every eigenvalue of $A$ is positive.  Then $\exp:\mathfrak r\longrightarrow R$
is a diffeomorphism with inverse $\log$. 
An element of $R$ is called an {\em e-matrix}{\edit,} and a group $G\subset R$ is called
an {\em e-group}. For example $\UT(n)$ is an e-group. 
 The property of being an e-group is preserved by conjugation. If $G$ is a connected e-group,
then $\exp:\mathfrak g\longrightarrow G$ is a diffeomorphism. 
If $S\subset R$ define $\langle\log S\rangle$ to be the vector
subspace of $\mathfrak{gl}$ spanned by $\log S$.

\begin{definition}\label{egrpdef} Given a discrete subgroup
$\Gamma\subset\GL(n,\RR)$ a {\em virtual e-hull} for $\Gamma$ is
 a connected Lie group $G$ that is an e-group and $|\Gamma:G\cap\Gamma|<\infty$ and 
 $(G\cap\Gamma)\backslash G$ is compact. There might not be such {\edit $G$}.
 \end{definition}

\begin{proposition}[(9.3) of \cite{CLT1}]\label{9.3}
\label{maxcuspislattice} 
Suppose that $\Gamma$ is a finitely-generated,  discrete nilpotent subgroup of $\GL(n,{\mathbb R})$. 
Then $\Gamma$ contains a subgroup of finite index $\Gamma_0$, which has a syndetic hull $G \leq  \GL(n,{\mathbb R})$ 
that is nilpotent, simply-connected{\edit,} and a subgroup of the Zariski closure of $\Gamma_0.$
\end{proposition}

\begin{lemma}\label{ehullexists} If $\Gamma\subset \UT(n)$  is a finitely-generated discrete nilpotent subgroup, 
 then it has an e-hull  $G\subset \UT(n)$. \end{lemma}
 \begin{proof}  
    By (\ref{9.3})  there is a finite index subgroup 
  $\Gamma_0\subset\Gamma$ which has a syndetic hull $G$. The Zariski closure of $UT(n)$ is
  the Borel subgroup $B$ of all
  upper triangular matrices
  in $\GL(n,\RR)$. It follows that the Zariski closure of $\Gamma$ is in $B$,
  so $G\subset B$. Moreover $G$ is connected so $G\subset UT(n)$. Since $\left|\Gamma:\Gamma_0\right| <\infty$ and
  $\Gamma_0\subset G$ it follows that
   $\langle\log\Gamma\rangle=\langle\log\Gamma_0\rangle\subset\mathfrak g$  thus
    $\Gamma\subset G$. Since $G\subset\UT(n)$ it is an e-group. 
    Moreover $\Gamma\backslash G$ is a quotient of $\Gamma_0\backslash G$ and
    so is compact.
    Hence $G$ is a syndetic hull of $\Gamma$.\end{proof}

\begin{lemma}\label{ehullunique} If $G_0$ and $G_1$ are virtual e-hulls of  a discrete subgroup
$\Gamma\subset\GL(n,\RR)$ then $G_0=G_1$.
\end{lemma}
\begin{proof} The group $H=G_0\cap G_1$ is connected because if $h\in H,$ then the one parameter group
$\exp\langle\log h\rangle$ is contained in both $G_0$ and $G_1$.   With $R$ defined above set
$\Gamma'=\Gamma\cap R$. If $\gamma\in\Gamma'$
 then $\gamma^{\edit m}\in G_i$ for some ${\edit m}>0$. Thus $\log \gamma^{\edit m}={\edit m}\log\gamma \in \log G_i$
so $\gamma\in G_i$. Thus  $\Gamma'=\Gamma\cap G_i\subset G_i$ so $\Gamma'\subset H$.
 Since $\Gamma'$ is  a lattice in $G_i$, and $H$ is a closed subgroup of $G_i$, it follows
that $\Gamma'$ is also a lattice in $H$. 
Since $H$ and $G_i$ are diffeomorphic to their Lie algebras, if $H\ne G_i$ then $\dim H<\dim G_i$
which contradicts that $\Gamma'$ is a lattice in $G_i$.\end{proof}

\begin{definition}\label{Thulldef} If $\Gamma\subset GL(n,\RR)$  is finitely-generated and $\Gamma\in VFG,$ then
the {\em translation group} of  $\Gamma$
 is $T(\Gamma)=\exp\langle\log (\core(\Gamma,n))\rangle$.
\end{definition}

\begin{theorem}\label{transgrp}  If $\Omega/\Gamma$ is a quasi-cusp 
then  $T(\Gamma)$ is the unique virtual e-hull of $\Gamma$.
\end{theorem}
\begin{proof} {\edit Set $n=1+\dim\Omega$.} Recall that the definition of quasi-cusp implies $\Gamma$ is virtually nilpotent.
By (\ref{gencuspisVFG})  $\core(\Gamma,n)$
 is conjugate into $\UT(n)$ and is therefore 
an e-group. 
By (\ref{ehullexists}) $\core(\Gamma,n)$ has an e-hull, $T$, that is conjugate into $\UT(n)$. Thus $T$ is
a virtual e-hull of $\Gamma$. Uniqueness of $T$
follows from (\ref{ehullunique}). It is  now clear that $T=T(\Gamma)$.
\end{proof}

The next thing to do is show that{\edit, if $\Omega/\Gamma$ is a generalized cusp of dimension $n$,} the orbit under $T(\Gamma)$ of a point $x\in\partial\Omega$ is a strictly
convex hypersurface. {\edit The key to doing this is to show that, if the cusp is minimal, then $\Omega$
is a closed convex subset of $\RR^n$ bounded by $\partial\Omega$, see (\ref{RFendlemma}).}

A {\em projective flow} $\flow$ on $\RPn$ is a continuous monomorphism 
$\flow:{\mathbb R}\longrightarrow \PGL(n+1,{\mathbb R})$. 
There is an infinitesimal generator $A\in\mathfrak{gl}_{n+1}$
with $\Phi_t:=\Phi(t)=\exp(tA)$.
If  $p\in\RPn$ and  $\flow_t(p)=p$ for all $t,$ then $p$ is a {\em stationary point} of $\Phi$. 
 A  {\em radial flow} is a projective flow that is stationary on a hyperplane $H\cong {\mathbb R}P^{n-1}${\edit,}
and that is parameterized so that $\Phi_t(p)\to  r\in H$ as $t\to-\infty$ whenever $p$ is not stationary.
It follows that  $\flow_t=\exp(t A),$
where $A\in \mathfrak{gl}_{n+1}$
is a rank one matrix{\edit,} and $H$ is the projectivization
of  $\ker A$. The projectivization of the image of $A$ is a point $p\in{\mathbb R}P^n$, called
the {\em center of the flow},
 that is also fixed by $\flow$. Every orbit is contained in a line containing the center. This property  characterizes radial flows.

A radial flow is  {\em parabolic} if $p\in H$ and {\em hyperbolic} otherwise. 
Every radial flow is conjugate to one generated by an elementary matrix $E_{i,j}$.
A parabolic flow is conjugate to
$(I + t\cdot E_{1,n+1})$ and a hyperbolic flow is conjugate to
the diagonal group
$(\exp(t),1,\cdots,1)$.  
The {\em backward orbit} of $X\subset\RPn$ is
$\flow_{(-\infty,0]}(X)$. A set $X\subset\RPn$ is {\em backward invariant} if $X$ contains its backward orbit,
and it is {\em  backward vanishing} if  $\cap_{t<0} \Phi_t(X)=\emptyset$.

A {\em displacing hyperplane} for a radial flow $\Phi$ is a hyperplane $P$ such that 
$P$ and $\Phi_t(P)$ are disjoint in $\RPn\setminus H$ for  all $t\ne 0$. A hyperplane $P$
is displacing if and only if $P\ne H$ and
$P$ does not contain the center of $\Phi$. 

 \begin{proposition}\label{gencuspRF} Suppose $\Omega/\Gamma$ is a quasi-cusp and 
$\Gamma\subset\UT(n+1)$. {\edit If} $\lambda:\Gamma\to\RR^*$ {\edit is a weight}
with generalized weight space $V=V(\lambda)$ {\edit then}
 there is a radial flow $\Phi=\Phi^{\lambda}$ that is centralized
by $\Gamma$, and $\Phi$ acts trivially on each generalized weight space other than $V$. 

If $\dim V\ge2$, then $\Phi$ is parabolic, and if $\dim V=1,$ then $\Phi$ is hyperbolic.
The center of $\Phi$ is {\edit contained in} $\PP(E(\lambda))$.
 The group $G(\Gamma):=  T(\Gamma)\times\Phi(\RR)$ 
generated by $T(\Gamma)$ and $\Phi(\RR)$ is  their internal direct product.
If the orbit of $x\in\RPn$ under $T(\Gamma)$ is a strictly{\edit-}convex hypersurface then 
 $G(\Gamma)\cdot x\subset\RPn$ is open.
\end{proposition}
\begin{proof} We may assume $\Gamma$ {\edit is} upper-triangular and block diagonal{\edit,} with one  block
for each generalized weight space. 
We may assume $V$ is the first block and set $m=\dim V$. As above, let $E_{i,j}\in\mathfrak{gl}(n{\edit +1})$
be the elementary matrix with $1$ in row $i$ and column $j$. Define $\Phi(t)=\exp(t E_{1,m})$.
Then $\R^{n\edit+1}=V\oplus W$ where $W$ is the sum of the other generalized weight spaces
and the action of $\Phi$ on $W$ is trivial. 
If $m=1$ then $\Phi(t)=\diag(\exp(t),1,\cdots,1)$ is a hyperbolic flow. 
If $m\ge 2$ then $\Phi(t)$ is a parabolic flow
given by the
unipotent subgroup with $t$ in the top right corner of the block for $V$. 
The center is $\PP(e_1)$ and the stationary hyperplane is
$H=\PP(\langle e_1,\cdots,e_{m-1}\rangle\oplus W)$.
It is easy to check that $\Gamma$ centralizes $\Phi$.

Since $T(\Gamma)=\exp(\mathfrak t)$ and $\Phi(\RR)=\exp(\mathfrak f)$ are $e$-groups, 
 if they have a nontrivial intersection, then $\Phi(\RR)\subset T(\Gamma)$.
The orbits of $\Phi$ are lines. 
If $S=T(\Gamma)\cdot x$ is a strictly{\edit-}convex hypersurface,
then it does not contain a line so
 $\Phi(\RR)\cap\Gamma$ is trivial{\edit, and $\Phi(\RR)\cdot S\subset\RPn$ is open.} 
Since $\Phi(\RR)$ and $T(\Gamma)$
commute they generate $G= T(\Gamma)\times \Phi(\RR)$.
 \end{proof}

\begin{definition}\label{RFcuspdef} A radial flow $\flow_t$ is {\em compatible} with a properly
convex manifold $M=\Omega/\Gamma$ if $\flow(\RR)$ commutes with $\Gamma${\edit,}
and $\Omega$ is disjoint from the stationary hyperplane of $\Phi${\edit,} and $\Omega$
is backward invariant and backward vanishing.

A {\em radial flow end} is a properly convex manifold $M=\Omega/\Gamma$ with compact, strictly
convex boundary{\edit,} and for which there is a compatible radial flow. A {\em radial flow cusp} is a radial flow end that is also a generalized cusp.
 \end{definition}

The hypersurfaces $\widetilde{S}_t:=\Phi_{-t}(\partial\Omega)$ are strictly{\edit-}convex and $\Gamma$-invariant. {\edit Those with $t\ge 0$  
  foliate $\Omega$.} They are all disjoint from $H$.  
 Their images under the projection $\pi:\Omega\to M$ give a product 
 foliation of $M$ by compact{\edit,} strictly{\edit-}convex{\edit,} hypersurfaces $S_t=\pi(\widetilde{S}_t)$. 
 There is a transverse foliation of $\Omega$ by flow-lines that limit on the center of $\Phi$.
 These project to a transverse foliation of $M$ by rays.

The {\em flow time function} is  $\widetilde{T}:\Omega\to[0,\infty)$ defined by
 $\widetilde{T}(x)=t$ if $\Phi_t(x)\in\partial\Omega$. Thus
 $\widetilde{T}(x)$ is the amount of time for $x$ to flow into $\partial\Omega$
 and $\widetilde{T}(\widetilde{S}_t)=t$.
Let $\pi:\Omega\to M$ be the projection. Then 
there is a map $T:M\to[0,\infty)$ defined by $T(\pi x)=\widetilde{T}(x)$.
 The level sets of $T$ are the hypersurfaces $S_t$.

\begin{lemma}\label{RFcenter} Suppose $\Phi$ is a radial flow with center $p$ and stationary hyperplane $H$.
Suppose $\Omega\subset\RPn\setminus H$ {\edit is properly convex.} If $\Phi$ is hyperbolic and $p\notin\cl(\Omega),$
then $\Omega$ is backward vanishing.
If $\Phi$ is parabolic, then $\Omega$ is backward vanishing for either $\Phi(t)$ or $\Phi'(t):=\Phi(-t)$.
\end{lemma}
 \begin{proof} If $\Phi$ is hyperbolic and $p\notin \cl(\Omega),$ then
 by the Hahn-Banach separation theorem there is {\edit a hyperplane} $P$ that separates $\Omega$ from $p$.
If $\Phi$ is parabolic, then choose any hyperplane $P$ disjoint from $\Omega$ that does
not contain $ p$. In either case $P$ is a displacing hyperplane. After possibly reversing $\Phi$ in the parabolic case,
{\edit the component of $\RR^n\setminus P$ that contains $\Omega$
is  a half-space that} is backward vanishing, and hence so is $\Omega$.
 \end{proof}
{\edit The reason for introducing radial flow ends is:}
\begin{lemma}\label{RFendlemma} Suppose $\Omega\subset\RPn$ and 
$M=\Omega/\Gamma$ is a radial flow end with radial flow $\Phi.$
Let $H\subset \RPn$ be the stationary hyperplane for $\Phi${\edit, 
and $\overline{\Omega}=\cl(\Omega)\subset\RP^n$}. Then 
$\partial\overline\Omega=\partial\Omega\sqcup(H\cap\overline\Omega)$. 
{\edit In particular, $\Omega$ is a closed convex subset of $\RR^n=\RP^n\setminus H$
bounded by the properly embedded, strictly{\edit-}convex hypersurface $\partial\Omega$.}
\end{lemma}
\begin{proof}  Let $\RR^n=\RPn \setminus H,$ so that $\Omega\subset\RR^n$. 
 It suffices to show that $\partial\Omega$ is properly embedded in $\RR^n$ and therefore 
$\Omega$ is a  closed convex set in $\RR^n$ bounded by $\partial\Omega$.

   Let $p$ be the center of $\Phi$. 
 Choose  a displacing hyperplane $P\subset\RPn$ that is disjoint from $\Omega$
such that if $\Phi$ is hyperbolic then $P$ separates
$p$ from $\Omega$ in $\RR^n$.  

Let $U$ be the closure of the component of $\RR^n\setminus P$
that is the half-space containing $\Omega.$ Then
$U$ is backward invariant. Thus $U$ is the backward orbit of $P$.
 Define the function $\tau:U \to[0,\infty)$ by
 $\tau(x)=t$ if $\Phi_t(x)\in P$. This is the amount of time it takes $x$ to flow into $P$.
  Observe that if $x,y\in\Omega,$ then $\widetilde{T}(x)-\widetilde{T}(y)=\tau(x)-\tau(y)$.

  Because  {\edit $\widetilde{S}_t:=\Phi_{-t}(\partial\Omega)$} is strictly{\edit-}convex{\edit,} it follows that 
 the only critical points of the restriction of  $\widetilde{T}$ to a segment are maxima, and therefore there is
  at most one critical point on a segment. Thus $T:M\to[0,\infty)$ has the same critical point behavior along
  segments.
  
  Choose a metrically-complete Riemannian metric
 on $M$ and use the lifted metric  
 on $\Omega$.
Suppose $\partial\Omega$ is not properly embedded in $\RR^n$. Then there is a sequence 
 $\widetilde p_k\in\partial\Omega$
 which converges in $\RR^n$ to a point $\widetilde p_{\infty}\notin\partial\Omega$. 
   
 Let $\alpha_k$ be the length of $[\widetilde p_0,\widetilde p_k]\subset\Omega$. 
 Then $\alpha_k\to\infty$ because $\widetilde{p}_{\infty}\notin\Omega$ and the metric on $\Omega$ is complete. Let
 $\widetilde{\ell}_k:[0,1]\to[\widetilde p_0,\widetilde p_k]$ 
 be the unit segment.
  {\edit Then} $\widetilde{\ell_k}$ converges to $\widetilde{\ell}_{\infty}:[0,1]\to[\widetilde p_0,\widetilde p_{\infty}]$.
      The restriction of $\widetilde\ell_{\infty}$ to $[0,1)$ is a ray, $\widetilde\ell:[0,1)\to\Omega$, of infinite length in $\Omega$.
 Since $\widetilde{p}_k\to\widetilde{p}_{\infty}$ there is $\beta>0$ such that 
${\edit \widetilde T}\circ\widetilde\ell_k\le\beta$ for all $k\in[0,\infty]$. 

  The projection $\ell_k=\pi\circ\widetilde\ell_k:[0,1]\to M$ is an immersed affine segment 
   and $T\circ\ell_k\le \beta$. Thus $\ell_k$ is contained in the compact set $M_{\beta}:=\cup_{0\le t\le \beta}S_t$.
These segments converge to the ray $\ell=\pi\circ\widetilde{\ell}$ of infinite length that is also contained $M_{\beta}$.
 Now $T\circ\ell:[0,1)\to[0,\beta]$ is eventually monotonic. {\edit Thus there is a segment, $\ell^*:[0,1]\to M_{\beta}$, of length $1$ 
 that is a limit of subsegments of $\ell$ of length 1, and
 $T\circ \ell^*$ is some constant $\alpha$. Thus $\ell^*$}
 is contained in 
$S_{\alpha}$.
But this contradicts the fact that $S_{\alpha}$ is strictly{\edit-}convex.
It follows that $\partial\Omega$ is properly embedded in $\RR^n$.
Hence $\Omega$ is a  closed convex set in $\RR^n$.
 \end{proof}

{\edit To apply this we need}
\begin{proposition}\label{gencuspcontainsRFcusp} Every minimal generalized cusp $C=\Omega/\Gamma$ with 
$\Gamma\subset\UT(n+1)$
is a radial flow cusp.
\end{proposition}
\begin{proof}   It suffices to show there is a radial flow that is compatible with $C$.

{\bf  Claim 1.} $\Omega$ is disjoint from every $\Gamma$-invariant hyperplane $ H$.

\begin{proof}[Proof of claim 1]
If $ H\cap\Omega\ne\emptyset$, then  
$ H\cap\partial\Omega\ne\emptyset$ since $C$ is minimal. Observe that
 $ H\cap\Omega$ is properly convex and preserved by $\Gamma$. 
Thus $R=(H\cap\Omega)/\Gamma$ is a convex codimension-1 submanifold of $C$  with nonempty boundary{\edit,}
 which contradicts (\ref{Kpione}).\end{proof}

There are now two cases:

{\bf Parabolic case.} There is a generalized weight space $W$ for $\Gamma$ with $\dim W\ge2$.
 Let $\Phi$ be
the parabolic radial flow  that centralizes $\Gamma$ given by
(\ref{gencuspRF}). Let $H$ be the stationary hyperplane
and $p\in H$ the center of $\Phi$. 
Let $P$ be a displacing hyperplane that is tangent to $\Omega$ at $q\in\partial\Omega$.

{\bf Hyperbolic case.} Every generalized weight space has dimension $1$, so $\Gamma$ is diagonalizable.
The weight spaces projectivize to give points $p_0,\cdots ,p_n\in\RPn$ that are in general position.
The hyperplane $P_i$  contains all these points except $p_i$. 
These hyperplanes divide $\RPn$
into  $2^{n}$ open $n$-simplices. These hyperplanes are $\Gamma$-invariant so $\Omega$
is contained in  one of these simplices{\edit, say $\Delta$,}  by minimality of the cusp. There is a vertex $p$ of
$\Delta$ with  $p\notin \overline\Omega$ because $\partial\Omega$ is a strictly{\edit-}convex hypersurface in $\Delta$.
After relabelling $p=p_0$,
let $H=P_0$ and let $\Phi$ be the radial
flow with center $p$ and stationary hyperplane $H$. Then $\Phi$ centralizes $\Gamma$ and $p$ is disjoint
from $\cl(\Omega)$. By  the proof of (\ref{RFcenter})
 there is a displacing hyperplane $P$ that separates $p$ from $\Omega$.

In each case, $\Omega$ is disjoint from $H$ by claim 1.
Set $\RR^n=\RPn\setminus  H$ so $\Omega\subset\RR^n$.
Let $U$ be the closure of the component of $\RR^n\setminus P$ that contains $\Omega$.
Choose linear coordinates on $\RR^n$ such that $q=e_1=(1,0,\cdots,0),$ {\edit and
$U$ is the half-space} $x_1\ge 1,$ and, moreover, $p=0$ in the hyperbolic case and $p$ is the limit of the positive $x_1$-axis in the
parabolic case.
 Then $P=\partial U$ is the horizontal hyperplane $x_1=1$. 

We may assume $U$ is backward invariant after possibly reversing the flow in the parabolic case.
We reparameterize $\Phi$ so that
 in these coordinates{\edit,} in the parabolic case $\Phi_t(x) = x-t\cdot e_1${\edit,} and in the hyperbolic
 case $\Phi_t(x)=\exp(-t)\cdot x$.

Let $p_1:U\to P$ be the projection along flow-lines. In the parabolic case
$p_1(x_1,\cdots,x_n)=(1,x_2,\cdots,x_n)$ and in the hyperbolic case
$p_1(x_1,\cdots,x_n)=(1,x_2/x_1,\cdots,x_n/x_1)$.   Let $\Omega_1$ be the backward orbit of $\interior \Omega$.

{\bf Claim 2.}  {\edit $\Omega_1$ is open in $\RR^n$, properly convex, backward invariant and contains $\interior(\Omega)$}.
\begin{proof}[Proof of claim 2] {\edit Clearly $\Omega_1$ is open, backward invariant, and contains $\interior(\Omega)$.}
Suppose $a,b\in\Omega_1.$ Then $a=\Phi_{\alpha}(a')$
and $b=\Phi_{\beta}(b')$ for some $\alpha,\beta\le 0$ and $a',b'\in\interior \Omega$. Let $\ell=[a',b']$ be the line segment with endpoints $a',b'$.
Since $\Omega$ is convex $\ell\subset\Omega$. Then $\cup_{t\le0}\Phi_t(\ell)$ is a planar convex set in $\Omega_1$
that contains $a$ and $b$. Hence $\Omega_1$ is convex.

Let $C$ be the cone of $\Omega$ from $0$. Since $\Omega$ 
is properly convex{\edit,} and ${\edit 0}\notin\overline{\Omega}${\edit,} it follows that $C$ is properly convex. Moreover
$C$  contains $\Omega_1$, so $\Omega_1$ is properly
convex, proving  claim 2.
\end{proof}

We want to add a boundary to $\Omega_1$ and show this gives $\Omega$.
Define $\Omega_M$ to be the {\em flow closure} of $\Omega_1$, i.e.\thinspace  the set of all points $x$
such that $\Phi_{t}(x)\in\Omega_1$ for all $t<0$.  Clearly $\Omega_M\subset\cl(\Omega_1)\edit\subset\RR^n$. There is
a homeomorphism $\widetilde F:\partial\Omega_M\times[0,\infty)\to \Omega_M$
given by $F(x,t)=\Phi_{-t}(x)$. {\edit Since $\Omega_1$ is open, $\Omega_M$ is a manifold with boundary $\partial\Omega_M$.}  
It is clear that $\Omega_M$ is disjoint from $H$, backward invariant and backward vanishing.

{\bf Claim 3.} $\Gamma$ acts freely and properly discontinuously on $\Omega_M.$

\begin{proof}[Proof of claim 3]
Since $\Gamma$ commutes with $\Phi$ it follows that $\Omega_1$ is preserved by $\Gamma$.
{\edit Also} $\Gamma$ acts freely on $\Omega${\edit,} it contains no elliptics{\edit,} and therefore acts freely
on $\Omega_1$. By (1.3) of \cite{CLT1}  $\Gamma$ is discrete and therefore acts properly discontinuously on
$\Omega_1$. The map $\Phi_{-1}$ embeds $\Omega_M$ into $\Omega_1$ and, since
$\Phi_{-1}$ commutes with $\Gamma$, it follows that $\Gamma$ acts freely and properly discontinuously
on $\Omega_M$.
\end{proof}

Thus $M=\Omega_M/\Gamma$ is a properly convex manifold and 
there is a homeomorphism $F:\partial M\times[0,\infty)\to M$ covered by $\widetilde F$.  
{\bf Claim 4.}  $M=C$.

 First we show $C\subset M$. Since $ \interior(\Omega)\subset\Omega_1$ 
 it follows that $\interior(C)\subset\interior(M)$. 
By (\ref{smoothbdry}) there is 
a collar neighborhood $P\subset C$ of $\partial C$ with $\partial P=\partial C\sqcup Q$
and $Q$ is strictly{\edit-}convex. Let $R=\cl(C\setminus P)$ so $\partial R=Q$ and $R$
is a generalized cusp   contained in $M$. 
{\edit Thus $X=\cl(M\setminus R)\cong\partial X\times I$
 is a compact,}  and $P\setminus\partial C\subset X,$ so
it follows that $\partial C {  \subset X}\subset M$, thus  $C\subset M$.

 The intersection of an orbit under the flow $\Phi$ with $\Omega_M$, 
 and also its image in $M$, is called a {\em flow-line}.
Every flow-line  $\lambda$ in $M$ ends on $\partial M$ and is properly
 embedded. Moreover $C\cap\lambda$ is convex.
Since $\partial C$ separates $\partial M$ from the end of $M$,
 it follows that $\lambda\cap C$ is all of $\lambda$ except, possibly, a
bounded interval at the end of $\lambda.$ It follows that  
for all $t<0$ that $\Phi_t(C)\subset C$ so $C$ is backward invariant.
It now follows that $\Omega_1=\Omega$ and $M=C${\edit, and this implies $\Phi$ is compatible with $C$.}
\end{proof}

{\edit Now we know that $\Omega$ is a closed subset of $\RR^n$ we are ready to show the orbits of the translation group are properly embedded
convex hypersurfaces in $\RR^n$.}
\begin{proposition}\label{convexTorbit} Suppose $C=\Omega/\Gamma$ is a minimal generalized cusp and $\Gamma\subset\UT(n+1)$.
 Let  $T=T(\Gamma)$ be the translation group.
Then $C$ contains a minimal homogeneous cusp $C_T=\Omega_T/\Gamma$ and
 $T$ acts transitively on $\partial\Omega_T$.\end{proposition}
 \begin{proof} By (\ref{gencuspcontainsRFcusp})  $C$  is a radial flow cusp for some flow $\flow$. 
Let $H$ be the stationary hyperplane of $\Phi$
and set $\RR^n=\RP^n\setminus H$. Since $T$ centralizes $\Phi${\edit,} it preserves $H${\edit,}
and acts affinely on $\RR^n$.
{\bf Claim.}  There is $x\in\Omega$ such that $T\cdot x\subset\Omega$.

\begin{proof}[Proof of claim]
Let $\pi:\Omega\longrightarrow C$ be the
covering space projection. 
There is a continuous map $F:T\times\partial\Omega\to\RR^n/\Gamma$ given by $F(t,x)=\pi(t\cdot x)$.
Since  $\partial C=\partial\Omega/\Gamma$ is compact{\edit,}
there is  a compact subset $D\subset\partial\Omega$ such that $\Gamma\cdot D=\partial\Omega.$
So $T\cdot\partial\Omega$=$T\cdot D$ because $\Gamma\subset T$.
There is compact $X\subset T$ such that $\Gamma X=T.$
So $T\cdot D=(\Gamma X)\cdot D$.
Hence $\image(F)=\pi(X\cdot D)$ because
$\pi(\Gamma\cdot x)=\pi(x)$. Thus ${\edit K=}\image(F)\subset \RR^n/\Gamma$
is compact{\edit,}  because it equals $F(X\times D)${\edit,} and $X\times D$ is compact. 
Choose $x\in\interior(\Omega)$ such that $\pi(x)\notin K$. {\edit Then $T\cdot x$ is connected, and disjoint from $\partial\Omega$.
By  (\ref{RFendlemma}) $\interior(\Omega)\subset\R^n$ is 
bounded by  $\partial\Omega$.
It follows that}   $T\cdot x\subset\Omega$. This proves the claim.
\end{proof}

The set $\Omega_{\edit T}=\cl(\CH(T\cdot x))\subset\Omega$ is properly convex and  $T$--invariant.
{\edit {\bf Claim} $\partial\Omega_{\edit T}$ is strictly convex.}
Since $\Omega_{\edit T}$ is a closed properly convex set in $\RR^n${\edit,} there is an extreme point
 $y\in\partial \Omega_{\edit T}$ at which $\partial \Omega_{\edit T}$ is strictly{\edit-}convex. 
Thus $\partial\Omega_{\edit T}$ is strictly{\edit-}convex at every point in the orbit $S=T\cdot y\subset\partial\Omega_{\edit T}$. 
 Then $\Omega_{\edit S}:=\CH(S)\subset\Omega_{\edit T}$ is properly
convex. Since $\Gamma\subset T$ it follows that $\Omega_{\edit S}$ is $\Gamma$-invariant.
Thus $C_{\edit S}=\Omega_{\edit S}/\Gamma$ is a generalized cusp and a submanifold of $C$.  It follows
from (\ref{Kpione}) that
 $\dim\Omega_{\edit S}=\dim\Omega$.
Moreover $S\subset \partial\Omega_{\edit T}$ so $S\subset\partial\Omega_{\edit S}$. 
Since $\Omega_{\edit S}=CH(S)$
 it follows that $S= \partial\Omega_{\edit S}{\edit=\partial\Omega_T}$ and $\Omega_{\edit S}=\Omega_{\edit T}$, {\edit which proves the claim.}
 \end{proof}

{\edit If $\Omega/\Gamma$ is a generalized cusp, by (\ref{convexTorbit})
 there is a homogeneous domain $\Omega_T\subset\Omega$
that is preserved by the finite index subgroup of $\Gamma\cap T(\Gamma)$. Next we show that $\Omega_T$ is preserved by all of $\Gamma$.}

\begin{lemma}\label{virthomogcusp} Suppose $C=\Omega/\Gamma$ is a minimal generalized cusp 
and $T=T(\Gamma)\subset\UT(n+1)$ and $\Gamma_0=T\cap\Gamma$.
Suppose  $\Omega/\Gamma_0$
contains a homogeneous cusp $\Omega_T/\Gamma_0$ and $\Omega_T$ is
preserved by $T$. Then $\Gamma$ preserves $\Omega_T$ so $C$ contains the homogeneous generalized cusp
$\Omega_T/\Gamma$.
\end{lemma}
\begin{proof} By (\ref{gencuspcontainsRFcusp})
  $C^*=\Omega_T/\Gamma_0$
is a radial flow cusp and by (\ref{RFendlemma}) $\Omega_T\subset\RR^n$ is bounded by
 the strictly{\edit-}convex properly embedded hypersurface $\partial\Omega_T$.
By (\ref{transgrp}) $T=T(\Gamma)$
  is the unique translation group that contains $\Gamma$. 
  
 Since $\Gamma$ normalizes itself it follows that $\Gamma$ normalizes
  $T$ and therefore $\Gamma$ permutes the decomposition of $\RPn$ into $T$--orbits. 
  The domain $\Omega_T$ is foliated by $T$--orbits and $\Omega_T/T\cong[0,1)$.
  Since $\Gamma\cap T$ preserves $\Omega_T$ and $|\Gamma:\Gamma\cap T|<\infty$ it follows the
  $\Gamma$--orbit of $\Omega_T$ is a finite number of pairwise disjoint convex sets all contained
  in $\Omega$. Thus $\Gamma\cap T$ permutes these domains. There is a finite index subgroup
  $\Gamma_1\subset\Gamma\cap T$ that preserves each domain. We may assume $\Gamma_1$
  is normal in $\Gamma$.
  Thus $M=\Omega/\Gamma_1$ is a regular cover of $C$ that
   contains one copy of $P=\Omega_T/\Gamma_1$ for each domain.
 However $M$ and each copy of $P$ is a generalized cusp.
 Each copy of $\partial P$ separates $\partial M$ from the end of $M$. Since the copies of $P$ are
 disjoint
 there is only one copy of $P$
 and $\Gamma$ preserves $\Omega_T$.
 \end{proof}
 \begin{proof}[Proof of (\ref{sttdcusp})]  Suppose $C=\Omega/\Gamma$ is a generalized cusp {\edit of dimension $n$}.
 We may assume $C$ is minimal by  (\ref{mincusplem}).
 Since $\Gamma$ is virtually nilpotent,  it follows from
  (\ref{gencuspisVFG}) there is a finite index subgroup $\Gamma'<\Gamma$ 
  that is conjugate into $\UT(n+1)$. We will assume this conjugacy has been done.
Then $\widetilde C=\Omega/\Gamma'$ is a generalized cusp that is a finite cover of $C$ and
is minimal by (\ref{mincusplem}). 
It follows from (\ref{transgrp}) that $\Gamma'$ is a lattice in a connected upper-triangular Lie group $T=T(\Gamma)$.
 By (\ref{gencuspcontainsRFcusp}) it follows that $\widetilde{C}$ is a radial flow cusp 
 for a radial flow $\Phi$ with stationary hyperplane $H$. Let $\RR^n=\RPn\setminus H.$ By (\ref{RFendlemma}) $\Omega\subset \RR^n$ is a closed strictly{\edit-}convex set bounded
 by the strictly{\edit-}convex hypersurface $\partial\Omega$. By (\ref{convexTorbit})  there
 is a properly convex $\Omega_T\subset\Omega$ that is $T$-invariant and thus $\Gamma'$-invariant. By (\ref{virthomogcusp})
 $\Omega_T$ is preserved by all of $\Gamma$ hence $\Omega_T/\Gamma$ is a homogeneous
 cusp in $C$ and $\Gamma< \PGL(\Omega_T)$.  \end{proof}

 \begin{lemma}\label{deformC} Suppose $G$ is a connected group with $\dim G=n-1$. 
  For $x\in\RPn$  the subset of $\Hom(G,\GL(n+1,\RR))$ consisting of all $\rho$ with $\rho(G)\cdot x$ a  Hessian
 convex hypersurface is open. \end{lemma}
 \begin{proof}
Suppose the map 
 $f:G\longrightarrow\RP^n$ given by $f(g)=(\rho(g))\cdot x$ has image a strictly{\edit-}convex hypersurface $S$.
 Because $G$ acts transitively on $S$ by projective maps it follows that
  $S$ is strictly{\edit-}convex everywhere if and only if it is strictly{\edit-}convex
 at the single point $x$. 
 {\edit Let $\nu$ be the normal to $S$ at $x$, and $e\in G$ the identity.
 Hessian
  convexity of $S$ at $x$ is equivalent to the quadratic form $Q=\nu\cdot D^2_{e}f$ being positive or negative definite.}
 This form $Q=Q(\rho)$ is a smooth function of $\rho${\edit, and the} set of definite quadratic forms is open in the set of
 all quadratic forms.
 \end{proof}
 
{\edit \begin{definition}\label{gencusphol} If $C$ is a generalized cusp of dimension $n$, then $\GC(C)$ is the subspace of $\Hom(\pi_1C,\GL(n+1,\RR))$ 
 consisting of holonomies of all generalized cusp structures on $C$, and $
\VFG({\edit C})$ is the subspace of all $\rho$ with $\rho(\pi_1 {\edit C})\in \VFG$\end{definition}}

\begin{theorem}[stability of generalized cusps]\label{cuspstab} Suppose ${\edit C}$ is
 a generalized cusp of dimension $n$. {\edit Then 
 \begin{itemize}
 \item[(1)] $\VFG({\edit C})$ is semi-algebraic.
 \item[(2)] $\GC(C)\subset\VFG(C).$ 
 \item[(3)]  $\theta:\Kleinian({\edit C})\rightarrow\VFG({\edit C})$ given by $\theta(\Omega,\rho)=\rho$ is a continuous open map.
 \item[(4)] $\GC(C)=\image(\theta)$.
 \end{itemize}}  \end{theorem}
\begin{proof} By (\ref{VnilVFG}) $\VFG({\edit C})$ is semi-algebraic. 
{\edit By definition $(\Omega,\rho)\in\Kleinian({\edit C})$ if and only if $\Omega/\Im(\rho)$
is a generalized cusp diffeomorphic to ${\edit C}$, thus $\GC(C)=\image(\theta)$.
Every generalized cusp is homogeneous by (\ref{sttdcusp}), and
 (\ref{gencuspisVFG}) implies the holonomy contains a finite index subgroup that is
 conjugate into an upper triangular group, so  $\GC(C)\subset \VFG({\edit C})$. It is obvious that $\theta$ is continuous, it only remains to show it is open.}
 
Set $H=\core(\pi_1{\edit C},n+1)$. By  (\ref{Thulldef}) , given $(\Omega,\rho)\in\Kleinian({\edit C})$
 the translation group is  $$T(\rho):=T(\rho(\pi_1{\edit C}))=\exp\langle\log (\rho H)\rangle.$$
 This is clearly a continuous function of $\rho$.

Choose $x\in\partial\Omega$. By (\ref{deformC}) for $\sigma\in \VFG({\edit C})$ close enough to $\rho$
the hypersurface $S=T(\sigma)\cdot x$ is  Hessian convex{\edit, and close to $\partial\Omega$.}
By (\ref{gencuspRF}) there is a radial flow $\Phi$ that is centralized by $T(\sigma)$ and the
group $G=T(\sigma)\oplus\Phi(\RR)$   has an open orbit $W$ in $\RPn$.
 Moreover $W$ is foliated by the strictly{\edit-}convex hypersurfaces
$S_t=\Phi_t(S)$.

After replacing $t$ by $-t$ we may assume for $t<0$ and close to $0$ that $S_t$ is on the convex
side of $S=S_0$. Let
$\Omega^+=\cup_{t\le0}S_t.$ Then $\partial \Omega^+=S_0$.
This set is preserved by $T(\sigma)$ and therefore by $\sigma(H)$. It is contained in a properly convex cone
 by the argument of
 (\ref{RFhessianisconvex}) using Figure~\ref{fig1}. 
 Hence  $\Omega(\sigma):=\CH(\Omega^+)$ is properly convex and $T(\sigma)$--invariant. The argument of claim 3 in (\ref{gencuspcontainsRFcusp}) shows that $\sigma(H)$ acts freely
and properly discontinuously on $\Omega(\sigma)$. Since $\Omega(\sigma)$ is $T(\sigma)$--invariant, 
$\partial\Omega(\sigma)$
 is  Hessian convex.
Thus $\Omega(\sigma)/\sigma(H)$ is a homogeneous generalized cusp.

It only remains to show that $\Omega(\sigma)$  is preserved by all of  $\sigma(\pi_1{\edit C})$. 
The argument is very similar to the proof of (\ref{virthomogcusp}).
The $\sigma(\pi_1{\edit C})$--orbit of $\Omega(\sigma)$ is finite because $|\pi_1{\edit C}M:H|<\infty$.
By (\ref{ehullunique}) $T(\sigma)$ is the unique virtual e-hull of $\sigma(\pi_1{\edit C})$.
Thus $\sigma(\pi_1{\edit C})$ preserves the decomposition of $\RPn$ into $T(\sigma)$--orbits.
Moreover $\Omega(\sigma)$ is a union of such orbits.
Thus if $g\in \pi_1{\edit C}$ then $(\sigma g)(\Omega(\sigma))$ is either
$\Omega(\sigma)$ or disjoint from $\Omega(\sigma)$. We need only look at finitely many such $g$.
Observe that
$\Omega(\sigma)$ is close to $\Omega(\rho)$ 
and $\rho(g)$ is close to $\sigma(g)$ and $\rho(\pi_1{\edit C})$ preserves $\Omega(\rho).$
Thus $\rho(g)$ preserves $\Omega(\rho)$ so  $(\sigma g)(\Omega(\sigma))$ intersects $\Omega(\sigma)$.
It follows that $(\sigma g)(\Omega(\sigma))=\Omega(\sigma)$.
\end{proof}

 We remark that when ${\edit C}\cong T^2\times[0,\infty)$ the subset of diagonal representations in $\GL(3,\RR)$
given by holonomies of generalized cusps is not closed. The boundary consists of
properly convex structures on ${\edit C}$ with $\partial {\edit C}$ {\edit flat, thus} not strictly{\edit-}convex. 


\begin{theorem}[Main Theorem]\label{mainthm}
Suppose $N$ is a compact connected $n$--manifold 
and $\Vcal= \cup_{i=1}^kV_i\subseteq \partial N$ is the union of some of the boundary components of $N.$
Assume $\pi_1V_i$
is virtually nilpotent for each $i$. Let $M=N\setminus\Vcal$.
Then the {\edit holonomy map $\Hol:\devGX_{ce}(M)\to\VFG(M)$ is continuous and open,
and $\VFG(M)$ is a semi-algebraic subset of $\Hom(\pi_1M,\GL(n+1,\RR))$.}\end{theorem}

\begin{proof} {\edit Continuity is obvious. That $\VFG(M)$ is semi-algebraic follows
from (\ref{VnilVFG}). Let $B_i$ be the end of $M$ corresponding to $V_i$ {\edit and $\Bcal=\cup B_i$}.
 Given a developing map $\dev\in \devGX_{ce}(M)$ the holonomy 
$\rho=\Hol(\dev)\in\VFG(M)$ by (\ref{gencuspisVFG}).
The restriction $(\dev|B_i)$ determines a generalized cusp  $B_i\cong\Omega_i/\Gamma_i$ 
  with holonomy $\Gamma_i=\rho(\pi_1B_i)$.
If $\rho'\in\VFG(M)$ is sufficiently close to $\rho$,
then by (\ref{cuspstab}) there are nearby generalized
cusps $B_i\cong\Omega'_i/\Gamma'_i$ with holonomy $\Gamma'_i=\rho'(\pi_1B_i)$. Thus
$x=(\rho,(\Omega_1,\Gamma_1),\cdots,(\Omega_k,\Gamma_k))$ and
$x'=(\rho',(\Omega_1',\Gamma'_1),\cdots,(\Omega'_k,\Gamma'_k))$ 
are close in $\enddata(M,\Bcal,\Kleinian)$.
By (\ref{kleinianedndatamapopen}) the map 
$$\enddatamap_{\Kleinian}:\devGX_{ce}(M,\proj)\to\enddata(M,\Bcal,\Kleinian)$$ is open,
so there is $\dev'$ close to $\dev$ with $\enddatamap_{\Kleinian}(dev')=x'$
and $\Hol(x')=\rho'$. Hence $\Hol$ is open.   }
\end{proof}

{\edit \begin{proof}[Proof of  (\ref{deformmfd})] By definition $\PC(M)=\image(\Hol)$, so the result follows from (\ref{mainthm}).\end{proof}}

We may avoid appealing to the theorem of Schoen and Yau used in the proof of (\ref{PCcontrol})
for the manifolds appearing in the main theorem using:
\begin{lemma}\label{controlfn} Every homogeneous cusp has an exhaustion function.
\end{lemma}

\begin{proof} 
Suppose  ${\edit C}=\Omega/\Gamma$ is a homogeneous cusp 
and $T=T(\Gamma)$ is the translation subgroup.
Then $\Omega$ has a codimension--1 foliation by $T$--orbits that covers a smooth foliation of ${\edit C}$. 
Pick $y$ in the interior of $\Omega$ and define
$F:\Omega\to\R$ by $F(x)=d_{\Omega}(x,T\cdot y)$ if $T\cdot y$ separates $\partial\Omega$ from $x$, and 
define $F(x)=0$ otherwise.
 Then $F$ covers a map
$f:{\edit C}\to \R$ 
and $Y=f^{-1}(0)$ is a compact collar neighborhood of $\partial {\edit C}$ and
$f(x)=d(y,Y)$ and $f^{-1}(t)$ is a leaf of the foliation of ${\edit C}$ for $t>0$.

 It is clear
that $\| df\|\le 1$ when $f>0$. Thus it suffices to show that
$\|D^2f\|$ is bounded.
Suppose there is a sequence $({\edit C}_k,f_k,x_k)$ such that $\|D^2f_k\|_{x_k}>k$. Then ${\edit C}_k=\Omega_k/\Gamma_k$ and
$\Gamma_k$ is a lattice in $G_k=\PGL(\Omega_k)$. We may assume all the $\Omega_k$ are in Benzecri position and $0$
covers $x_k$. We may also assume $\Omega_k\to\Omega$ in the Hausdorff topology. Then $G_k\to G\subset\PGL(\Omega)$.
The $T$ orbits are a smooth foliation of $\Omega$ and we define a smooth function
$F:\Omega\longrightarrow\R$ using $y=0$ as above. 
Then $F_k$ converges to $F$ in $C^{\infty}$ on compact sets.
But $\|D^2f_k\|=\|D^2F_k\|\to\infty$ contradicts $\| D^2F\|<\infty$ because $F$ is smooth.\end{proof}

\section{Three dimensional generalized cusps}\label{3mfd}

 An orientable three-dimensional generalized
  cusp is diffeomorphic to $T^2\times[0,\infty)$, and Leitner \cite{AL} shows {\edit
  that the holonomy
   is conjugate into a unique group of the form below}
with $\beta\ge \alpha>0$  fixed, where $n$ is the number of non-trivial weights:
$$
C_0=\bpmat 1 & s & t  &  \frac{s^2+t^2}{2} \\
0 &1  & 0 &s\\
0 & 0&1 &t \\
0 & 0 & 0 & 1\epmat  
\qquad
 C_1=\bpmat e^s & 0 & 0  &  0 \\
0 & 1  & t  & \frac{1}{2}t^2-s\\
0 & 0&1 &t\\
0 & 0 & 0 & 1\epmat 
$$

$$
C_2(\alpha)=\bpmat e^s & 0 & 0  &  0 \\
0 & e^t  & 0 &0\\
0 & 0& 1 & -t-\alpha s\\
0 & 0 & 0 & 1\epmat \qquad
C_3(\alpha,\beta) =\bpmat e^s & 0 & 0  &  0 \\
0 &e^t  & 0 &0\\
0 & 0&e^{-\alpha s-\beta t} &0\\
0 & 0 & 0 & 1\epmat
$$

   \begin{figure}[ht]	
   \begin{center} 
		 \includegraphics[scale=0.9]{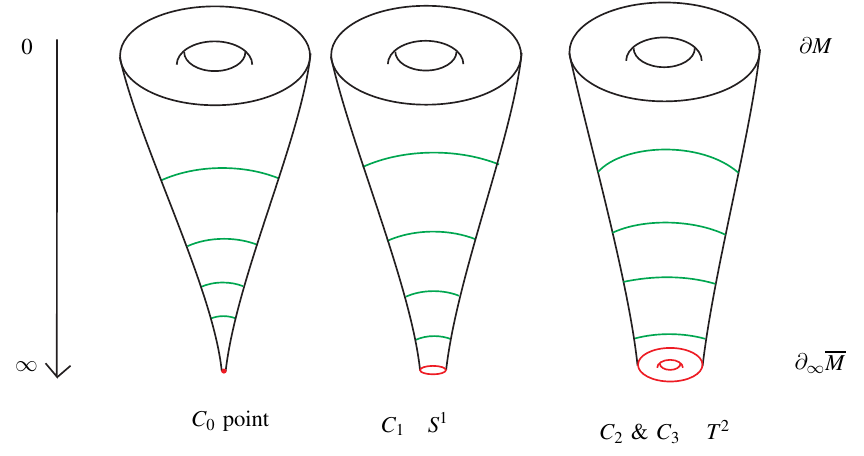}
\end{center}
  \caption{Generalized cusps in dimension $3$}
\end{figure}

 A related statement, due to Benoist, is in (2.7) of  \cite{BenoistIV}.
 There is a {\em compact}, properly convex domain $\Omega_n=\Omega_n(\alpha,\beta)$ preserved by $C_n=C_n(\alpha,\beta)$ and $\partial\Omega_n=A\sqcup B$ where $A=C_n\cdot x$ is an orbit and $B$ is a simplex
contained in a projective
hyperplane. If $\Gamma_n\subset C_n$ is a lattice then $\overline{M}=\Omega_n/\Gamma_n$ is a compactification
of the generalized cusp $M=(\Omega_n\setminus B)/\Gamma_n$ obtained by adding
$\partial_{\infty}\overline M=B/\Gamma$ which is
a point for $C_0$, a circle for $C_1$, and a torus for $C_2$ or $C_3$. The group $C_n$ is a  {\em translation
group} and is contained  in  the cusp Lie group  $\PGL(\Omega_n)$.

The group $C_0$ is conjugate into $PO(3,1)$ and
 contains the holonomy of a cusp of a hyperbolic 3-manifold. 
Ballas \cite{SB} found a lattice in  $C_1$  that is the holonomy of a generalized cusp  for a properly convex
structure on the figure eight knot complement. 
The groups $C_3(\alpha,\beta)$ are diagonal affine groups that satisfy
the {\em uniform middle eigenvalue condition} of Choi \cite{choiends}. Gye-Seon Lee found lattices
in some of these groups that are holonomies of generalized cusps for a 
properly convex structure on the figure eight knot complement.
 At the
time of writing it is not known if there is
 3-manifold that admits a finite volume hyperbolic structure 
 and also a properly convex structure that is a lattice in some $C_2(\alpha)$. 
 The classification of generalized cusps in all dimensions is given in Ballas, Cooper and Leitner \cite{BCL}.

\section{Convex Smoothing}\label{smoothing}
 We are concerned with various kinds of convexity. A  function $f:(a,b)\to\RR$ is {\em convex} if  $$\forall (c,d)\subset (a,b)\ \forall t\in(0,1)\quad f(tc + (1-t)d)\le tf(c)+(1-t)f(d)$$ and  {\em strictly{\edit-}convex}
if the above inequality is always {\em strict}.
It is  {\em Hessian-convex} 
if $f$ is smooth and $f''>0$ everywhere. A convex function is {\em strictly{\edit-}convex} at $p\in(a,b)$ if the graph of $y=f(x)$ intersects the tangent line
at $x=p$ at the single point $(p,f(p))$. 
 
 Each of these definitions extends to a  function $f:M\to\RR$ on an affine manifold $M$  by requiring
 its restriction to each line segment in $M$ to have the corresponding property. In the case of Hessian convex we also require $f$ is smooth. It follows
 that if a convex
 function  $f:M\to\RR$ is strictly{\edit-}convex at the point $p\in M$, then there is 
  hyperplane $H$  and a neighborhood $U\subset M$ of $p$ such the graph of $f|U$ intersects $H$ only at $p$.
For affine manifolds, we show how to approximate a convex function {\em which is strictly{\edit-}convex somewhere}
 by a smooth, Hessian-convex one. 
 
 The main application is that given 
a projective manifold which has a convex boundary that is strictly{\edit-}convex at some point,
we can shrink the manifold slightly to produce a submanifold with
 Hessian-convex boundary: locally the graph of a Hessian-convex function. 
 One might imagine using sandpaper to smooth the boundary and produce a submanifold with smooth
strictly{\edit-}convex boundary.
  
The idea is to improve a convex function which is already Hessian-convex on some open subset,
by changing it  in a small convex set $C$, and leaving it unchanged outside $C$. This is done
so that it is Hessian-convex inside a slightly smaller convex set $C^-\subset C$, and also Hessian-convex
at any point where it was previously Hessian-convex. In this way the problem is reduced
to a local one in Euclidean space.

  Greene and Wu \cite[Theorem 2]{Gw2},
and also 
  \cite{GW1}, showed that on a Riemannian manifold, 
 any  function $f$  with the property  that locally there is a function
 $g$ with positive definite Hessian such that
  $f-g$ is convex along geodesics ({\em\bf they} call $f$ {\em strictly{\edit-}convex})
  can be uniformly approximated by smooth, Hessian-convex functions.
  Smith \cite{Smith} gives an example, for each $k\ge 0$, 
 of a $C^k$ convex function on a non-compact 
 Euclidean surface which is not approximated by a $C^{k+1}$ convex function.

A function $f$ is   {\em convex down} if $-f$ is convex. This
 means secant lines lie {\em below} the graph: $tf(a)+(1-t)f(b)\le f(ta+(1-t)b) $ for all $a,b$ and $0\le t\le 1$. 
Equivalently the set of points {\em below} the graph of $f$ is convex.

If $f,g$ are smooth convex down functions, then $\min(f,g)$ is convex down, but need not be smooth at points
where $f=g$.
We construct a smooth approximation $m^{\kappa}$ on ${\mathbb R}_+^2$ which 
 agrees with $\min$ outside a certain neighborhood of the diagonal and has good convexity properties.
 
 \begin{lemma}[smoothing min]\label{lem:smoothing min}
 Given $\kappa\in(0,1)$ there is a smooth function  
 $m^{\kappa}:{\mathbb R}_+^2\longrightarrow{\mathbb R}_+,$ which is  convex-down and  non-decreasing in
 each variable  ($m^{\kappa}_x\ge0, m^{\kappa}_y\ge 0$), 
 such that if $x\le\kappa y$ or $y\le\kappa x,$ then
 $m^{\kappa}(x,y)=\min(x,y)$.  Moreover $m^{\kappa}$ is linear along rays: 
 $m^{\kappa}(tx,ty)=t\cdot m^{\kappa}(x,y)$ for $t\ge 0$. It follows
  that if $f,g:C\longrightarrow {\mathbb R}_+$ are convex down, then so is 
 $h(x)=m^{\kappa}(f(x),g(x))$.
\end{lemma}
\begin{proof}  
 On ${\mathbb R}_+^2$  $$\min(x,y)=(x+y)\cdot k(x/(x+y)),\qquad\text{where } k(t)=\min(t,1-t).$$
Choose $\delta$ so that $\kappa=\delta/(1-\delta).$ Then $\delta\in(0,1/2)$. Let $K:[0,1]\longrightarrow[0,1]$ 
be a convex-down smooth function that agrees with $k$ 
outside $(\delta,1-\delta)$.
Define $m:{\mathbb R}_+^2\longrightarrow{\mathbb R}$ by
 $$m(x,y)=(x+y)\cdot K(x/(x+y)).$$
   Clearly $m$ is linear along rays. If $x/(x+y)\le\delta,$ then $m(x,y)=x$. This happens when $x\le\kappa y$.
 Similarly $m(x,y)=y$ when $y\le\kappa x$, thus
 $$m(x,y)=\min(x,y)\quad \text{if}\quad x\le\kappa y\ \text{or\ } y\le\kappa x.$$
 The subset of ${\mathbb R}_+^2$ where neither $x\le \kappa y$ nor
 $y\le\kappa x$ is called the {\em transition region}. Outside the transition region
 $m=\min$.

The graph of $m$ is a convex-down surface above ${\mathbb R}^2_+$ 
 that is a union of rays starting at the origin. 
  One can picture the graph of $m$: it is the cone from the origin of the convex-down arch
that is the part of the graph lying above $x+y=1$. This arch is given by $K(x)$.
Since $K(x)$ is convex down, the graph of $m$ is convex-down; though in the radial direction it is, of course, linear.

 This surface is comprised of three 
 parts. The central part is curved down. The other two parts are  sectors of flat planes, one
 containing the $x$--axis and the other containing the $y$--axis.

We claim $m(x,y)$ is a non-decreasing function of each variable.  This is clear on the two parts of the graph
of $m$ that are flat, since they are planes containing either the $x$-axis or the $y$-axis.
Now 
$$ m_x(a,b)=\frac{\partial m}{\partial x}=K(a/(a+b)) + (a+b)\cdot b(a+b)^{-2}K'(a/(a+b)).$$
Since $m_x(ta,tb)=m_x(a,b)$ we may assume $a+b=1.$  Then
$$ m_x(a,b)=K(a) + (1-a)K'(a)=:c.$$
 The point $(x,y)=(1,c)$ lies on the tangent line to the
graph of $y=K(x)$ at $x=a$. Since this graph is convex down, and underneath the graph of $y=k(x)$, it follows that $c\ge 0$. This is best
seen by staring at a diagram.
Similar calculations work for $m_y$. This proves the claim.

Next we deduce that $h$ is convex-down using these two properties of $m$, where
$$h(ta+(1-t)b)=m(f(ta+(1-t)b),g(ta+(1-t)b)).$$
Since  $m_x\ge 0$ and $f$ is convex-down
$$m(f(ta+(1-t)b),g(ta+(1-t)b))\ge m(tf(a)+(1-t)f(b),g(ta+(1-t)b))$$
Similarly  $m_y\ge 0$ and $g$ is convex-down
$$ \begin{array}{rcl}
m(tf(a)+(1-t)f(b),g(ta+(1-t)b))& \ge & m(tf(a)+(1-t)f(b),tg(a)+(1-t)g(b))\\
 & = & m(t(f(a),g(a))+(1-t)(f(b),g(b))).\end{array}$$
Finally since $m$ is convex down
$$\begin{array}{rcl}
m(t(f(a),g(a))+(1-t)(f(b),g(b))) & \ge & t\cdot m(f(a),g(a))+(1-t)m(f(b),g(b))\\
  & = &t\cdot h(a)+(1-t)h(b).{\hspace{4.8cm}\qedhere}\end{array}$$
\end{proof}

\begin{corollary}[relative convex smoothing]\label{relativesmooth} Suppose $C\subset{\mathbb R}^n$ is a compact convex set with 
nonempty interior and $C^-$ is a compact convex set in the interior of $C$. Suppose
$f:C\longrightarrow{\mathbb R}$ is a non-constant, 
convex function, which is Hessian-convex on a (possibly empty)
subset $S\subset C$. Assume $f|\partial C=0$. Then there is a convex function $F:C\longrightarrow{\mathbb R}$
such that $F$ is Hessian-convex on $S\cup C^{-}$ and $f=F$ on some neighborhood of $\partial C$.
\end{corollary}
\begin{proof} Observe $f<0$ on the interior of $C$. Let $g$ be a Hessian-convex function on ${\mathbb R}^n$ which
is negative everywhere on $C$ and $g\ge f/2$ 
everywhere on $C^-$. 
Since $f$ is not identically zero this can be done with, for example, 
$g(x)=\alpha\|x\|^2+\beta$  with suitable constants.

 For $\kappa\in(0,1/2)$ define $F(x)=-m^{\kappa}(-f(x),-g(x))$ 
 {\edit then} $F(x)=\max(f(x),g(x))$ except when $f(x)$
  is close enough to $g(x)$, depending on $\kappa$. Since $g<f=0$ on $\partial C$
  it follows that $F=f$ on some neighborhood of $\partial C$. Moreover $F=g$ on $C^-$
  and therefore $F$ is Hessian-convex on $C^-$.
  
  By (\ref{lem:smoothing min}) $F$ is convex. Since $m=m^{\kappa}$ and $g$ are smooth and the composition
  of smooth functions is smooth, it follows $F$ is smooth
 on $S$. It only remains to show $D^2F$ is positive definite on $ S$.
 It suffices to show for every $a\in  S$ and every unit vector $u\in{\mathbb R}^n$ the function $p(t)=-F(a+t\cdot u)$ satisfies $p''(0)<0.$
Computing 
 $$p'=-m_xf_u-m_yg_u$$ where$$
 m_x=\frac{\partial m}{\partial x},\ 
  m_y=\frac{\partial m}{\partial y}$$ and $f_u,g_u$ are the derivatives in direction $u$ at $a\in C,$
  $$ f_u=df(u),\quad  g_u=dg(u).$$
 Then
 $$p''=[m_{xx}(f_u)^2 + 2m_{xy}f_ug_u + m_{yy}(g_u)^2]\ -\ [ m_xf_{uu}+m_yg_{uu}].$$
Since $m$ is smooth and convex down it follows that $D^2m$ is negative semi-definite, so the first term is 
$\le 0$.
 By \ref{lem:smoothing min} we have $m_{x}\ge 0$ and $m_{y}\ge 0$. Also $g_{uu}>0$ everywhere and $f_{uu}>0$ on $S$.
 Since $m$ is linear along rays, and $m(x,y)>0$ on the positive quadrant ($x>0$ and $y>0$), it follows that $m_x+m_y>0$ on the positive quadrant, so the $m_xf_{uu}+m_yg_{uu}>0$ everywhere on $S$. Hence $p''<0$ everywhere on $S$ as required.
 \end{proof}

A component $N$ of the boundary of a projective manifold $M$ is {\em Hessian-convex} if
$N$ is locally the graph over the tangent hyperplane of a smooth function
with positive definite Hessian in some chart.
\begin{proposition}[smoothing convex boundary]\label{smoothbdry}
Suppose $M$ is a projective manifold and $\partial M$ is everywhere locally convex, and also strictly{\edit-}convex at 
one point on each component of $\partial M$.
Then there is a submanifold $N\subset M$ such that  $M\setminus \interior(N)\cong [0,1]\times\partial N$
and $\partial N$ is Hessian-convex. 
\end{proposition}
\begin{proof} Suppose $\partial M$ is strictly{\edit-}convex at $x\in\partial M$. Choose a (subset of a) hyperplane $H\subset M$
close to $x$ so that the component, $C$, of $M\setminus H$ containing $x$ is a small convex set
$V$. Using local affine coordinates,
 $S=C\cap\partial M$ is the graph over $H$ of a convex function $f,$ which
is $0$ on $H\cap\partial M$. Apply (\ref{relativesmooth}) to produce a smooth function $g$ with
positive definite Hessian and satisfying $0\le g\le f$. The graph of $g$ is a smooth hypersurface between
$H$ and $S$. Replace $S$ by this graph. This smoothes out part of $\partial M$.
Repeating this procedure smoothes the entire boundary.
 \end{proof}
   In a similar way one can prove:
\begin{corollary}[smoothing convex functions]
 Suppose $M$ is a connected affine manifold and $f:M\longrightarrow \R$ is  a convex function, which
 is strictly{\edit-}convex at some point. Given $\epsilon>0$ there is  $g:M\longrightarrow \R,$ which is smooth, Hessian-convex and satisfies $|f-g|<\epsilon$.
\end{corollary} 

\section{Benz{\'e}cri's Theorem}
\label{sec:Benzecri}

\begin{theorem}[Benz{\'e}cri \cite{Benz}]\label{benzthm} For each $n>1$ there is
a {\em Benz{\'e}cri constant} $R=R_{\mathcal B}(n)\le 5^{n-1}$ with the following property.
Suppose $\Omega$ is a properly convex open subset of $\RPn$ and $p\in\Omega $. Then
there is a projective transformation $\tau\in \PGL(n+1,{\mathbb R})$ such that $\tau(p)=0$ and
$B(1)\subset \tau(\Omega)\subset B(R),$ where $B(t)$ is the closed ball of radius $t$ in ${\mathbb R}^n$ centered at $0$.
\end{theorem}

The projective transformation $\tau$ is called a \emph{Benz{\'e}cri chart for $\Omega$ centered at $p$}
and the image $\tau(\Omega,p)$ is called {\em Benz{\'e}cri position}.
The following proof  is shorter and more elementary than the traditional proof using John ellipsoids, and also
  provides an algorithm to find a Benz{\'e}cri chart.
The set of Benz{\'e}cri charts for $(\Omega,p)$ is a compact subset of $\PGL(n+1,{\mathbb R})$.
\begin{proof} The proof is by induction on $n$. If $n=1,$ then $\Omega$ is an open interval in ${\mathbb R}P^1$ 
with closure a closed interval. There is a projective transformation taking $\Omega$ to $(-1,1)$ and $p$ to $0$ so
$R_{\mathcal B}(1)=1$. 

For the inductive step, choose a projective hyperplane $H^{n-1}\subset \RPn$ containing $p$. 
Then $\Omega'=\Omega\cap H$ is an open convex set in $H\cong{\mathbb R}P^{n-1}$ and $p\in\Omega'$.
Since $\Omega$ is properly convex, $\overline{\Omega}$ is disjoint from some projective hyperplane $K^{n-1}.$ Thus
$\overline{\Omega}'=\overline{\Omega}\cap H$ is disjoint from $H\cap K,$ which is a hyperplane in $H$. It
follows that $\Omega'$ is properly convex in $H$. By induction, and after choosing
appropriate coordinates on an affine patch in $H$ 
(or using a fixed coordinate system and applying a Benz{\'e}cri transformation to $\Omega'$), we may assume
 that $\Omega'\subset{\mathbb R}^{n-1}\times 0\subset H$
  with $p=0$ and $B^{n-1}(1)\subset\Omega'\subset B^{n-1}(r),$ where $r=R_{\mathcal B}(n-1)$.

 \begin{figure}[ht]
 \begin{center}
 	 \includegraphics[scale=1]{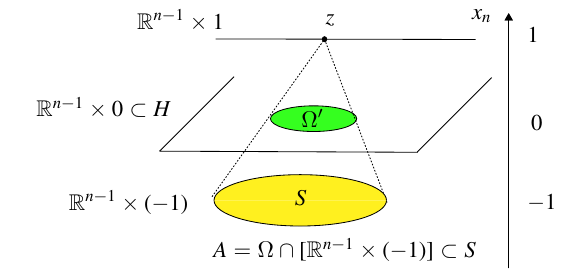}
	 \end{center}
 \caption{Shadows} \label{benzfig1}
\end{figure}

There are affine coordinates on $\RPn\setminus K={\mathbb R}^n$ so that the affine part of
 $H$ is ${\mathbb R}^{n-1}\times 0$. In what follows we will apply projective transformations in $\PGL(n+1,{\mathbb R})$
 which are the identity on $H$. This moves $\Omega$ while keeping $\Omega'$ fixed. The first step is to arrange
 that $$\Omega\subset{\mathbb R}^{n-1}\times[-1,1]$$ and $\partial\Omega$ contains a point $z\in{\mathbb R}^{n-1}\times 1$. Then we
 may shear so that $z=(0,\cdots,0,1)$.

 Next consider the one-parameter group $A(t)\in \PGL(n+1,{\mathbb R})$ fixing $z$ and $H$. As $t$ varies, the points
 that are not fixed move between $z$ and $H$.  This group preserves the family of affine  planes $\{x_n=\text{const}\}$ in ${\mathbb R}^n$. 
Since it fixes $z$ the affine
 plane ${\mathbb R}^{n-1}\times 1$ is preserved (though not fixed) by this group.
 Thus we may move $\Omega$ by an element of this group so that 
  it is still contained in ${\mathbb R}^{n-1}\times[-1,1]$, 
  still contains $z$, and  $$A=\Omega\cap\big({\mathbb R}^{n-1}\times(-1)\big)$$  is
   not empty.  Let $C\subset {\mathbb R}^{n-1}\times[-1,1]$ be the set of points
    on all lines $\ell$ passing through $z$ and
 some point in $\Omega'$. Then $S=C\cap [{\mathbb R}^{n-1}\times (-1)]$
  is the shadow from 
 the point $z$ of $\Omega'$ on 
 ${\mathbb R}^{n-1}\times (-1)$. Since $\Omega$ is convex it follows that $A\subset S$. 
 Since $\Omega'\subset B_r(0)$ it follows that $S$ is contained in the shadow of $B_r(0),$ which
 is  the ball $D\subset{\mathbb R}^{n-1}\times(-1)$ 
 of radius $2r$ center $(0,\cdots,0,-1)$. Finally, let $X$ be the union of all line segments
  in ${\mathbb R}^{n-1}\times[-1,1]$ 
 containing a point of $S$ and $B_r$. This is contained in the union of the shadows on ${\mathbb R}^{n-1}\times 1$ of
 $B_r(0)\subset H$ from points in $D$. This is a ball in ${\mathbb R}^{n-1}\times 1$ of radius $4r$ center $z.$
It lies within distance $1+r\le 5r$ from $0$. \end{proof}

Let ${\mathcal S}$ be the set of all $\Omega\subset{\mathbb R}P^n,$ which
are disjoint from some hyperplane and compact,  convex, and with non-empty interior, equipped with the Hausdorff topology.
Let ${\mathcal S}_*\subset{\mathcal S}\times{\mathbb R}P^n$ be the set of all pairs $(\Omega,p)$ with $p$ in the interior of $\Omega$
with the subspace topology of the product topology. There is an action of $\tau\in\PGL(n+1,{\mathbb R})$ on ${\mathcal S}_*$ given by
$\tau(\Omega,p)=(\tau\Omega,\tau p)$. The quotient of ${\mathcal S}_*$ by this action is given the quotient
topology and denoted ${\mathcal B}$. 

\begin{corollary}[Benz{\'e}cri's compactness theorem \cite{Benz}]\label{benzcpct}  ${\mathcal B}$ is compact.
\end{corollary}

It follows that there is a compact set of preferred charts centered on a point in a properly convex manifold $M$.
Different preferred charts give Euclidean coordinates around $p$ which vary in a compact family
independent of $M$, depending only on dimension.

\small
\bibliography{koszul.bib} 

\begin{thebibliography}{10}

\bibitem{BC1}
M.~Baker and D.~Cooper.
\newblock A combination theorem for convex hyperbolic manifolds, with
  applications to surfaces in 3-manifolds.
\newblock {\em J. Topol.}, 1(3):603--642, 2008.

\bibitem{SB}
S.~{Ballas}.
\newblock {Finite Volume Properly Convex Deformations of the Figure Eight
  Knot}.
\newblock {\em Geom. Dedicata}, 178:49--73, 2015.

\bibitem{BCL}
S.~Ballas, D.~Cooper, and A.~Leitner.
\newblock A classification of generalized cusps in projective manifolds.
\newblock {\em In Preparation}, 2015.

\bibitem{BenoistIV}
Y.~Benoist.
\newblock Convexes divisibles. {IV}. {S}tructure du bord en dimension 3.
\newblock {\em Invent. Math.}, 164(2):249--278, 2006.

\bibitem{Benz}
J.-P. Benz{\'e}cri.
\newblock Sur les vari\'et\'es localement affines et localement projectives.
\newblock {\em Bull. Soc. Math. France}, 88:229--332, 1960.

\bibitem{CEG}
R.~D. Canary, D.~B.~A. Epstein, and P.~Green.
\newblock Notes on notes of {T}hurston.
\newblock In {\em Analytical and geometric aspects of hyperbolic space
  ({C}oventry/{D}urham, 1984)}, volume 111 of {\em London Math. Soc. Lecture
  Note Ser.}, pages 3--92. Cambridge Univ. Press, Cambridge, 1987.

\bibitem{CHOIGS}
S.~Choi.
\newblock Geometric structures on orbifolds and holonomy representations.
\newblock {\em Geom. Dedicata}, 104:161--199, 2004.

\bibitem{Choi1}
S.~{Choi}.
\newblock {The convex real projective manifolds and orbifolds with radial ends:
  the openness of deformations}.
\newblock {\em ArXiv e-prints}, Nov. 2010.

\bibitem{choiends}
S.~{Choi}.
\newblock {The classification of ends of properly convex real projective
  orbifolds II: Properly convex radial ends and totally geodesic ends}.
\newblock {\em ArXiv e-prints}, Jan. 2015.

\bibitem{CHOW}
B.~Chow, S.-C. Chu, D.~Glickenstein, C.~Guenther, J.~Isenberg, T.~Ivey,
  D.~Knopf, P.~Lu, F.~Luo, and L.~Ni.
\newblock {\em The {R}icci flow: techniques and applications. {P}art {III}.
  {G}eometric-analytic aspects}, volume 163 of {\em Mathematical Surveys and
  Monographs}.
\newblock American Mathematical Society, Providence, RI, 2010.

\bibitem{CDW}
D.~{Cooper}, J.~{Danciger}, and A.~{Wienhard}.
\newblock {Limits of geometries}.
\newblock {\em ArXiv e-prints}, Aug. 2014.

\bibitem{CL}
D.~Cooper and D.~D. Long.
\newblock A generalization of the {E}pstein-{P}enner construction to projective
  manifolds.
\newblock {\em Proc. Amer. Math. Soc.}, 143(10):4561--4569, 2015.

\bibitem{CLT1}
D.~Cooper, D.~D. Long, and S.~Tillmann.
\newblock On convex projective manifolds and cusps.
\newblock {\em Adv. Math.}, 277:181--251, 2015.

\bibitem{CM1}
M.~Crampon and L.~Marquis.
\newblock Un lemme de {K}azhdan-{M}argulis-{Z}assenhaus pour les g\'eom\'etries
  de {H}ilbert.
\newblock {\em Ann. Math. Blaise Pascal}, 20(2):363--376, 2013.

\bibitem{CM2}
M.~Crampon and L.~Marquis.
\newblock Finitude g\'eom\'etrique en g\'eom\'etrie de {H}ilbert.
\newblock {\em Ann. Inst. Fourier (Grenoble)}, 64(6):2299--2377, 2014.

\bibitem{FK}
J.~Faraut and A.~Kor{\'a}nyi.
\newblock {\em Analysis on symmetric cones}.
\newblock Oxford Mathematical Monographs. The Clarendon Press, Oxford
  University Press, New York, 1994.
\newblock Oxford Science Publications.

\bibitem{Gold2}
W.~{Goldman}.
\newblock {\em Geometric Structures on manifolds}.
Book, in preparation. http://www.math.umd.edu/~wmg/gstom.pdf 

\bibitem{Gold1}
W.~{Goldman}.
\newblock {\em Projective geometry on manifolds}.
\newblock Versions 11 March 1988 and 2 June 2009 downloaded from
  http://www2.math.umd.edu/$\sim$wmg/, 1988--2009.

\bibitem{GOLD}
W.~M. Goldman.
\newblock Geometric structures on manifolds and varieties of representations.
\newblock In {\em Geometry of group representations ({B}oulder, {CO}, 1987)},
  volume~74 of {\em Contemp. Math.}, pages 169--198. Amer. Math. Soc.,
  Providence, RI, 1988.

\bibitem{GW1}
R.~E. Greene and H.~Wu.
\newblock Approximation theorems, {$C^{\infty }$} convex exhaustions and
  manifolds of positive curvature.
\newblock {\em Bull. Amer. Math. Soc.}, 81:101--104, 1975.

\bibitem{Gw2}
R.~E. Greene and H.~Wu.
\newblock {$C^{\infty }$} convex functions and manifolds of positive curvature.
\newblock {\em Acta Math.}, 137(3-4):209--245, 1976.

\bibitem{Hirsch}
M.~W. Hirsch.
\newblock {\em Differential topology}.
\newblock Springer-Verlag, New York-Heidelberg, 1976.
\newblock Graduate Texts in Mathematics, No. 33.

\bibitem{HM}
M.~W. Hirsch and B.~Mazur.
\newblock {\em Smoothings of piecewise linear manifolds}.
\newblock Princeton University Press, Princeton, N. J.; University of Tokyo
  Press, Tokyo, 1974.
\newblock Annals of Mathematics Studies, No. 80.

\bibitem{Koecher1957}
M.~Koecher.
\newblock Positivit\"atsbereiche im {$R^n$}.
\newblock {\em Amer. J. Math.}, 79:575--596, 1957.

\bibitem{Kos1}
J.-L. Koszul.
\newblock Vari\'et\'es localement plates et convexit\'e.
\newblock {\em Osaka J. Math.}, 2:285--290, 1965.

\bibitem{Kos2}
J.-L. Koszul.
\newblock D\'eformations de connexions localement plates.
\newblock {\em Ann. Inst. Fourier (Grenoble)}, 18(fasc. 1):103--114, 1968.

\bibitem{AL}
A.~{Leitner}.
\newblock {A Classification of subgroups of $SL(4,R)$ Isomorphic to $R^{3}$ and
  Generalized Cusps in Projective 3--Manifolds}.
\newblock {\em ArXiv e-prints}, July 2015.

\bibitem{LIMA}
E.~L. Lima.
\newblock On the local triviality of the restriction map for embeddings.
\newblock {\em Comment. Math. Helv.}, 38:163--164, 1964.

\bibitem{Marq1}
L.~Marquis.
\newblock Espace des modules marqu\'es des surfaces projectives convexes de
  volume fini.
\newblock {\em Geom. Topol.}, 14(4):2103--2149, 2010.

\bibitem{MXY}
L.~Marquis.
\newblock Surface projective convexe de volume fini.
\newblock {\em Ann. Inst. Fourier (Grenoble)}, 62(1):325--392, 2012.

\bibitem{Marquishandbook}
L.~Marquis.
\newblock Around groups in {H}ilbert geometry.
\newblock In {\em Handbook of {H}ilbert geometry}, volume~22 of {\em IRMA Lect.
  Math. Theor. Phys.}, pages 207--261. Eur. Math. Soc., Z\"urich, 2014.

\bibitem{Pap}
A.~Papadopoulos and M.~Troyanov.
\newblock Harmonic symmetrization of convex sets and of {F}insler structures,
  with applications to {H}ilbert geometry.
\newblock {\em Expo. Math.}, 27(2):109--124, 2009.

\bibitem{SY}
R.~Schoen and S.-T. Yau.
\newblock {\em Lectures on differential geometry}.
\newblock Conference Proceedings and Lecture Notes in Geometry and Topology, I.
  International Press, Cambridge, MA, 1994.
\newblock Lecture notes prepared by Wei Yue Ding, Kung Ching Chang [Gong Qing
  Zhang], Jia Qing Zhong and Yi Chao Xu, Translated from the Chinese by Ding
  and S. Y. Cheng, Preface translated from the Chinese by Kaising Tso.

\bibitem{Shima}
H.~Shima.
\newblock {\em The geometry of {H}essian structures}.
\newblock World Scientific Publishing Co. Pte. Ltd., Hackensack, NJ, 2007.

\bibitem{SHIMAYAQI}
H.~Shima and K.~Yagi.
\newblock Geometry of {H}essian manifolds.
\newblock {\em Differential Geom. Appl.}, 7(3):277--290, 1997.

\bibitem{Smith}
P.~A.~N. Smith.
\newblock Counterexamples to smoothing convex functions.
\newblock {\em Canad. Math. Bull.}, 29(3):308--313, 1986.

\bibitem{TAM}
L.-F. Tam.
\newblock Exhaustion functions on complete manifolds.
\newblock In {\em Recent advances in geometric analysis}, volume~11 of {\em
  Adv. Lect. Math. (ALM)}, pages 211--215. Int. Press, Somerville, MA, 2010.

\bibitem{WPT}
W.~P. Thurston.
\newblock {\em The Geometry and Topology of Three-Manifolds}.
\newblock Princeton Univ. Math. Dept. Available from
  http://msri.org/publications/books/gt3m/, 1979.

\bibitem{Vinberg1963homog}
{\`E}.~B. Vinberg.
\newblock The theory of homogeneous convex cones.
\newblock {\em Trudy Moskov. Mat. Ob\v s\v c.}, 12:303--358, 1963.

\bibitem{WH}
J.~H.~C. Whitehead.
\newblock On {$C^1$}-complexes.
\newblock {\em Ann. of Math. (2)}, 41:809--824, 1940.

\bibitem{Witte}
D.~Witte.
\newblock Superrigidity of lattices in solvable {L}ie groups.
\newblock {\em Invent. Math.}, 122(1):147--193, 1995.

\end{thebibliography}
\bibliographystyle{abbrv}

\end{document}